\documentclass{article}
\usepackage{amsthm}
\usepackage{thmtools}
\usepackage[colorlinks=true,linkcolor=blue]{hyperref}
\usepackage{enumitem}
\usepackage{mathrsfs}
\usepackage[all]{xy}
\usepackage{caption}
\usepackage{graphics}
\usepackage{graphicx}
\usepackage{tikz}
\usepackage{url}
\usepackage{amsmath}
\usepackage{amsxtra}
\usepackage{amssymb}
\usepackage{lineno}
\usepackage{turnstile}
\usepackage{verbatim}
\usepackage{enumitem}
\usepackage[retainorgcmds]{IEEEtrantools}
\usepackage{undertilde}
\usepackage[all]{xy}
\usetikzlibrary{matrix}
\usetikzlibrary{calc,decorations.pathmorphing,shapes}

\declaretheoremstyle[bodyfont=\sl]{slanted}

\swapnumbers
\declaretheorem[name=Definition,style=definition,qed=$\dashv$,
numberwithin=section]{dfn}
\declaretheorem[name=Definition,style=definition,numbered=no,qed=$\dashv$]{dfn*}
\declaretheorem[name=Definition,style=definition,numbered=no]{dfnnoqed*}

\declaretheorem[name=Theorem,style=slanted,sibling=dfn]{tm}
\declaretheorem[name=Theorem,style=slanted,numbered=no]{tm*}
\declaretheorem[name=Lemma,style=slanted,sibling=dfn]{lem}
\declaretheorem[name=Corollary,style=slanted,sibling=dfn]{cor}
\declaretheorem[name=Corollary,style=slanted,numbered=no]{cor*}
\declaretheorem[name=Remark,style=definition,sibling=dfn]{rem}

\swapnumbers
\declaretheoremstyle[headfont=\scshape]{claimstyle}
\declaretheorem[name=Claim,style=claimstyle]{clm}

\declaretheorem[name=Claim,style=claimstyle,numbered=no]{clm*}

\declaretheorem[name=Subclaim,style=claimstyle,numbered=no]{sclm*}

\declaretheorem[name=Subsubclaim,style=claimstyle,numbered=no]{ssclm*}

\declaretheoremstyle[headfont=\scshape]{casestyle}

\declaretheorem[name=Case,style=casestyle]{case}

\declaretheorem[name=Case,style=casestyle]{casetwo}
\declaretheorem[name=Case,style=casestyle]{casethree}

\newcommand{\PP}{\mathbb{P}}

\newcommand{\DC}{\mathrm{DC}}
\newcommand{\RR}{\mathbb{R}}
\newcommand{\cd}{\circledast}

\newcommand{\stk}{\mathrm{st}}

\newcommand{\wt}{\widetilde}

\newcommand{\lgcd}{\mathrm{lgcd}}

\newcommand{\compmode}{1}
\newcommand{\compopt}[2]{\ifthenelse{\equal{\compmode}{0}}{#1}{#2}}
\newcommand{\putamin}{\mathrm{putamin}}
\newcommand{\J}{\mathcal{J}}
\newcommand{\her}{\mathcal{H}}
\newcommand{\sub}{\subseteq}

\newcommand{\inter}{\cap}
\newcommand{\om}{\omega}
\newcommand{\pow}{\mathcal{P}}
\newcommand{\OR}{\mathrm{OR}}
\newcommand{\Hull}{\mathrm{Hull}}
\newcommand{\cut}{\backslash}

\newcommand{\Tt}{\mathcal{T}}

\newcommand{\Uu}{\mathcal{U}}
\newcommand{\Vv}{\mathcal{V}}

\newcommand{\rg}{\mathrm{rg}}
\newcommand{\dom}{\mathrm{dom}}
\newcommand{\ins}{\trianglelefteq}

\newcommand{\pins}{\triangleleft}
\newcommand{\npins}{\ntriangleleft}
\newcommand{\crit}{\mathrm{cr}}
\newcommand{\rest}{\!\upharpoonright\!}
\newcommand{\com}{\circ}

\newcommand{\lh}{\mathrm{lh}}
\newcommand{\Ult}{\mathrm{Ult}}

\newcommand{\HOD}{\mathrm{HOD}}
\newcommand{\HC}{\mathrm{HC}}

\newcommand{\ZF}{\mathsf{ZF}}
\newcommand{\es}{\mathbb{E}}
\newcommand{\eps}{\varepsilon}
\newcommand{\Ttvec}{{\vec{\Tt}}}

\newcommand{\pred}{\mathrm{pred}}
\newcommand{\dirlim}{\mathrm{dir lim}}
\newcommand{\id}{\mathrm{id}}
\newcommand{\sq}{\mathrm{sq}}
\newcommand{\conc}{\ \widehat{\ }\ }

\newcommand{\nutilde}{\wt{\nu}}

\DeclareMathOperator{\card}{card}
\DeclareMathOperator{\cof}{cof}

\newcommand{\rPi}{\mathrm{r}\Pi}
\newcommand{\rSigma}{\mathrm{r}\Sigma}
\newcommand{\bfrSigma}{\utilde{\rSigma}}

\newcommand{\Yyvec}{\vec{\Yy}}
\newcommand{\Yy}{\mathcal{Y}}

\newcommand{\Xx}{\mathcal{X}}

\newcommand{\tu}{\textup}

\renewcommand{\cut}{\backslash}

\newcommand{\dfnemph}{\textbf}
\newcommand{\pvec}{\vec{p}}

\newcommand{\exit}{\mathrm{ex}}
\newcommand{\eqdef}{=_{\mathrm{def}}}
\newcommand{\tembto}{\hookrightarrow}

\newcommand{\clint}{\mathrm{clint}}
\newcommand{\ddd}{\mathrm{ddd}}
\newcommand{\dds}{\mathrm{dds}}
\newcommand{\Pivec}{\vec{\Pi}}
\newcommand{\successor}{\mathrm{succ}}
\newcommand{\passive}{\mathrm{pv}}
\newcommand{\dropset}{\mathscr{D}}
\newcommand{\inflatearrow}{\rightsquigarrow}
\newcommand{\jensenurl}{\jensenurla\jensenurlb}
\newcommand{\jensenurla}{http://www-irm.mathematik.hu-berlin.de}
\newcommand{\jensenurlb}{/\textasciitilde raesch/org/jensen.html}
\newcommand{\steelurl}{http://math.berkeley.edu/\textasciitilde steel}

\renewcommand{\cut}{\backslash}

\newcommand{\dr}{\mathscr{D}}
\newcommand{\mininflatearrow}{\inflatearrow_{\minim}}
\newcommand{\minim}{\mathrm{min}}

\title{Full normalization for transfinite stacks}
\author{Farmer Schlutzenberg\footnote{Teilweise gef\"ordert durch die Deutsche Forschungsgemeinschaft (DFG) im Rahmen der Exzellenzstrategie des Bundes und der L\"ander EXC 2044--390685587, Mathematik M\"unster: Dynamik–Geometrie–Struktur}\\ farmer.schlutzenberg@gmail.com}

\begin{document}

\maketitle

\begin{abstract}
We describe the extension of normal iteration strategies with appropriate condensation properties to strategies for stacks of normal trees, with full normalization. Given a regular uncountable cardinal $\Omega$ and an $(m,\Omega+1)$-iteration strategy $\Sigma$ for a premouse $M$, such that $\Sigma$ and $M$ both have appropriate condensation properties,  we extend $\Sigma$ to a strategy $\Sigma^*$ for the optimal-$(m,\Omega,\Omega+1)^*$-iteration game such that for all $\lambda<\Omega$ and all stacks $\vec{\mathcal{T}}=\left<\mathcal{T}_\alpha\right>_{\alpha<\lambda}$ via $\Sigma^*$, consisting of normal trees $\mathcal{T}_\alpha$, each of length ${<\Omega}$, there is a corresponding normal tree $\mathcal{X}$ via $\Sigma$ with $M^{\vec{\mathcal{T}}}_\infty=M^{\mathcal{X}}_\infty$. Moreover,
if there are no drops in model or degree along the main branches of these trees then the overall iteration maps $i^{\vec{\Tt}}:M\to M^{\vec{\Tt}}_\infty$ and $i^\Xx:M\to M^\Xx_\infty$ agree. The construction is the result of a combination of work of John Steel and of the author.

We also establish some further useful properties of $\Sigma^*$, and use the methods to analyze the comparison of multiple iterates via a common such strategy.
\end{abstract}

\section{Introduction}

The theory of normalization of iteration trees
has been developed in the last few years by John Steel
\cite{str_comparison}, \cite{steel_local_HOD_comp},
the author \cite{iter_for_stacks}, Ronald Jensen \cite{jensen_book}
and Benjamin Siskind. This has built on results
of various others in this direction,
as discussed more in
 \cite{str_comparison} and  \cite{iter_for_stacks}. In \emph{embedding 
normalization}
or \emph{normal realization},
a stack $\Ttvec$ of normal trees is realized into 
a single normal tree $\Xx$,
with a key outcome that we obtain
a final realization map
\[ \pi:M^{\Ttvec}_\infty\to M^\Xx_\infty, \]
where it is certainly allowed that $M^{\Ttvec}_\infty\neq M^\Xx_\infty$.
In \emph{full normalization} (which we also call just \emph{normalization}
here), it is also demanded that $M^{\Ttvec}_\infty=M^\Xx_\infty$
and $\pi=\id$.

The main aim of this article is to 
describe the analogue of the main results of \cite{iter_for_stacks}
for (full) 
normalization.
That is, we start with an $(m,\Omega+1)$-iteration
strategy $\Sigma$ (hence, for normal trees) for an $m$-sound premouse $M$,
where $\Omega$ is an uncountable regular cardinal,
such that $\Sigma$ has certain condensation properties,
and extend it naturally to an $(m,\Omega,\Omega+1)^*$-strategy
(hence, for stacks of normal trees) with appropriate
 normalization properties.\footnote{Recall that the superscript ``$*$'' in ``$(m,\Omega,\Omega+1)^*$'' means that if the $\alpha$th round of the iteration game produces a normal tree $\Tt_\alpha$ of length $\Omega+1$,
 then the game finishes at that point; the game only proceeds to round $\alpha+1$ if $\Tt_\alpha$ has length $<\Omega$.}

The main results in this paper hold for fine structural
mice $M$, either of the pure $L[\es]$ or  strategy variety (such
as in \cite{str_comparison}),
with either MS- or $\lambda$-indexing,
together with iteration strategies $\Sigma$ for corresponding iteration 
rules (though see \cite{rule_conversion}),
assuming that $M$ and $\Sigma$ satisfy the appropriate
condensation properties.
So whenever we say \emph{premouse} without specifying a restriction on the kind,
we mean any of these.
Some results are more specific to pure $L[\es]$-mice
(when the result involves a traditional comparison argument),
though in those cases the methods of \cite{str_comparison},
and ongoing work of Steel and Siskind
should help with generalization. And some are specific to MS-indexing,
though these would presumably adapt to $\lambda$-indexing,
by adapting the results of
\cite{premouse_inheriting} and \cite{fsfni} in this manner.

We will focus on iteration strategies with two key condensation properties:
\emph{minimal inflation condensation}
(\emph{mic}, Definition \ref{dfn:minimal_inflation_condensation})
and \emph{minimal hull condensation} (\emph{mhc},
Definition \ref{dfn:mhc}).
Minimal inflation condensation is a direct adadptation of inflation condensation
(see \cite{iter_for_stacks})
from normal realization to normalization,
and minimal hull condensation is an analogous adaptation of  strong hull condensation. Note that Steel defines
\emph{very strong hull condensation} to incorporate
something like the conjunction of strong hull condensation
and minimal hull condensation. We don't incorporate
either inflation condensation or strong hull condensation
(the notions relevant to normal realization) in the notions considered here, as they are not relevant to what we do.

The main result to be shown is the following.
It is the result of a combination of work of  Steel and the author.
The terminology \emph{optimal} and \emph{$m$-standard} is explained
somewhat just after the statement of the theorem.

\begin{tm}\label{tm:full_norm_stacks_strategy}
Let $\Omega>\om$ be a regular cardinal.
Let $M$ be an $m$-standard
 premouse. Let $\Sigma$
 be an $(m,\Omega+1)$-strategy for $M$ with minimal inflation
 condensation.
 Then there is an optimal-$(m,\Omega,\Omega+1)^*$-strategy
 $\Sigma^*$ for $M$ such that:
 \begin{enumerate}
  \item $\Sigma\sub\Sigma^*$,
 \item\label{item:X_for_Tvec} for every stack 
$\vec{\Tt}=\left<\Tt_\alpha\right>_{\alpha<\lambda}$ via $\Sigma^*$ with a last 
model
and with $\lambda<\Omega$,
there is an $m$-maximal successor length tree $\Xx$
on $M$ such that:
\begin{enumerate}
\item  $\Xx\rest\Omega+1$ is via $\Sigma$,  
\item if $\lh(\Tt_\alpha)<\Omega$ for all $\alpha<\lambda$
then $\lh(\Xx)<\Omega$ (and hence $\Xx$ is via $\Sigma$ in this case),
\item $M^{\vec{\Tt}}_\infty=M^\Xx_\infty$
and $\deg^{\vec{\Tt}}_\infty=\deg^\Xx_\infty$,
\item $b^{\vec{\Tt}}$ drops in model (degree) iff $b^\Xx$ does,
\item if $b^{\vec{\Tt}}$ does not drop in model or degree
then $i^{\vec{\Tt}}_{0\infty}=i^\Xx_{0\infty}$,
\end{enumerate}
\item $\Sigma^*$ is $\Delta_1(\{\Sigma\})$,
uniformly in $\Sigma$,
\item if $\card(M)<\Omega$ then $\Sigma^*\rest\her_{\Omega}$
is $\Delta_1^{\her_{\Omega}}(\Sigma\rest\her_\Omega)$,
uniformly in $\Sigma$.
\end{enumerate}
\end{tm}
So if $\Omega=\omega_1$ then $\Sigma^*\rest\HC$
is $\Delta_1^{\HC}(\Sigma\rest\HC)$.

The prefix \emph{optimal}
prevents player I from making artificial drops
(see \cite[\S1.1.5]{iter_for_stacks}).
And \emph{$m$-standard} (for $m\leq\om$,
Definition \ref{dfn:m-standard})
just requires $m$-soundness plus some standard condensation facts
which for pure $L[\es]$-mice
follow from $(m,\omega_1,\omega_1+1)^*$-iterability
(in fact from $(m,\omega_1+1)$-iterability for MS-indexing).

In part
 \ref{item:X_for_Tvec}, note that  $\Xx\rest(\Omega+1)$ is 
uniquely determined
 by the conditions mentioned. If $\lh(\Tt_\alpha)=\Omega+1$
 then $\lambda=\alpha+1$ and player II wins,
 by the rules of the $(m,\Omega,\Omega+1)^*$-iteration game.
 We can't demand that $\Xx$ be fully ``via $\Sigma$''
 in this case, as it can be that $\lh(\Xx)>\Omega+\om$.

 \begin{rem}\label{rem:normalization_different_contexts}
 One should be aware of a key difference between
 the context of Theorem \ref{tm:full_norm_stacks_strategy}
 and Steel's context for normal realization in \cite{str_comparison}, which reflect different goals.
  In Theorem \ref{tm:full_norm_stacks_strategy},
  we begin with a \emph{normal} strategy $\Sigma$ with certain condensation properties,
  and \emph{extend} this to a stacks strategy $\Sigma^*$. Steel starts with a stacks strategy $\Gamma$ with certain properties,
  and deduces further properties of $\Gamma$.
  (In a particular case of interest,
  the $M$ is a (strategy) mouse built by background construction in a background universe $R$, and $\Gamma$ is the strategy induced by a stacks strategy $\Gamma'$ for $R$, where $\Gamma'$ is assumed to have certain properties. Moreover, in the eventually corrected version of Steel's setup, $M$ is not literally a premouse in the sense of this paper, but a pfs-premouse. But we do not need to consider pfs-premice here.)
  These differences aside, at least in the case of finite stacks, many of the calculations involved are the same, irrespective of which context in which one is working.
 \end{rem}

A key  consequence of Theorem \ref{tm:full_norm_stacks_strategy},
together with Theorem \ref{tm:wDJ_implies_cond} (which only applies to pure $L[\es]$-premice), is for example the following, which was observed by Steel (prior to the proof of Theorem \ref{tm:full_norm_stacks_strategy} being worked out in detail):
\begin{cor}
Assume that $M^\#=M_1^\#$ or $M^\#=M_\om^\#$ is fully iterable, and $M=M_1$ or $M=M_\om$
respectively. Let $\delta_0$ be the least Woodin
of $M$.
Let $M_\infty$ be the direct limit of non-dropping countable
 iterates of $M$ via trees based on $M|\delta_0$. 
 Then $M_\infty$ is a normal iterate of $M$
 via its unique strategy.
\end{cor}
\begin{proof}
 The strategy $\Sigma$ for $M$ has Dodd-Jensen.
 By Theorem \ref{tm:wDJ_implies_cond}, $\Sigma$
 therefore has minimal inflation condensation,
 and hence the theorem applies, as witnessed by $\Sigma^*$.
 Moreover, if $\Sigma'$ is any $(\om,\om_1,\om+1)^*$-strategy
 for $M$, then $\Sigma'$ agrees with $\Sigma^*$
 on non-dropping iterates, by a standard
 uniqueness argument. This yields the corollary.
\end{proof}

For the following corollary, the author does not know whether
$\DC$ is necessary, but see \cite[\S10]{iter_for_stacks}.

\begin{cor}
Assume $\DC$. Let $\Omega>\om$ be regular.
Let $M$ be a countable, $m$-sound, $(m,\Omega,\Omega+1)^*$-iterable pure $L[\es]$-premouse.
Then there is an optimal-$(m,\Omega,\Omega+1)^*$-strategy $\Sigma^*$ for $M$,
with first round $\Sigma$, such that $\Sigma,\Sigma^*$
are related as in Theorem \ref{tm:full_norm_stacks_strategy}.
\end{cor}
\begin{proof}
 We may take an $(m,\Omega,\Omega+1)^*$-strategy
 $\Gamma$ for $M$ with weak DJ (using $\DC$). By $(m,\Omega,\Omega+1)^*$-iterability
 (and that $M$ is pure $L[\es]$),
 $M$ is $m$-standard. 
 By Theorem \ref{tm:wDJ_implies_cond},
 the first round $\Sigma$ of
 $\Gamma$ has minimal inflation condensation, and so the theorem applies to $\Sigma$.
\end{proof}

The key observation beyond the methods of normal realization,
which leads from there to (full) normalization,
is due to Steel,
and  was described by him in preprints of
 \cite{str_comparison} in 2015.
 In normal realization, given an embedding
$\pi:M\to N$ 
between premice $M,N$, 
and $E$ in the extender sequence $\es_+^M$ of $M$, 
if one wants to copy $E$
using $\pi$, then one copies to $\pi(E)$
(or to $F^N$, if $E=F^M$).
But for normalization,
this need not be the appropriate copy.
Let $P\ins M$ be such that $E=F^P$.
Assuming $\pi$ satisfies some further properties with respect to $P$
which will be detailed later,
then 
the appropriate copy 
is the active extender $E'$ of $P'=\Ult_0(P,F)$,
where $F$ is a certain extender derived from $\pi$.
(In fact, we will be considering the case that $\pi$ arises
as an ``abstract iteration map'' resulting from a sequence $\left<F_\alpha\right>_{\alpha<\lambda}$
of extenders, and $F$ will be the concatenation of some initial segment of this sequence.)

For example, if $Q$ is a premouse and $E=F^P$ (where $P$ is as above) is a $Q$-extender and $\crit(E)<\rho_m^Q$ and
 $R=\Ult_m(Q,E)$,
 and $F$ is the active extender of a premouse, and is an $R$-extender with $\crit(E)<\crit(F)<\nu(E)$,
 then $\Ult_m(R,F)=\Ult_m(Q,E')$, where $E'=F^{P'}$ is as above.
 (See e.g.~\cite[3.13--3.20]{mim}
 or \cite[\S3]{premouse_inheriting}.)  
But to be able to make this copy, we of course
need that $E'\in\es_+^{N}$.
Steel showed that the latter is true,
assuming that $M$ satisfies some standard
condensation facts.

As partially described in an early draft of \cite{str_comparison}, Steel used this copying process
to introduce and outline a (full) normalization analogue of tree embeddings, and corresponding
adaptation of strong hull condensation,  \emph{very strong hull condensation}, 
and also a procedure for (full) normalization of \emph{finite} stacks adapting
that for normal realization (that is, within his overall context; see Remark \ref{rem:normalization_different_contexts}). This material was not discussed in full detail there, however.
Some time later, some details involved were ironed out independently by Steel and the author
(see Footnote \ref{ftn:dropdown_pres}, Remark \ref{rem:pres_dropdown}). Steel and the author then (in 2016, independently) proceeded to flesh out the finite stack version in detail (working in the respective contexts described in \ref{rem:normalization_different_contexts}), and the author worked out the main ideas for the infinite stack version, in the context of Theorem \ref{tm:full_norm_stacks_strategy}, which we will use in this paper. Steel made his handwritten notes
\cite{steel_local_HOD_comp} available at that time.
We  develop both the finite stack and infinite stack material in detail here (in the context of \ref{tm:full_norm_stacks_strategy}).

{\color{red}}
It turns out that most of the ideas needed for the proof of Theorem \ref{tm:full_norm_stacks_strategy}
are already present in the papers \cite{str_comparison}
and \cite{iter_for_stacks}; those methods 
combined
with a bit more analysis is enough.
But it does take some work to set things up,
so that the further analysis can be carried out.
Conveniently, however, various complications which arise
in normal realization 
are eliminated when dealing with (full) normalization.

The structure of the proof of Theorem \ref{tm:full_norm_stacks_strategy}
we give, and much of the 
detail, is very similar to that of 
\cite[Theorem 9.1]{iter_for_stacks},
and where possible, we will omit details of proofs which
are (essentially) the same. So the reader should
have \cite{iter_for_stacks} available.
If the reader is not familiar with that paper, it seems one
might just consult it as needed (and as mentioned,
certain complications in that paper do not arise here).
The notation  in \cite{iter_for_stacks}
is 
different from that employed by Steel in \cite{str_comparison},
although many notions match up in meaning.
Because this paper is very tightly related to \cite{iter_for_stacks},
and in order to make things easier on the reader, 
we opted to maintain consistency with the notation of \cite{iter_for_stacks}
as far as possible.
We will give 
most key definitions in full (but not all definitions),
even though some of these are  very similar to those in \cite{iter_for_stacks}.

We proceed as follows.  In
\S\ref{sec:fine_structure} we discuss fine structural background.
\S\ref{sec:min_strat_cond} covers the strategy condensation properties we consider.
Adapting notions from \cite{iter_for_stacks},
\S\ref{sec:factor_tree} discusses the minimal version of the factor tree,
and \S\ref{sec:min_comp}  the minimal version of tree comparison.
In \S\ref{sec:inf_comm} we discuss minimal inflation stacks,
which adapts \cite[Lemmas 6.1, 6.2]{iter_for_stacks} on commutativity of inflation,
with extra features developed here regarding infinite stacks.
In \S\ref{sec:normalization} we reach the actual proof of Theorem \ref{tm:full_norm_stacks_strategy}, together with a variant, Theorem \ref{tm:stacks_iterability_2}. In \S\ref{sec:analysis_of_comparison}
we given an analysis of comparison of normal iterates $N_0,N_1$
of a given mouse $M$
via a strategy $\Sigma$ with minimal inflation condensation,
via the strategies $\Sigma_{N_0},\Sigma_{N_1}$
for $N_0,N_1$ given by (the proof of)
Theorem \ref{tm:full_norm_stacks_strategy}. In \S\ref{sec:gen_abs_it} we sketch the adaptation of  \cite[Theorem 7.3]{iter_for_stacks} on generic absoluteness of iterability (with condensation properties for the strategy). Finally  in \S\ref{sec:properties} we establish
a few further useful properties of the stacks strategy $\Sigma^{\stk}$ derived from $\Sigma$
in the proof of Theorem \ref{tm:full_norm_stacks_strategy}.\footnote{The work on the main material in this paper
began in 2015 or early 2016, some time after Steel and the author had communicated
with one another on their respective work  in
 \cite{str_comparison}
and \cite{iter_for_stacks}, and Steel suggested considering full normalization for infinite stacks.
After this we both considered the problem, basically independently.
By the time of the 2016 UC Irvine conference in inner model theory,
the author had sorted out pretty much the proof of full normalization presented in this paper.
The analysis of comparison  in \S\ref{sec:analysis_of_comparison}
was observed by the author in early 2016,
in conjunction with discussions with Steel
regarding $\HOD^{L[x]}$. The material in \S\ref{sec:properties} was worked out a few years later.

The method for dealing with infinite stacks in this paper relies heavily on  \cite{iter_for_stacks},
and of course our goal here is to extend a given normal strategy to a strategy for stacks.
Steel has considered infinite stacks via a strategy induced for a (strategy) premouse constructed by background construction
in a universe for which there is a sufficiently nice coarse strategy. He has (tentative?) results on normalization for infinite stacks in that context, which may rely somewhat on dealing with the complications discussed in \cite{lambda-errors} and its successors.}

Steel and  Siskind have also been developing  a paper with work on full normalization,
combined with more of the theory of strategy mice, and extending Steel's \cite{str_comparison}.

Finally, I would like to thank John Steel, Benjamin Siskind, and Ronald Jensen, for various discussions on the topic over the last few years.

\subsection{Notation}\label{subsec:notation}

See \cite[\S1.1]{iter_for_stacks} and \cite[\S1.1]{premouse_inheriting} for most of the notation and terminology we use.
(However, we use $\lh(E)$ for the index of an extender $E\in\es_+^M$, for a premouse $M$,
which is denoted $\mathrm{ind}(E)$ in \cite{iter_for_stacks}.)
We just mention below a few of the more obscure terminological items that show up in the paper.

We deal with both MS-indexed and $\lambda$-indexed 
pure/strategy premice,
except that we allow extenders of superstrong type on the sequence (for both kinds of indexing). We use MS-fine structure
(that is, $\rSigma_n$, etc), as simplified in \cite{V=HODX}, and also use MS-fine structure for
$\lambda$-indexed premice. 
We use some definitions/facts from \cite{premouse_inheriting},
which are literally stated there for MS-indexed pure-$L[\es]$ premice,
but as long as there are direct translations to other forms, we assume such a translation.
See in particular \cite[\S2]{premouse_inheriting} for the notion \emph{$n$-lifting 
embedding}.

For an active premouse $N$, $\lgcd(N)$ denotes the largest cardinal of $N$,
and $\nutilde(F^N)=\nutilde(N)$ denotes 
$\max(\lgcd(N),\nu(F^N))$. For a passive premouse
$N$, $\nutilde(N)$ denotes $\OR^N$. We write $M\pins_{\card} N$ to say
that $M\pins N$ and $\OR^M$ is a cardinal of $N$.

Let $\Tt$ be an $m$-maximal tree. If $\alpha+1<\lh(\Tt)$,
then $\exit^\Tt_\alpha$ denotes $M^\Tt_\alpha|\lh(E^\Tt_\alpha)$ (\emph{ex} for \emph{exit})
and $\nutilde^\Tt_\alpha=\nutilde(E^\Tt_\alpha)=\nutilde(\exit^\Tt_\alpha)$,
so note $\nutilde^\Tt_\alpha$ is the exchange ordinal associated to $E^\Tt_\alpha$
in $\Tt$ (for either MS-iteration rules or $\lambda$-iteration rules).
If $\Tt$ has successor length $\alpha+1$,
we say $E\in\es_+(M^\Tt_\alpha)$
is \emph{$\Tt$-normal}
iff $\lh(E^\Tt_\beta)\leq\lh(E^\Tt_\alpha)$ for all $\beta<\alpha$.
We say $\Tt$ is \emph{terminally non-dropping}
iff it has successor length $\alpha+1$
and $[0,\alpha]_\Tt$ does not drop in model or degree.
A \emph{putative $m$-maximal tree} $\Tt$ is like an $m$-maximal
tree, except that if $\lh(\Tt)=\alpha+1$ then we do not demand that $[0,\alpha)_\Tt$ has only finitely many drops
(and hence $M^\Tt_\alpha$ might be ill-defined), and if $[0,\alpha)_\Tt$ does have only finitely many drops,
we do not demand that $M^\Tt_\alpha$ is wellfounded.

\section{Fine structural preliminaries}\label{sec:fine_structure}

\subsection{Degree $\om$ and degree $0$}\label{subsec:deg_om_deg_0}

Let $M$ be an $\om$-sound premouse.
 Then there is a natural 1-1 correspondence
 between $\om$-maximal iteration trees $\Tt$ on $M$
 and $0$-maximal iteration trees $\Uu$ on $\J(M)$,
such that for all corresponding pairs $(\Tt,\Uu)$,
we have $\lh(\Tt)=\lh(\Uu)$,  ${<^\Tt}={<^\Uu}$,
$E^\Tt_\alpha=E^\Uu_\alpha$ for each $\alpha+1<\lh(\Tt,\Uu)$,
and for each $\alpha<\lh(\Tt,\Uu)$, we have:
\begin{enumerate}[label=--]
\item $[0,\alpha]_\Tt\inter\dropset^\Tt_{\deg}=\emptyset\iff[0,\alpha]_\Uu\inter\dropset^\Uu=\emptyset$,
\item  if $[0,\alpha]_\Tt\inter\dropset^\Tt_{\deg}=\emptyset$
 then $M^\Uu_\alpha=\J(M^\Tt_\alpha)$ and $i^\Tt_{0\alpha}\sub i^\Uu_{0\alpha}$
 and $i^\Uu_{0\alpha}(M)=M^\Tt_\alpha$,
\item if $[0,\alpha]_\Tt\inter\dropset^\Tt_{\deg}\neq\emptyset$
 then $M^\Tt_\alpha=M^\Uu_\alpha$ and there is the natural agreement of iteration maps.
\end{enumerate}
 This is straightforward to see. Likewise for $\om$-maximal stacks on $M$ and $0$-maximal
 stacks on $\J(M)$. In fact, such a correspondence holds not only for such iteration trees,
 but also for abstract iterations via sequences $\vec{E}$ of extenders considered in what follows.
 
 So throughout the paper, for a little more uniformity, we will ignore iteration trees and strategies
 at degree $\om$ for $\om$-sound premice $M$, by instead considering the corresponding degree $0$ trees and strategies
 for $\J(M)$, assuming $M$ is a set;
 if $M$ is proper class, then of course degree $n$ for $M$ is equivalent for all $n\leq\om$,
 so in this case we just consider degree $0$ for $M$.
 
\subsection{Dropdown preservation}

Just as in \cite[Definition 2.10]{premouse_inheriting}, we abstract out some condensation we need to assume
holds of the base premouse $M$ we will be iterating:

\begin{dfn}\label{dfn:m-standard}
Let $m<\om$ and let $M$ be an $(m+1)$-sound premouse.
We say that $M$ is \emph{$(m+1)$-relevantly-condensing}
iff for all $P,\pi$, if
\begin{enumerate}
 \item $P$ is an $(m+1)$-sound premouse,
 \item $\rho_{m+1}^P$ is an $M$-cardinal,
 \item\label{item:pi_is_m-lifting_etc} $\pi:P\to M$ is a 
$\pvec_{m+1}$-preserving $m$-lifting embedding,
 \item $\crit(\pi)\geq\rho_{m+1}^P$ and
 \item\label{item:pi``rho_m_bounded} $\pi``\rho_m^P$ is bounded in $\rho_m^M$
\end{enumerate}
then $P\pins M$.

Say that $M$ is \emph{$(m+1)$-sub-condensing}
iff for all $\pi:P\to M$ as above,
except that we replace conditions \ref{item:pi_is_m-lifting_etc}
and \ref{item:pi``rho_m_bounded} respectively with
\begin{enumerate}[label=\arabic*'.]
\setcounter{enumi}{2}
 \item  $\pi:P\to M$ is a 
$\pvec_{m+1}$-preserving $m$-embedding,
\setcounter{enumi}{4}
\item\label{item:rho_m+1^P<rho_m+1^M}  $\rho_{m+1}^P<\rho_{m+1}^M$,
\end{enumerate}
then $P\pins M$.

For $n<\om$, a premouse $N$ is  \emph{$n$-standard} iff:
\begin{enumerate}[label=--]
\item $N$ is $n$-sound
and $(m+1)$-relevantly-condensing for every $m<n$, and
\item every $M\pins N$ is $(m+1)$-relevantly-condensing
and $(m+1)$-sub-condensing for each $m<\om$.\footnote{See 
\ref{rem:condensation_def}.}
\end{enumerate}
And $N$ is \emph{$\om$-standard} iff $n$-standard for each $n<\om$.\footnote{By \S\ref{subsec:deg_om_deg_0},
we won't use the notion of \emph{$\om$-standard} in the main calculations; it is only included
as it is used in the statement of some theorems.}
\end{dfn}

\begin{rem}\label{rem:condensation_def}
If $N$ is an 
$n$-sound, $(n,\om_1+1)$-iterable MS-indexed pure $L[\es]$-premouse,
 then $N$ is $n$-standard, by \cite{fsfni}.
The author expects  this should also work  for $\lambda$-indexing,
but has not attempted to work through the details. For this reason
and because we are also considering strategy premice,
we include \emph{$m$-standard} explicitly as one of the hypotheses of the main theorems.
 
Like in \cite[\S2]{premouse_inheriting}, 
\emph{$0$-standard} is an $\rPi_1$ property of premice,
 and \emph{$(m+1)$-standard} is an $\rPi_{m+1}(\pvec_{m+1})$ property over 
$(m+1)$-sound premice $N$;
 therefore, in general, $n$-standardness is preserved by degree $n$ 
 ultrapower maps.
 
 Let $\pi:P\to M$ be as in the definition of
 \emph{$(m+1)$-sub-condensing},
 except that we drop requirement
 that $\rho_{m+1}^P<\rho_{m+1}^M$.
 As also discussed in \cite{premouse_inheriting}, if also 
$\crit(\pi)\geq\rho_{m+1}^M$, then $P=M$ and $\pi=\id$;
 this just follows directly from fine structure.\end{rem}

\begin{dfn}\label{dfn:(M,m)-good} Let $M$ be an $m$-sound premouse.
Let $\vec{E}=\left<E_\alpha\right>_{\alpha<\lambda}$ be a sequence of short 
extenders.
We say that $\vec{E}$ is \emph{$(M,m)$-pre-good} iff
there is a sequence $\left<M_\alpha\right>_{\alpha\leq\lambda}$
such that:
\begin{enumerate}[label=--]
 \item $M_0=M$,
 \item for each $\alpha<\lambda$, $E_\alpha$ is
 a weakly amenable  $M_\alpha$-extender 
with
\[ \crit(E_\alpha)<\min(\rho_m^{M_\alpha},\nutilde(M_\alpha)),\]
 \item for each $\alpha<\lambda$, $M_{\alpha+1}=\Ult_m(M_\alpha,E_\alpha)$, 
 \item for each limit $\gamma\leq\lambda$, 
$M_\gamma=\dirlim_{\alpha<\gamma}M_\alpha$, under the (compositions and direct 
limits of) the ultrapower maps,
 \item for each $\alpha<\lambda$, $M_\alpha$ is wellfounded.
\end{enumerate}
We write $\Ult_m(M,\vec{E})=M_\lambda$ and $i^{M,m}_{\vec{E}}$ for the 
ultrapower map.
We say that $\vec{E}$ is \emph{$(M,m)$-good} iff $\vec{E}$ is $(M,m)$-pre-good
and $M_\lambda$ is wellfounded.
If $\vec{E}$ is $(M,m)$-pre-good, given $\kappa\leq\rho_m^M$, we say that 
$\vec{E}$ is \emph{${<\kappa}$-bounded}
iff
$\crit(E_\alpha)<\sup i^{M,m}_{\vec{E}\rest\alpha}``\kappa$
 for each $\alpha<\lambda$;
 and if $\kappa<\rho_m^M$, say
 $\vec{E}$ is \emph{$\kappa$-bounded}
 iff it is ${<(\kappa+1)}$-bounded.

We say  $\vec{E}$ is \emph{$(M,m)$-pre-pre-good}
iff either $\vec{E}$ is $(M,m)$-pre-good or there is $\gamma<\lh(\vec{E})$
such that $\vec{E}\rest\gamma$ is $(M,m)$-pre-good but not $(M,m)$-good.
\end{dfn}

As a corollary to \ref{rem:condensation_def}, we have:
\begin{lem}\label{lem:ult_pres_standard}
 Let $N$ be $n$-standard and $\vec{E}$ be $(N,n)$-good.
 Then $\Ult_n(N,\vec{E})$ is $n$-standard.
\end{lem}

The following fact was established in the proof of
\cite[Lemma 3.17(11)]{premouse_inheriting}, but that proof is
within a context which makes it a little annoying to
isolate, so we 
repeat the proof here 
for convenience:

\begin{lem}\label{lem:Ult_m_seg_Ult_n}
Let $N$ be $n$-standard and $m<n<\om$ with $\rho_{m+1}^N=\rho_n^N$.
Let  $\vec{E}$ be an $(N,n)$-good sequence.
Then $\vec{E}$ is $(N,m)$-good.
Let $U_k=\Ult_k(N,\vec{E})$ for $k\in\{m,n\}$.
Then $U_m\ins U_n$ and $\rho_{m+1}^{U_m}=\rho_{n}^{U_n}$.
\end{lem}
\begin{proof}
The fact that $\vec{E}$ is $(N,m)$-good will follow by induction on 
$\lh(\vec{E})$ from the rest.
So assume this holds.
 
 Let $i_k:M\to U_k$ be the ultrapower map
and $\pi:U_m\to U_{n}$ be the standard factor map.
By \ref{lem:ult_pres_standard} and calculations as in
\cite[Corollary 2.24]{extmax} and \cite{premouse_inheriting}, we have
\begin{enumerate}[label=--]
 \item $U_n$ is $n$-standard and $U_m$ is $(m+1)$-sound,
 \item $\rho\eqdef\rho_{m+1}^{U_m}=\sup i_m``\rho_{m+1}^N=\sup 
i_{n}``\rho_{n}^N=\rho_{n}^{U_n}$,
  \item if $\rho<\OR^{U_n}$ then $\rho$ is a $U_n$-cardinal
  (using that if $\rho_{m+1}^M=(\theta^+)^{M}$ then there is
  no cofinal $\bfrSigma_{m+1}^M$  function $f:\theta\to\rho_{m+1}^M$),  
 \item $\pvec_{m+1}^{U_k}=i_k(\pvec_{m+1}^N)$ for $k\in\{m,n\}$, and
 \item $\pi$ is $\pvec_{m+1}$-preserving $m$-lifting, with $\crit(\pi)\geq\rho$.
\end{enumerate}
So by \ref{rem:condensation_def} applied to $\pi:U_m\to U_n$ (noting $U_n$ is $n$-standard) either:
\begin{enumerate}[label=--]
 \item $U_m\pins U_n$ (when $\sup\pi``\rho_m^{U_m}<\rho_m^{U_n}$), or
 \item $U_m=U_n$ (when $\sup\pi``\rho_m^{U_m}=\rho_m^{U_n}$),
\end{enumerate}
completing the proof.
\end{proof}

\begin{dfn}
 Let $N$ be an $n$-sound premouse and $(M,m)\ins(N,n)$, where $m<\om$.
 The \emph{extended $((N,n),(M,m))$-dropdown}
 is the sequence $\left<(M_i,m_i)\right>_{i\leq k}$, with $k$ as large as 
possible,
 where $(M_0,m_0)=(M,m)$, and $(M_{i+1},m_{i+1})$ is the least 
$(M',m')\ins(N,n)$ such that either
 \begin{enumerate}[label=--]
  \item  $(M',m')=(N,n)$, or
  \item $(M_i,m_i)\pins(M',m')$ and $\rho_{m'+1}^{M'}<\rho_{m_i+1}^{M_i}$.
  \end{enumerate}
 
 The \emph{reverse extended $((N,n),(M,m))$-dropdown} is 
$\left<(M_{k-i},m_{k-i})\right>_{i\leq k}$.

Abbreviate \emph{reverse extended} with \emph{revex}
and  \emph{dropdown} with 
\emph{dd}.
\end{dfn}

Steel proved the following dropdown-preservation lemma (for $\lambda$-indexing);
a small part of it is independently due to the author:\footnote{\label{ftn:dropdown_pres}Steel first showed (a) in a
2015 preprint of \cite{str_comparison}. He and the author then 
noticed independently that it is also important for full normalization to know how the dropdown sequence is propagated by the ultrapower, and extended (a) to (b). Steel's formulation and proof of this in \cite[Lemma 2.4]{steel_local_HOD_comp} is somewhat different to the author's (which is given here, but which follows
readily from an examination of Steel's original proof of (a)).}

\begin{lem}\label{lem:ult_dropdown}
 Let $N$ be $n$-standard and $(M,m)\ins(N,n)$.
Let $\left<(M_i,m_i)\right>_{i\leq k}$ be the extended $((N,n),(M,m))$-dropdown.
Let $\vec{E}=\left<E_\alpha\right>_{\alpha<\lambda}$ be a
sequence
which is $(M_i,m_i)$-good for each $i\leq k$,
and $(M_i,m_i+1)$-good for each $i<k$. Let $U_i=\Ult_{m_i}(M_i,\vec{E})$ 
and $M'=U_0$
and $N'=U_k$. 
 Then:
 \begin{enumerate}[label=(\alph*)]\item  $(M',m)\ins(U_i,m_i)\pins(U_{i+1},m_{i+1})\ins (N',n)$
 for each $i<k$,
 and in fact, \item  the extended $((N',n),(M',m))$-dropdown is 
$\left<(U_i,m_i)\right>_{i\leq k}$.\footnote{The proof for almost
the same fact was given in \cite[Lemma 10.3]{fsfni}, but we include
the proof here also for self-containment, and since the fact is very central to our purposes.}
\end{enumerate}
\end{lem}
\begin{proof}
If $k=0$ it is trivial so suppose $k>0$ and fix $i<k$. It easily suffices to 
prove the following:
\begin{enumerate}
 \item\label{item:U_i,m_i_pins_U_i+1,m_i+1} $(U_i,m_i)\pins (U_{i+1},m_{i+1})$.
 \item\label{item:dropdown_of_consec_pair} The extended 
$((U_{i+1},m_{i+1}),(U_i,m_i))$-dropdown is 
$\left<(U_i,m_i),(U_{i+1},m_{i+1})\right>$,
 \item\label{item:strict_drop_pres}If $i+1<k$ (so 
$\rho_{m_{i+1}+1}^{M_{i+1}}<\rho_{m_i+1}^{M_i}$) then 
$\rho_{m_{i+1}+1}^{U_{i+1}}<\rho_{m_i+1}^{U_i}$.
 \item\label{item:projecta_at_end}If $k>0$ then either:
 \begin{enumerate}
  \item $M_{k-1}=N$ and $n>m_{k-1}$ and $\rho_n^N=\rho_{m_{k-1}+1}^N$ and 
$\rho_n^{N'}=\rho_{m_{k-1}+1}^{U_{k-1}}$, or
  \item $M_{k-1}\pins N$ and 
$\rho_n^N=\rho_{m_{k-1}+1}^{M_{k-1}}=\rho_\om^{M_{k-1}}$
  and $\rho_n^{N'}=\rho_{m_{k-1}+1}^{U_{k-1}}=\rho_\om^{U_{k-1}}$, or
  \item $M_{k-1}\pins N$ and 
$\rho_n^N>\rho_{m_{k-1}+1}^{M_{k-1}}=\rho_{\om}^{M_{k-1}}$ and
  $\rho_n^{N'}>\rho_{m_{k-1}+1}^{U_{k-1}}=\rho_\om^{U_{k-1}}$.
 \end{enumerate}
\end{enumerate}

We just give the proof assuming that $\lambda=\lh(\vec{E})=1$;
the general case is then a straightforward induction on $\lh(\vec{E})$.
Let $E=E_0$ and $\kappa=\crit(E)$. Note that for each $i<k$, we have
$\kappa<\rho_{m_i+1}^{M_i}$ and $(\kappa^+)^N=(\kappa^+)^{M_i}$,
and also $\kappa<\rho_n^N$.
Write $R=M_i$, $r=m_i$, $S=M_{i+1}$, $s=m_{i+1}$,
$R'=U_i$, and $S'=U_{i+1}$. So $(R,r)\pins(S,s)$.

To start with we prove parts 
\ref{item:U_i,m_i_pins_U_i+1,m_i+1}--\ref{item:strict_drop_pres}
assuming that $i+1<k$.

\begin{case} $R\pins S$ and if we are using MS-indexing then $R\pins S^\sq$.
 
 Let $\rho=\rho_{r+1}^R$. So $\rho=\rho_\om^R$ is a cardinal of $S$
and $\rho_{s+1}^S<\rho\leq\rho_s^S$. Therefore the functions 
$[\kappa]^{<\om}\to\alpha$, for $\alpha<\rho$,
which are used in forming $R'=\Ult_r(R,E)$, are exactly those used in forming 
$S'=\Ult_s(S,E)$.
Let $i^R:R\to R'$ and $i^S:S\to S'$ be the ultrapower maps and
$\pi:R'\to i^S(R)\pins S'$
the natural factor map. Then
like in the proof of Lemma \ref{lem:Ult_m_seg_Ult_n},
\[ \rho_{r+1}^{R'}=\sup i^R``\rho=\sup i^S``\rho\leq\crit(\pi)\]
and
$\rho_{r+1}^{R'}\leq\sup i^S``\rho_s^S=\rho_s^{S'}$.
Moreover, $\rho_{r+1}^{R'}$ is a cardinal of $S'$, because $\rho$ is a cardinal 
of $S$,
and if $\rho=(\gamma^+)^S$ then $\rho$ is regular in $S$. 
Note that either $\pi$ satisfies the requirements for $(r+1)$-relevant
(if $\pi``\rho_r^R$ is bounded in $\rho_r^{i^S(R)}$),
or for $(r+1)$-sub-condensing
(if $\pi``\rho_r^R$ is unbounded in $\rho_r^{i^S(R)}$ but $\rho_{r+1}^R<\rho_{r+1}^{i^S(R)}$),
 or $R'=i^S(R)$ and $\pi=\id$ (if $\pi``\rho_r^R$ is unbounded in $\rho_r^{i^S(R)}$ and $\rho_{r+1}^R=\rho_{r+1}^{i^S(R)}$).
But $S'$ is $0$-standard by \ref{lem:ult_pres_standard}, so $R'\ins i^S(R)\pins S'$.
Since also $\rho_{s+1}^{S}<\rho$ and $\rho_{s+1}^{S'}=\sup i^S``\rho_{s+1}^S$,
we have $\rho_{s+1}^{S'}<\rho_{r+1}^{R'}\leq\rho_s^{S'}$.
So parts \ref{item:U_i,m_i_pins_U_i+1,m_i+1}--\ref{item:strict_drop_pres} for 
this case follow.
\end{case}

\begin{case} $R\pins S$ but we are using MS-indexing and $R\npins S^\sq$.\footnote{This case is
a variant of an observation due to the author from a separate context.}
 
 Argue as in the previous case, replacing $i^S(R)$ (which is not defined as 
$\dom(i^S)=S^\sq$)
 with $\widehat{i^S}(R)$,
 noting that $\rho_{r+1}^{R'}\leq\nu(F^{S'})$, so $R'\pins S'$.
 Here
 \[ \widehat{i^S}:\Ult(S,F^S)\to\Ult(S',F^{S'})
 \]
 is the map induced by $i^S$ via the Shift Lemma.
\end{case}

\begin{case} $R=S$.

So $r<s$, and note that $\rho_{s+1}^S<\rho_s^S=\rho_{r+1}^S$.
Lemmas \ref{lem:Ult_m_seg_Ult_n} gives that $R'\ins S'$, and note that
\begin{equation}\label{eqn:R=S_ult_drop} 
\rho_{s+1}^{U_s}<\rho_s^{U_s}=\rho_{r+1}^{U_r}\eqdef\rho'.
\end{equation}
Note that $\left<(U_r,r),(U_s,s)\right>$
is the extended dropdown of $((U_s,s),(U_r,r))$:
For if $U_r=U_s$ this follows from line (\ref{eqn:R=S_ult_drop}) above;
if $U_r\pins U_s$ it is by line (\ref{eqn:R=S_ult_drop}) and because 
$\rho'=\rho_{r+1}^{U_r}=\rho_\om^{U_r}$ is a cardinal of $U_s$
(if $\rho_s^S=(\gamma^+)^S$ then $\rho_s^S$ is $\bfrSigma_s^S$-regular).
\end{case}

This completes the proof of parts 
\ref{item:U_i,m_i_pins_U_i+1,m_i+1}--\ref{item:strict_drop_pres}
assuming that $i+1<k$.

Now suppose that $k>0$. Suppose $M_{k-1}\pins N$. Then 
$\rho\eqdef\rho_{m_{k-1}+1}^{M_{k-1}}=\rho_\om^{M_{k-1}}$
is an $N$-cardinal. We have $\rho_n^N\geq\rho$, because if $\rho_n^N<\rho$
then $n>0$, and letting $n'$ be least such that $\rho_{n'+1}^N<\rho$,
then $(N,n')$ should have been in the dropdown sequence, a contradiction.

Suppose instead that $M_{k-1}=N$. Then by definition, $m_{k-1}<n$, so 
$\rho_n^N\leq\rho_{m_{k-1}+1}^N$.
But if $\rho_n^N<\rho_{m_{k-1}+1}^N$ then again, there should have been another 
element in the dropdown sequence.
So $\rho_n^N=\rho_{m_{k-1}+1}^N$.

Using these observations, one proceeds as before to establish parts 
\ref{item:U_i,m_i_pins_U_i+1,m_i+1},
\ref{item:dropdown_of_consec_pair} and \ref{item:projecta_at_end} for the case 
that $i+1=k$.
\end{proof}

\begin{rem}
 Note that in the context above,  if $M_i\pins M_{i+1}$ then $U_i\pins U_{i+1}$,
 but it is possible that $M_i=M_{i+1}$ and $U_i\pins  U_{i+1}$.
\end{rem}

\subsection{Ultrapower commutativity}

Essentially the following lemma was shown in \cite[\S3]{premouse_inheriting},
and there were related facts in \cite{mim} and \cite[\S2]{extmax}.
We discuss it in detail here though, in order to prepare for a generalization
which we need.

\begin{lem}[Extender commutativity]\label{lem:extender_comm}
 See Figure \ref{fgr:extender_comm}. Let $M$ be $m$-sound and $P$ be active.
 Let
 $G=F^P$ and $\kappa=\crit(G)$. Suppose
 $M||(\kappa^+)^M=P|(\kappa^+)^P$
 and $\kappa<\rho_m^M$. Suppose either $(\kappa^+)^M<\OR^M$
 or $M$ is active
 and $\kappa<\wt{\nu}(M)$.\footnote{This assumption
 is not so important, but will always hold where we use the lemma.} Let 
$U=\Ult_m(M,G)$
and suppose $U$ is wellfounded. 
 
 Let $\vec{E}$ be $(M,m)$-good, $(P,0)$-good
 and $\kappa$-bounded.
  Let
  \[ M_{\cd}=\Ult_m(M,\vec{E})\text{ and 
}P_\cd=\Ult_0(P,\vec{E})\text{ and }
  G_\cd=F^{P_\cd}, \]
so   \[ \kappa_{\cd}\eqdef 
i^{M,m}_{\vec{E}}(\kappa)=i^{P,0}_{\vec{E}}(\kappa)=\crit(G_\cd)<
\min(\rho_m^{M_{\cd}},\rho_0^{P_{\cd}}). \]
If $\vec{E}$ is $(U,m)$-pre-good, also let
$U_{\cd}=\Ult_m(U,\vec{E})$.

  Let $\vec{F}$ be $(P_\cd,0)$-good with
$\kappa_{\cd}<\crit(\vec{F})$.
Let $\vec{D}=\vec{E}\conc\vec{F}$. Let
\[ P^\cd=\Ult_0(P,\vec{E}\conc\vec{F})\text{ and }
G^\cd=F^{P^\cd}.\]
Let 
$\kappa^{\cd}=\crit(G^{\cd})=\crit(G_{\cd})=\kappa_{
\cd}$.

If $\vec{E}\conc\vec{F}$ is $(U,m)$-pre-good, also let
$U^{\cd}=\Ult_m(U,\vec{E}\conc\vec{F})=\Ult_m(U_{\cd},\vec{F})$.

Let
 \[ \wt{U}_{\cd}=\Ult_m(M_{\cd},G_\cd)\text{ and 
}\wt{U}^{\cd}=\Ult_m(M_{\cd},G^\cd)\]
\tu{(}see part \ref{item:M',N'_agmt} below\tu{)},
and suppose $\wt{U}^{\cd}$ is 
wellfounded.  Then:
 \begin{enumerate}
   \item $\vec{E}\conc\vec{F}$ is $(U,m)$-good.
   \item\label{item:M',N'_agmt} $M_{\cd}||
   \kappa_{\cd}^{+{M_{\cd}}}
=P_\cd|\kappa_{\cd}^{+P_\cd}
=P^\cd|(\kappa^{\cd})^{+P^\cd}$
   (so $\wt{U}_{\cd}$ and $\wt{U}^{\cd}$
   are well-defined),
  \item $U_{\cd}=\wt{U}_{\cd}$ and 
$U^\cd=\wt{U}^{\cd}$.
  \item\label{item:maps_commute} The various ultrapower maps commute,
  as indicated in Figure \ref{fgr:extender_comm}; that is,
  \[ i^{U,m}_{\vec{E}\conc\vec{F}}\com 
i^{M,m}_G=
i^{\wt{U}_{\cd},m}_{\vec{F}}\com 
i^{M_{\cd},m}_{G_\cd}\com i^{M,m}_{\vec{E}}=
i^{M_{\cd},m}_{G^\cd}\com 
i^{M,m}_{\vec{E}}.\]
   \item\label{item:ult_map_is_SL_map} The hypotheses for the Shift Lemma hold
with respect to $(M,P)$,
$(M_{\cd},P^\cd)$,
   and the maps
   \[ i^{M,m}_{\vec{E}}:M\to M_{\cd}\text{ and
    }i^{P,0}_{\vec{E}\conc\vec{F}}:P\to P^\cd.\]
    Moreover, $i^{U,m}_{\vec{E}\conc\vec{F}}$ is just the Shift Lemma map.
 \end{enumerate}
\end{lem}

\begin{figure}
\centering
\begin{tikzpicture}
 [mymatrix/.style={
    matrix of math nodes,
    row sep=1.6cm,
    column sep=1.2cm}
  ]
   \matrix(m)[mymatrix]{
 U^{\cd}=\wt{U}^{\cd}&{}&P^\cd\\
 U_{\cd}=\wt{U}_{\cd} &M_{\cd}  &P_\cd   \\
U         &M       & P \\
};

\path[->,font=\scriptsize]
(m-3-1) edge node[left] {$\vec{E},m$} (m-2-1)
(m-2-1) edge node[left] {$\vec{F},m$} (m-1-1)
(m-3-2) edge node[below] {$G,m$} (m-3-1)
(m-3-2) edge node[left] {$\vec{E},m$} (m-2-2)
(m-2-2) edge node[below] {$G_\cd,m$} (m-2-1)
(m-2-2) edge node[right,pos=0.6] {$\ \ G^\cd,m$} (m-1-1)
(m-3-3) edge node[right] {$\vec{E},0$} (m-2-3)
(m-2-3) edge node[right] {$\vec{F},0$} (m-1-3)
(m-3-1) edge[bend left=65] node[left] {$\vec{D},m$} (m-1-1)
(m-3-3) edge[bend left] node[left] {$\vec{D},0$} (m-1-3)
;
\end{tikzpicture}
\caption{Extender commutativity. The diagrams commute,
where $\vec{D}=\vec{E}\conc\vec{F}$,
and a label $\vec{C},k$ denotes
a degree $k$ abstract iteration map given by $\vec{C}$.} 
\label{fgr:extender_comm}
\end{figure}
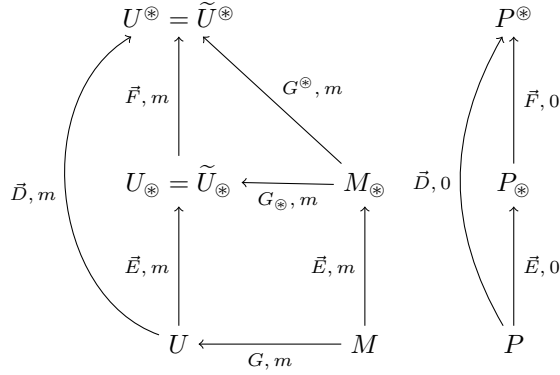

\begin{proof}
Let $\vec{D}=\vec{E}\conc\vec{F}$.

Since $G=F^P$, we have
$\OR^P\leq\rho_m^U$ and either
\begin{enumerate}[label=--]
 \item $P^\passive\pins_{\card}U$, or
\item $P^\passive=U^\passive$,
$F^P$ is of superstrong type
and $M$ is MS-indexed type 2 with largest cardinal $\kappa$.
\end{enumerate}

Write
$P_\alpha=\Ult_0(P,\vec{D}\rest\alpha)$
and  $U_\alpha=\Ult_m(U,\vec{D}\rest\alpha)$, where
$\alpha$ is is large as possible
that $\vec{D}\rest\alpha$ is $(U,m)$-pre-good.
Since $\vec{D}$ is $(P,0)$-good,
we therefore get by induction on $\beta\leq\alpha$
that  $\OR^{P_\beta}\leq\rho_m^{U_\beta}$ and either
\begin{enumerate}[label=--]
 \item $(P_\beta)^\passive\pins_{\card}U_\beta$, with $\OR^{P_\beta}$
 in the 
wellfounded part of $U_\beta$,
or
\item $(P_\beta)^\passive=(U_\beta)^\passive$, which is wellfounded,
\end{enumerate}
and the ultrapower maps agree over $P$,
or over $P^\sq$ if $P$ is MS-indexed type 3.
So either $\vec{D}$ is 
$(U,m)$-good,
 $\vec{D}\rest\alpha$ is $(U,m)$-pre-good
but $U_\alpha$ is illfounded.
So renaming, we may assume that $\alpha=\lh(\vec{D})$, 
so $\vec{D}$ is $(U,m)$-pre-good
and $P_\alpha=P^{\cd}$
and $U_\alpha=U^{\cd}$.

Let $H$ be the (long) $M$-extender measuring $\pow(\kappa)\inter M$,
derived from
\[ j= i^{M_{\cd},m}_{G^\cd}\com i^{M,m}_{\vec{E}}:M\to\wt{U}^{\cd}, \]
of length $\nutilde(G^\cd)$. Clearly $\wt{U}^{\cd}=\Ult_m(M,H)$
and $j=i^{M,m}_{H}$.

Let $H'$ be the (long) $M$-extender measuring $\pow(\kappa)\inter M$,
derived from
\[j'\eqdef i^{U,m}_{\vec{D}}\com i^{M,m}_G:M\to U^{\cd}, \]
of length $\nutilde(G^\cd)$. Let
$\sigma:\Ult_m(M,H')\to U^{\cd}$
be the standard factor map. Then $\Ult_m(M,H')=U^{\cd}$
and $\sigma=\id$,
because $G$ is generated by $\nutilde(G)$ and
$\vec{D}$ is ${<\nutilde(G)}$-bounded and
$H'$ has length $\nutilde(G^\cd)$, and
\[ \nutilde(G^\cd)=\sup i^{P,0}_{\vec{D}}``\nutilde(G)=
 \sup i^{U,m}_{\vec{D}}``\nutilde(G).
\]
Therefore $j'=i^{M,m}_{H'}$.

\begin{clm} $H'=H$.\end{clm}

\begin{proof}
Let $A\in\pow([\kappa]^{<\om})\inter M$.
For ease of reading we assume that $A\in\pow(\kappa)$
and that $P$ is not MS-indexed type 3
(hence $\nutilde(G^\cd)=i^{P,0}(\nutilde(G))$), but the other cases
are simple variants. 
We want to see
\[ j'(A)\inter\nutilde(G^\cd)=j(A)\inter\nutilde(G^\cd). \]
But 
\[ 
j'(A)\inter\nutilde(G^\cd)=i^{U,m}_{\vec{D}}
(i_G(A)\inter\nutilde(G))  = i^{P,0}_{\vec{D}}(i_G(A)\inter\nutilde(G)) \]
(the second equality as $i_G(A)\inter\nutilde(G)\in P$, over which 
$i^{U,m}_{\vec{D}}$ agrees with $i^{P,0}_{\vec{D}}$)
\[ = i_{G^\cd}(i^{P,0}_{\vec{D}}(A))\inter\nutilde(G^\cd) \]
(by definition of how $F^P$ shifts to $F^{P^\cd}$ under ultrapowers)
\[ = i_{G^\cd}(i^{P,0}_{\vec{E}}(A))\inter\nutilde(G^\cd)\]
(since $i^{P,0}_{\vec{D}}(A)=i^{P,0}_{\vec{E}}(A)$, since 
$\crit(\vec{F})>\kappa^{\cd}$)
\[ = i^{M_{\cd},m}_{G^\cd}(i^{M,m}_{\vec{E}}(A))\inter\nutilde(G^\cd) = 
j(A)\inter\nutilde(G^\cd)\]
(by agreement of ultrapower maps),
as desired.
\end{proof}

By the claim, $U^{\cd}=\wt{U}^{\cd}$ and the corresponding ultrapower maps 
commute.
The rest of parts \ref{item:M',N'_agmt}--\ref{item:maps_commute} follow from 
this, by considering the special cases
that either $\vec{F}=\emptyset$ or $\vec{E}=\emptyset$.

Part \ref{item:ult_map_is_SL_map} follows from the 
commutativity and agreement between $i^{U,m}_{\vec{D}}$
and $i^{P,0}_{\vec{D}}$, and by the elementarity of the maps
(it is also like in \cite[Lemma 4.20]{iter_for_stacks}).
\end{proof}

We next want to generalize the preceding lemma to deal with the case
of a (normal) sequence $\vec{G}$ of extenders, instead of just a single extender 
$G$.

\begin{dfn}
 Let $P$ be an active premouse and $\vec{F}$ be a sequence of extenders
 which is $(P,0)$-good. Let $P_\eta=\Ult_0(P,\vec{F}\rest\eta)$.
 We say that $\vec{F}$ is:
 \begin{enumerate}[label=--]
  \item  \emph{$(P,0)$-strictly-$\nutilde$-bounded} iff
 ${<\nutilde(P)}$-bounded,
 \item
 \emph{$(P,0)$-critical-bounded}
 iff $\crit(F^P)$-bounded.
\end{enumerate}

 Let $\vec{P}=\left<P_\alpha\right>_{\alpha<\lambda}$ be a sequence of active 
premice.
Say $\vec{P}$ and $\left<F^{P_\alpha}\right>_{\alpha<\lambda}$
 are \emph{normal} iff $\nutilde(P_\alpha)\leq\crit(F^{P_\beta})$
 and $(P_\alpha)^\passive\pins_{\card} P_\beta$
 for $\alpha<\beta<\lambda$.
\end{dfn}

\begin{dfn}\label{dfn:G*F}
 Let $\left<Q_\alpha\right>_{\alpha<\lambda}$
 be a normal sequence of active premice.
  Let $G_\alpha=F^{Q_\alpha}$
  and $\vec{G}=\left<G_\alpha\right>_{\alpha<\lambda}$.
  Let $\left<P_\alpha,F_\alpha\right>_{\alpha<\theta}$ and $\vec{F}$ be 
likewise.
 
 Let $\alpha<\lambda$. Let $\eta_\alpha$ be the largest $\eta\leq\theta$
 such that $\vec{F}\rest\eta$ is
 $(Q_\alpha,0)$-pre-good
 and ${<\nutilde(Q_\alpha)}$-bounded.
 Suppose that $\vec{F}\rest\eta_\alpha$ is $(Q_\alpha,0)$-good.
 Let $\xi_\alpha$ be the largest $\xi\leq\theta$ such that
 $\vec{F}\rest\xi$ is $(Q_\alpha,0)$-pre-good
 and $\crit(F^{Q_\alpha})$-bounded.
 Note that $\xi_\alpha\leq\eta_\alpha$,
 so $\vec{F}\rest\xi_\alpha$ is also $(Q_\alpha,0)$-good.
By normality, $\eta_\beta\leq\xi_\alpha$ for $\beta<\alpha$.

 Write $Q^\cd_\alpha=\Ult_0(Q_\alpha,\vec{F}\rest\eta_\alpha)$ and 
$G^\cd_\alpha=F^{Q^\cd_\alpha}$.
 Given $\beta<\theta$, say that $F_\beta$ is \emph{nested} (with respect to this $\cd$-product) iff 
$\xi_\alpha\leq\beta<\eta_\alpha$ for some $\alpha<\theta$;
 and \emph{unnested} otherwise.
  Then the $\cd$-product $\vec{G}\cd\vec{F}$ denotes the enumeration of
 \[ X=\{G^\cd_\alpha\}_{\alpha<\lambda}\cup\{F_\alpha\mid\alpha<\theta\text{ and 
}F_\alpha\text{ is unnested}\} \]
 in order of increasing critical point.
 And $\vec{Q}\cd\vec{P}$ denotes the corresponding enumeration of
 \[ \{Q^\cd_\alpha\}_{\alpha<\lambda}\cup\{P_\alpha\mid\alpha<\theta\text{ and 
}F_\alpha\text{ is unnested}\}.\]
 
 In this context, we also write 
$Q_{\alpha\cd}=\Ult_0(Q_\alpha,\vec{F}\rest\xi_\alpha)$
 and $G_{\alpha\cd}=F^{Q_{\alpha\cd}}$.
\end{dfn}
\begin{lem}
 Adopt the hypotheses and notation of \ref{dfn:G*F}. Let 
 $M$ be $m$-sound.
 Suppose that $\vec{G}$ is $(M,m)$-pre-good and
 $\vec{G}\cd\vec{F}$ is $(M,m)$-good.
 Let $U=\Ult_m(M,\vec{G})$.
Then:
 \begin{enumerate}
  \item\label{item:crits_disagree} If $E,F\in X$ then $\crit(E)\neq\crit(F)$, so 
the ordering of $\vec{G}\cd\vec{F}$ is well-defined.
  \item\label{item:seqs_are_normal} $\vec{G}\cd\vec{F}$ and $\vec{Q}\cd\vec{P}$ 
are normal sequences.
  \item\label{item:*_matches_comp} $\vec{G}\cd\vec{F}$ is equivalent to 
$\vec{G}\conc\vec{F}$; that is,
 $\vec{G}$ is $(M,m)$-good, $\vec{F}$ is $(U,m)$-good,
  \[ \Ult_m(M,\vec{G}\conc\vec{F})=\Ult_m(M,\vec{G}\cd\vec{F}) \]
  and the associated ultrapower maps \tu{(}and hence derived extenders\tu{)} 
agree.
 \end{enumerate}
\end{lem}
\begin{proof}\footnote{The notation in the proof does not match well with the previous lemma;
this will be remedied in a future version.}
Parts \ref{item:crits_disagree} and \ref{item:seqs_are_normal} are routine.
Consider part \ref{item:*_matches_comp}. Its proof is basically a verification 
that the diagram in Figure \ref{fgr:iterated_extender_comm} commutes.
Note that in the diagram, all arrows labelled with extenders or sequences 
thereof
correspond to degree $m$ ultrapowers by those extenders.
In the diagram and in what follows, given a sequence $\vec{E}$ of extenders, we 
write $\vec{E}_{[\alpha,\beta)}$
for $\vec{E}\rest[\alpha,\beta)$.
The maps $\sigma_{\alpha\beta}$
displayed in the diagram are the natural factor maps
between degree $m$ ultrapowers of $M$ by certain natural segments of 
$\vec{G}\cd\vec{F}$,
and will be specified below.
The reader will then happily verify that $U$, in the top right corner of the 
diagram,
 is just $\Ult_m(M,\vec{G}\cd\vec{F})$,
with its ultrapower map derived along the main diagonal of the diagram
(passing from $M$ to $U$); and also that $U$ is just 
$\Ult_m(M,\vec{G}\conc\vec{F})$,
with its ultrapower map derived from the composition
of the maps along the bottom and right side. So this will complete the proof.

\begin{figure}
\centering
\begin{tikzpicture}
 [mymatrix/.style={
    matrix of math nodes,
    row sep=0.5cm,
    column sep=0.7cm}
  ]
   \matrix(m)[mymatrix]{
       {}&              {}&                          {}&                        
{}&                      {}&  {}&                          {}&  U\\
       {}&              {}&                          {}&                       
{}&                      {}&  {}&                          {}&  U_\lambda\\
       {}&              {}&                          {}&                        
{}&                      {}&  {}&                    U_\delta&  {}\\
       {}&              {}&                          {}&                       
{}&                      {}&  {}&                          {}&  {}\\   
       {}&              {}&                          {}&                        
{}&              U_{\eps+1}&  {}&                          {}&  {}\\
       {}&              {}&                          {}&           
\bar{U}_{\eps+1}&         \wt{U}_{\eps+1}&  {}&                          {}&  
{}\\
       {}&              {}&                          {}&                     
U_\eps&                      {}&  {}&                          {}&  {}\\ 
       {}&              {}&                          {}&                       
{}&                      {}&  {}&                          {}&  {}\\
       {}&             U_1&                          {}&                        
{}&                      {}&  {}&                          {}&  {}\\
\bar{U}_1&        \wt{U}_1&                          {}&                       
{}&                      {}&  {}&                          {}&  {}\\
        M&             M_1&                          {}&                     
M_\eps&              M_{\eps+1}&  {}&                    M_\delta&  M_\lambda\\
};
\path[->,font=\scriptsize]
(m-11-1) edge node[below] {$G_0$} (m-11-2)
(m-11-2) edge node[below] {$\vec{G}_{[1,\eps)}$}(m-11-4)
(m-11-4) edge node[below] {$G_\eps$}(m-11-5)
(m-11-5) edge node[below] {$\vec{G}_{[\eps+1,\delta)}$} (m-11-7)
(m-11-7) edge node[below] {$\vec{G}_{[\delta,\lambda)}$}(m-11-8)
(m-11-1) edge node[left] {$\vec{F}_{[0,\xi_0)}$} (m-10-1)
(m-11-2) edge node[right] {$\vec{F}_{[0,\xi_0)}$} (m-10-2)
(m-10-2) edge node[right] {$\vec{F}_{[\xi_0,\eta_0)}$} (m-9-2)
(m-11-4) edge node[right] {$\vec{F}_{[0,\eta_{<\eps})}$} (m-7-4)
(m-7-4) edge node[left] {$\vec{F}_{[\eta_{<\eps},\xi_\eps)}$}(m-6-4)
(m-11-5) edge node[right] {$\vec{F}_{[0,\xi_\eps)}$} (m-6-5)
(m-6-5) edge node[right] {$\vec{F}_{[\xi_\eps,\eta_\eps)}$} (m-5-5)
(m-11-7) edge node[right] {$\vec{F}_{[0,\eta_{<\delta})}$} (m-3-7)
(m-11-8) edge node[right] {$\vec{F}_{[0,\eta_{<\lambda})}$} (m-2-8)
(m-2-8) edge node[left] {$\vec{F}_{[\eta_{<\lambda},\theta)}$} (m-1-8)
(m-10-1) edge node[below] {$G_{0\cd}$} (m-10-2)
(m-10-1) edge node[left,pos=0.8] {$G^\cd_0\ \ $} (m-9-2)
(m-9-2) edge node[below] {$\ \ \ \ \sigma_{1\eps}$} (m-7-4)
(m-6-4) edge node[below] {$G_{\eps\cd}$} (m-6-5)
(m-6-4) edge node[left,pos=0.8] {$G^\cd_{\eps}\ \ $} (m-5-5)
(m-5-5) edge node[below] {$\ \ \ \ \ \ \sigma_{\eps+1,\delta}$} (m-3-7)
(m-3-7) edge node[below] {$\ \ \ \ \sigma_{\delta\lambda}$} (m-2-8);
\end{tikzpicture}
\caption{The diagram commutes. Arrows labelled with (sequences of) extender(s) 
indicate
the degree $m$ ultrapower map determined by that (sequence of) extender(s),
and that the structure at the tip of the arrow is the degree $m$ ultrapower of 
the 
structure at its base.
The unlabelled arrows correspond to ultrapowers by the appropriate middle 
segment
of $\vec{G}\cd\vec{F}$.
Given a sequence $\vec{E}$, $\vec{E}_{[\alpha,\beta)}$ denotes 
$\vec{E}\rest[\alpha,\beta)$.
} \label{fgr:iterated_extender_comm}
\end{figure}
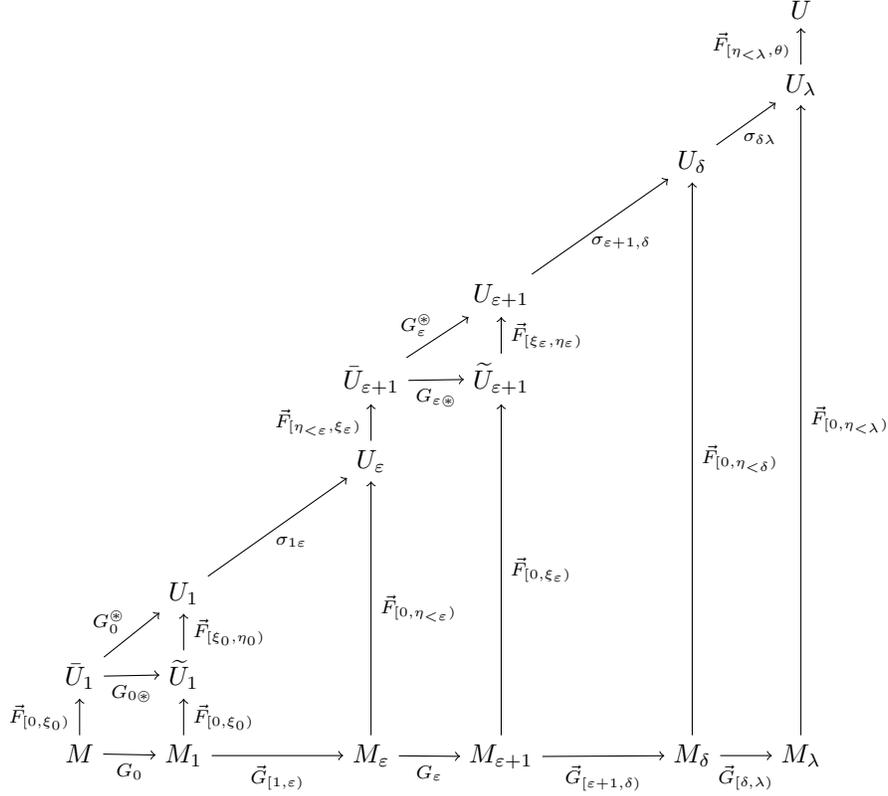

Now for each $\eps\leq\lambda$ let 
$\eta_{<\eps}=\sup_{\alpha<\eps}\eta_{\alpha}$.
Let
\[ M_\eps=\Ult_m(M,\vec{G}\rest\eps)\]
(so $M_0=M$ and $M_\lambda=U=\Ult_m(M,\vec{G})$) and
\[ U_\eps=\Ult_m(M_\eps,\vec{F}_{[0,\eta_{<\eps})})\]
(so $U_0=M$).

Let $\eps<\lambda$. Let $\kappa=\crit(G_\eps)$.
Note that \ref{lem:extender_comm} applies to the sub-diagram
of Figure \ref{fgr:iterated_extender_comm} with corners
$M_\eps$, $M_{\eps+1}$, $\bar{U}_{\eps+1}$ and $U_{\eps+1}$, and
where (by induction)
\[ 
\bar{U}_{\eps+1}=\Ult_m(U_\eps,\vec{F}_{[\eta_{<\eps},\xi_\eps)})=\Ult_m(M_\eps,
\vec{F}_{[0,\xi_\eps)}), \]
\[ 
\wt{U}_{\eps+1}=\Ult_m(\bar{U}_{\eps+1},G_{\eps\cd})=\Ult_m(M_{\eps+1},\vec{F}_{
[0,\xi_\eps)}), \]
and in particular,
\[ U_{\eps+1}=\Ult_m(\bar{U}_{\eps+1},G_\eps^\cd) \]
and the sub-diagram commutes. So let $\sigma_{\eps,\eps+1}:U_\eps\to U_{\eps+1}$ 
be the resulting map, that is,
\[ \sigma_{\eps,\eps+1}=i^{\bar{U}_{\eps+1},m}_{G^\cd_{\eps}}
\com
i^{U_\eps,m}_{\vec{F}_{[\eta_{<\eps},\xi_\eps)}}.\]

Now for $\alpha\leq\beta\leq\lambda$
let $\sigma_{\alpha\beta}:U_\alpha\to U_\beta$ be the commuting map
now induced by composition and direct limits. We claim this makes sense, in that
for each limit $\delta\leq\lambda$,
\begin{equation}\label{eqn:U_delta_dirlim}
U_\delta=\dirlim_{\alpha\leq\beta<\delta}\left(U_\alpha,U_\beta;\sigma_{
\alpha\beta}\right),\end{equation}
and that moreover, for all $\delta\leq\lambda$ and $\alpha\leq\delta$, we have
\[ k\eqdef\sigma_{\alpha\delta}\com
i^{M_\alpha,m}_{\vec{F}_{[0,\eta_{<\alpha})}}=
i^{M_\delta,m}_{\vec{F}_{[0,\eta_{<\delta})}}\com
i^{M_\alpha,m}_{\vec{G}_{[\alpha,\delta)}}\eqdef k'.\]
This is verified by a straightforward induction on $\delta$.
For successor $\delta$ it is as discussed above, and for limit $\delta$,
assuming for simplicity that $m=0$,
letting
\[ \sigma'_{\alpha\delta}:U_\alpha\to U_\delta \]
be
\[ \sigma'_{\alpha\delta}\big(i^{M_\alpha}_{\vec{F}_{[0,\eta_{<\alpha})}}(f)(a)\big)=
i^{M_\delta}_{\vec{F}_{[0,\eta_{<\delta})}}\big(i^{M_\alpha}_{\vec{G}_{[\alpha,\delta)}}(f)\big)(a)\]
for $a\in[\crit(\vec{F}_{[\eta_{<\alpha},\infty)})]^{<\om}$,
then for every $x\in U_\delta$, there is $\alpha<\delta$ with $x\in\rg(\sigma'_{\alpha\delta})$,
and note that this then yields the inductive hypotheses and that
$\sigma_{\alpha\delta}=\sigma_{\alpha\delta}'$.
(One can also argue like in part of the proof of \ref{lem:extender_comm}:
by commutativity, one derives the same extender
from $k$ as from $k'$ (the maps above), and the direct limit is in fact the ultrapower by this 
extender,
since it is a direct limit of smaller ultrapowers by sub-extenders thereof
(note here that on both sides, the derived extenders do have the same 
generators).)

So the diagram commutes, and the lemma easily follows.
\end{proof}

\section{Minimal strategy condensation}\label{sec:min_strat_cond}

We now proceed to adapt much of \cite{iter_for_stacks},
with the most fundamental change being in how extenders
are copied from a tree $\Tt$ into a (now \emph{minimal}) inflation $\Xx$.

\subsection{Minimal tree embeddings}

\begin{dfn}[Tree dropdown]\label{dfn:tree_dropdown}
Let $M$ be a $m$-sound premouse and let $\Tt$ be a putative $m$-maximal 
tree on $M$.

For $\beta+1<\lh(\Tt)$ let 
$(\lambda_\beta,d_\beta)=(\lh(E^\Tt_\beta),0)$. For $\beta+1=\lh(\Tt)$ (if 
$\lh(\Tt)$ is a successor and $M^\Tt_\beta$ well-defined) let 
$(\lambda_\beta,d_\beta)=(\OR(M^\Tt_\beta),\deg^\Tt(\beta))$. Let $\beta<\lh(\Tt)$. Let 
$\left<M_{\beta i},m_{\beta i}\right>_{i\leq k_\beta}$ be the 
\emph{reversed} extended dropdown of
\[ ((M^\Tt_\beta,\deg^\Tt(\beta)),(M^\Tt_\beta|\lambda_\beta,d_\beta))\] (note 
this defines $k_\beta$).
Then  $k_\beta^{\Tt}\eqdef k_\beta$ and $M_{\beta i}^{\Tt}\eqdef M_{\beta i}$ and $m^\Tt_{\beta i}=m_{\beta i}$. Let $\theta\leq\lh(\Tt)$.
We define the \emph{dropdown domain} $\ddd^{(\Tt,\theta)}$ of $(\Tt,\theta)$ by
\[ \Delta=\ddd^{(\Tt,\theta)}\eqdef\{(\beta,i)\mid\beta<\theta\ \&\ i\leq k_\beta\}, \]
and define the \emph{dropdown 
sequence} $\dds^{(\Tt,\theta)}$ of $(\Tt,\theta)$ by
\[ \dds^{(\Tt,\theta)}\eqdef\left<(M_{\beta 
i},m_{\beta i})\right>_{(\beta,i)\in\Delta}. \]

The \emph{dropdown sequence} $\dds^\Tt$ of $\Tt$ is 
$\dds^{(\Tt,\lh(\Tt))}$, and 
the \emph{dropdown domain} $\ddd^\Tt$ of $\Tt$ is $\ddd^{(\Tt,\lh(\Tt))}$.

Given $\kappa<\nutilde^\Tt_\alpha$
for some $\alpha+1<\lh(\Tt)$,
$\alpha^\Tt_\kappa$ denotes the least such $\alpha$,
and $n^\Tt_\kappa$ denotes the largest $n\leq k^\Tt_\alpha$
such that $n=0$ or
$\rho_{m_{\alpha n}+1}(M_{\alpha i})\leq\kappa$.
If instead $\lh(\Tt)=\alpha+1$
and $\kappa\leq\OR(M^\Tt_\alpha)$
but $\nutilde^\Tt_\beta\leq\kappa$
for all $\beta+1<\lh(\Tt)$, then $\alpha^\Tt_\kappa=\alpha$
and $n^\Tt_\kappa=0$.
\end{dfn}

So if 
 $\kappa=\crit(E^\Tt_\beta)$ then
 $\pred^\Tt(\beta+1)=\alpha^\Tt_\kappa$
 and 
$M^{*\Tt}_{\beta+1}=M^\Tt_{\alpha^\Tt_\kappa n^\Tt_\kappa}$.

\begin{rem}
 Recall that for an iteration tree $\Xx$, $\clint^\Xx$ denotes the set of closed $<_\Xx$-intervals.
\end{rem}

We now define the notion of an \emph{minimal tree embedding} $\Pi:\Tt\hookrightarrow_\minim\Xx$ between normal trees 
$\Tt,\Xx$ (actually we allow $\Tt$ to be a putative tree). The definition is just that of
\emph{tree embedding} from \cite{iter_for_stacks},
except that we modify how lift extenders $E^\Tt_\alpha$
of $\Tt$ into $\Xx$, and therefore must also modify how we lift the associated 
dropdown sequence.
In \cite{iter_for_stacks} the lift of $E^\Tt_\alpha$ is just its image 
$\pi(E^\Tt_\alpha)$ under a copy map $\pi$.
Here, associated with our copy maps $\pi$ we will also have a sequence $\vec{E}$ of extenders,
and $\pi$ will just be the ultrapower map associated to $\Ult_n(N,\vec{E})$,
for some $(N,n)$ in the dropdown sequence of $\Tt$, such that
$(\exit^\Tt_\alpha,0)\ins(N,n)\ins(M^\Tt_\alpha,\deg^\Tt(\alpha))$,
and $\vec{E}$ will be $(\exit^\Tt_\alpha,0)$-good.
We will  lift $E^\Tt_\alpha$ to
$\Ult_0(E^\Tt_\alpha,\vec{E})$;
that is, the active extender of $\Ult_0(\exit^\Tt_\alpha,\vec{E})$.
The dropdown sequence is lifted analogously. The rest is just a straightforward
modification of the notion of tree embedding.

\newcommand{\pre}{\mathrm{pre}}
\begin{dfn}[Tree pre-embedding]\label{dfn:tree_embedding}
(Cf.~\cite[Figure 1]{iter_for_stacks}.)
Let $M$ be an $m$-sound premouse, let $\Tt,\Xx$ be putative
$m$-maximal 
trees on $M$,
with $\Xx$ a true tree,
and $\theta\leq\lh(\Tt)$. A \emph{tree pre-embedding} from $(\Tt,\theta)$ 
to $\Xx$,
denoted
\[ \Pi:(\Tt,\theta)\hookrightarrow_{\pre}\Xx, \]
is a sequence $\Pi=\left<I_\alpha\right>_{\alpha<\theta}$
such that (cf.~\cite[Figure 1]{iter_for_stacks}):

\begin{enumerate}
\item\label{item:I_beta_Xx_clint} $I_\beta\in\clint^\Xx$ for each $\beta<\theta$. Let
$[\gamma_\beta,\delta_\beta]_\Xx\eqdef I_\beta$
and $\Gamma:\theta\to\lh(\Xx)$ be $\Gamma(\beta)=\gamma_\beta$.
\item $\gamma_0=0$.
\item $\Gamma$ preserves $<$, is continuous, and sends successors 
to successors.
\item $\beta_0<_\Tt\beta_1\iff\gamma_{\beta_0}<_\Xx\gamma_{\beta_1}$.
\item\label{item:degree_match} $\deg^\Xx(\gamma_\beta)=\deg^\Tt(\beta)$.
\item For $\beta+1<\theta$, we have $\gamma_{\beta+1}=\delta_\beta+1$.
\item\label{item:pred_pres} For $\beta+1<\theta$, letting $\xi=\pred^\Tt(\beta+1)$, we have
$\pred^\Xx(\gamma_{\beta+1})\in I_\xi$
(in \cite[Figure 1]{iter_for_stacks},
$\eta_{\beta+1}=\pred^\Xx(\gamma_{\beta+1})$)
and\footnote{\label{ftn:dropping_error}Version v1 of this document on arXiv stated $\dr^\Xx\inter(\gamma_\xi,\gamma_{\beta+1}]_\Xx=\emptyset\iff \beta+1\notin 
\dr^\Tt$, which is incorrect; it needs the ``sub-$\deg$''s on both. It can be that $\Xx$ drops in model,
but $\Tt$ only drops in degree.}
\[ \dr^\Xx_{\deg}\inter(\gamma_\xi,\gamma_{\beta+1}]_\Xx=\emptyset\iff \beta+1\notin 
\dr^\Tt_{\deg}.\]
\end{enumerate}
We say $\Pi$ has \emph{degree $m$}.
\end{dfn}

\begin{rem}\label{lem:intervals_I_cover_X-branches} It follows that:
\begin{enumerate}[label=(\roman*)]\item the 
$<$-intervals $[\gamma_\beta,\delta_\beta]$ partition  
$\sup_{\beta<\theta}\delta_\beta$, \item for
$\xi,\zeta<\theta$, we have
$(\gamma_\xi,\gamma_\zeta]_\Xx\inter\dropset^\Xx_{\deg}=\emptyset$
iff
$(\xi,\zeta]_\Tt\inter\dropset^\Tt_{\deg}=\emptyset$,\footnote{The same error mentioned in Footnote \ref{ftn:dropping_error} occured here, and in fact persisted to arXiv version v2.}
\item for each limit 
$\beta<\theta$, we 
have $\Gamma``[0,\beta)_\Tt\sub_{\text{cof}}[0,\gamma_\beta)_\Xx$,
\item if $\lh(\Tt)=\alpha+1$ then $M^\Tt_\alpha$ is well-defined, and
\item\label{item:<^X_down_close_I_xi} as in \cite{iter_for_stacks},
 if
 $\alpha\in I_{\xi}$ and $\delta\leq_\Xx\alpha$ then
$\delta\in I_{\zeta}$ for some $\zeta\leq_\Tt\xi$.
\end{enumerate}
\end{rem}

\begin{dfn}
 Let $\Xx$ be an iteration tree and $\alpha\leq_\Xx\beta$.
 Let
 \[ D=\{\gamma\mid\gamma+1\in(\alpha,\beta]_\Xx\}.\]
 Then $\vec{E}^\Xx_{\alpha\beta}$ denotes $\left<E^\Xx_\gamma\right>_{\gamma\in D}$
 (note that when $(\alpha,\beta]_\Xx$ does not drop,
 this extender sequence corresponds to $i^\Xx_{\alpha\beta}$).
\end{dfn}

\begin{dfn}\label{dfn:F-vec_tree_emb}
 Let $\Pi:(\Tt,\theta)\hookrightarrow_{\pre}\Xx$ 
 and write $I_\alpha=I^\Pi_\alpha$ etc.
For $\xi\in\bigcup_{\alpha<\theta}I_\alpha$, define the \emph{inflationary 
extender sequence}
$\vec{F}_\xi=\vec{F}^\Pi_\xi$ by:
\begin{enumerate}[label=--]
\item $\vec{F}_{\gamma_0}=\vec{F}_0=\emptyset$
\item for $\xi\in(\gamma_\alpha,\delta_\alpha]_\Xx$,
$\vec{F}_{\xi}=\vec{F}_{\gamma_\alpha}\conc\vec{E}^\Xx_{\gamma_\alpha\xi}$,
\item for $\alpha+1<\theta$, 
$\vec{F}_{\gamma_{\alpha+1}}=\vec{F}_{\delta_\alpha}$,
\item for limit $\alpha<\theta$,
$\vec{F}_{\gamma_\alpha}=\bigcup_{\xi<^\Xx\gamma_\alpha}\vec{F}_\xi$.\qedhere
\end{enumerate}
\end{dfn}

The kinds of tree embeddings relevant to full normalization are the \emph{minimal} ones,
defined next; the main point of this is captured in the
 \emph{weak hull embeddings}, defined by Steel (see \cite{steel_local_HOD_comp}).
The definition is actually much shorter than the analogous 
definition in \cite{iter_for_stacks}; we will only keep track of embeddings
from $\exit^\Tt_\alpha$ into segments of models of $\Xx$, not from the full models $M^\Tt_\alpha$.
Thus, the definition will not immediately yield that $\Tt$ has wellfounded models.
We will soon see, however, that if $\Pi$ is minimal and $M$ is $m$-standard
where $\Tt$ is $m$-maximal
then there is an embedding
$M^\Tt_\alpha\to M^\Xx_{\gamma_\alpha}$, so $\Tt$ will have wellfounded models.

\begin{dfn}[Minimal tree embedding]\label{dfn:min_tree_emb}
Let $\Pi:(\Tt,\theta)\hookrightarrow_{\pre}\Xx$  be a tree pre-embedding.
We say $\Pi$ is 
\emph{minimal}, denoted 
\[ \Pi:(\Tt,\theta)\hookrightarrow_\minim\Xx,\]
provided writing $\gamma_\alpha=\gamma^\Pi_\alpha$, etc,
for each $\alpha<\theta$
we have:
\begin{enumerate}[ref=(\arabic*)]
 \item\label{item:ult_exit_is_seg} If $\alpha+1<\lh(\Tt)$ then
$\vec{F}_{\delta_\alpha}$ is $(\exit^\Tt_\alpha,0)$-good
and $Q_{\alpha\xi}\ins M^\Xx_{\xi}$,
where $
Q_{\alpha\xi}$ denotes $\Ult_0(\exit^\Tt_\alpha,\vec{F}_{\xi})$,  for
each 
$\xi\in I_\alpha$.\footnote{One might relax the requirement
that $\vec{F}_{\delta_\alpha}$
be $(\exit^\Tt_\alpha,0)$-good, by allowing extenders in $\vec{F}_{\delta_\alpha}$
to measure more sets than those in the model they apply to.
But this leads to complications, and is anyway not relevant to our purposes here.}
 \item if $\alpha+1<\theta$ then 
$E^\Xx_{\delta_\alpha}=F^{Q_{\alpha\delta_\alpha}}$, and
 \item if $\alpha+1=\lh(\Tt)$ then $(\gamma_\alpha,\delta_\alpha]_\Xx$ does not 
drop in model or degree.
\end{enumerate}
If $\Pi$ is a minimal tree embedding,
we say  $\Pi$ is \emph{bounding} iff
 $\lh(E^\Xx_\xi)\leq\OR^{Q_{\alpha\xi}}$ for each $\alpha<\theta$ 
and $\xi\in I_\alpha$ such 
that $\alpha+1<\lh(\Tt)$ and $\xi+1<\lh(\Xx)$,
and \emph{exactly bounding}
iff $\lh(E^\Xx_\xi)<\OR^{Q_{\alpha\xi}}$
for each such $\alpha,\xi$ with $\xi\in[\gamma_\alpha,\delta_\alpha)_\Xx$.

We write $\Pi:\Tt\hookrightarrow_{\minim}\Xx$
iff $\Pi:(\Tt,\lh(\Tt))\tembto_{\minim}\Xx$.
\end{dfn}

We will mostly be interested in
exactly bounding minimal tree embeddings.

\begin{dfn}\label{dfn:min_dagger_tree_emb}
A tree pre-embedding $\Pi:(\Tt,\theta)\hookrightarrow_{\pre}\Xx$ is 
\emph{puta-minimal},
written
$\Pi:(\Tt,\theta)\hookrightarrow_{\putamin}\Xx$,
iff the requirements of minimality hold,
except that 
we replace condition
\ref{dfn:min_tree_emb}\ref{item:ult_exit_is_seg}
with the following, defining $Q_{\alpha\xi}$ as before:
\begin{enumerate}[label=\arabic*'.]
 \item Let $\alpha+1<\lh(\Tt)$. Then:
\begin{enumerate}[label=(\alph*)]
 \item If $\alpha+1<\theta$ then
$\vec{F}_{\delta_\alpha}$ is $(\exit^\Tt_\alpha,0)$-good
and
$Q_{\alpha\xi}\ins M^\Xx_{\xi}$ for each 
$\xi\in I_\alpha$.
\item
If $\alpha+1=\theta$ and
 $\vec{F}_{\gamma_\alpha}$
 is $(\exit^\Tt_\alpha,0)$-good and
 $Q_{\alpha\gamma_\alpha}\ins M^\Xx_{\gamma_\alpha}$ then 
$\vec{F}_{\delta_\alpha}$ is $(\exit^\Tt_\alpha,0)$-pre-good 
and $Q_{\alpha\xi}\ins M^\Xx_{\xi}$
 for each $\xi\in I_\alpha\cut\{\delta_\alpha\}$.\qedhere
\end{enumerate}
\end{enumerate}
\end{dfn}

We will need to define various bookkeeping devices (mice and maps) analogous to 
those used in \cite{iter_for_stacks},
in order to see that minimal tree embeddings make sense, for $m$-standard $M$.
The definition will list a lot of properties of the various objects, and we will verify that they
exist later in Lemma \ref{lem:Pi_is_standard}.

\begin{dfn}[Dropdown lifts]\label{dfn:min_inf_ults_maps}
Let 
$\Pi:(\Tt,\theta)\hookrightarrow_{\putamin}\Xx$
of degree $m$ on an $m$-standard $M$.
Write $I_\alpha=I_\alpha^\Pi$, etc. Let $\Delta=\ddd(\Tt,\theta)$.
For $x=(\beta,i)\in \Delta$ let 
$k_\beta=k^\Tt_\beta$ and\footnote{So $i\leq k_\beta$, $M_{\beta 0}=M_\beta$, $m_{\beta 0}=\deg^\Tt(\beta)$,
and if $\beta+1<\lh(\Tt)$ then $M_{\beta k_\beta}=\exit^\Tt_\beta$
and $m_{\beta k_\beta}=0$.}
\[ (M_{\beta i},m_{\beta i})=(M_x,m_x)=(M_x^\Tt,m_x^\Tt). \]

For $(\beta,i)\in\Delta$ let $\delta_{\beta i}$ be the largest $\delta\in\bigcup_{\alpha<\theta}I_\alpha$
such that $\vec{F}_\delta$ is $(M_{\beta i},m_{\beta i})$-pre-good.
Say that $\Pi$ is \emph{pre-standard} iff
for each $(\beta,i)\in\Delta$,
we have $\delta_{\beta k_\beta}=\delta_\beta$,
$\delta_{\beta i}\in I_\beta$
and if $(\beta,i+1)\in\Delta$ then 
$\delta_{\beta i}\leq\delta_{\beta,i+1}$.

Suppose $\Pi$ is pre-standard. For $\xi\in[\gamma_\beta,\delta_{\beta i}]_\Xx$
define
\begin{enumerate}[label=--]
\item $U_{\beta i\xi}=\Ult_{m_{\beta i}}(M_{\beta i},\vec{F}_\xi)$, and
\item $\pi_{\beta i\xi}:M_{\beta i\xi}\to U_{\beta i\xi}$ is the associated ultrapower map.
\end{enumerate}
Let $U_{\beta i}=U_{\beta i\gamma_\beta}$ and $\pi_{\beta i}=\pi_{\beta 
i\gamma_\beta}$ and $\pi_\beta=\pi_{\beta 0}$. Let $\gamma_{\beta 
0}=\gamma_\beta$ and $\gamma_{\beta,i+1}=\delta_{\beta i}$ for $i+1\leq 
k_\beta$.
Let $I_{\beta i}=[\gamma_{\beta i},\delta_{\beta i}]_\Xx$
and
$J_{\beta i}=[\gamma_\beta,\delta_{\beta 
i}]_\Xx$.
For $\xi\in I_\beta$ let $i_{\beta\xi}=$ the least
$i'$ such that $\xi\in I_{\beta i'}$.
If $\beta\in\lh(\Tt)^-$ then for $\xi\in I_\beta$ let
$Q_{\beta\xi}=U_{\beta k_\beta\xi}$ and
$\om_{\beta\xi}=\pi_{\beta k_\beta\xi}$.
\end{dfn}

\begin{rem}Note that by pre-standardness, for $(\beta,i)\in\Delta$ we have
 have $I_{\beta i}=[\gamma_{\beta i},\delta_{\beta i}]_\Xx\sub I_\beta$.
\end{rem}

\begin{dfn}\label{dfn:min_inf_coherent} (Cf.~\cite[Figures 2, 3, 
4]{iter_for_stacks}.)
Let $M,m,\Tt,\Xx,\Pi,\Delta$ be as in \ref{dfn:min_inf_ults_maps}.
We say $\Pi$ is \emph{standard} iff
\begin{enumerate}[label=T\arabic*.,ref=T\arabic*]
 \item\label{item:pre-standardness}  
$\Pi:(\Tt,\theta)\hookrightarrow_{\min}\Xx$ and $\Pi$ is pre-standard.  
\item\label{item:T_dropdowns_lift} (Dropdowns lift) For $(\alpha,i)\in\Delta$ and $\xi\in J_{\alpha i}$ we have:
\begin{enumerate}
\item\label{item:U_alpha0_is_model} $(U_{\alpha 0},m_{\alpha 0})=(M^\Xx_{\gamma_\alpha},\deg^\Xx(\gamma_\alpha))$,
 \item\label{item:U_alpha,i,xi_is_seg} $(U_{\alpha i\xi},m_{\alpha i\xi})\ins(M^\Xx_\xi,\deg^\Xx(\xi))$,
 \item\label{item:U_alpha,i,xi_is_model} $(U_{\alpha i\xi},m_{\alpha i\xi})=(M^\Xx_\xi,\deg^\Xx(\xi))$ if $\gamma_{\alpha i}<\xi\leq\delta_{\alpha i}$,
 \item\label{item:no_drop_to_delta_alpha,0}  $\dropset^\Xx_{\deg}\inter(\gamma_{\alpha 0},\delta_{\alpha 0}]_\Xx=\emptyset$.
\item\label{item:only_drop_at_eps} Suppose $i>0$ and $\gamma_{\alpha i}<\delta_{\alpha i}$. Let $\eps_{\alpha i}=\successor^\Xx(\gamma_{\alpha i},\delta_{\alpha i})$.
Then:
\begin{enumerate}
 \item $(\gamma_{\alpha i},\delta_{\alpha i}]_\Xx\inter\dropset_{\deg}^\Xx=\{\eps_{\alpha i}\}$,
 \item $(M^{*\Xx}_{\eps_{\alpha i}},\deg^\Xx(\eps_{\alpha i}))=(U_{\alpha i\gamma_{\alpha i}},m_{\alpha i})$.
\end{enumerate}
\item\label{item:dropdown_correspondence} Suppose $\alpha\in\lh(\Tt)^-$.
Then:
\begin{enumerate}
 \item\label{item:dropdown_correspondence_gamma_alpha} $\left<(U_{\alpha j},m_{\alpha j})\right>_{j\leq k_\alpha}$ is
 the revex
$((M^\Xx_{\gamma_\alpha},\deg^\Xx(\gamma_\alpha)),(Q_{\alpha\gamma_\alpha},
0))$-dd.
 \item If $\gamma_{\alpha i}<\xi\leq\delta_{\alpha i}$
 then $\left<(U_{\alpha j\xi},m_{\alpha j})\right>_{i\leq j\leq k_\alpha}$ is
 the revex
 $((M^\Xx_\xi,\deg^\Xx(\xi)),(Q_{\alpha\xi},0))$-dd.
\end{enumerate}
\end{enumerate}
\item\label{item:T_emb_agmt} (Embedding agreement) Let $\alpha<\theta$ with $\alpha+1<\lh(\Tt)$
and  $\kappa<\nutilde(E^\Tt_\alpha)$ 
with $\alpha=\alpha^\Tt_\kappa$ and
$i=n^\Tt_{\kappa}$.
Let $\xi\in J_{\alpha i}$ be least such that,
letting $\mu=\pi_{\alpha i\xi}(\kappa)$,
either $\xi=\delta_{\alpha i}$ or
$\mu<\crit(E^\Xx_\eta)$
where $\eta+1=\successor^\Xx(\xi,\delta_{\alpha i})$; note that
$\xi\geq\gamma_{\alpha i}$. Write $\xi_{\kappa}=\xi$. 
Let $U=U_{\alpha i\xi}$ and $\pi=\pi_{\alpha i\xi}$.

Whenever $(\alpha',i',\xi')\geq(\alpha,i,\xi)$,
 $U'=U_{\alpha'i'\xi'}$ and $\pi'=\pi_{\alpha'i'\xi'}$, we have:
\begin{enumerate}[label=--]
\item $U||(\mu^+)^U=
U'||(\mu^+)^{U'}$,\footnote{We might have $(\mu^+)^{U}=\OR^U$,
but then $i=k^\Tt_\alpha$, $\xi=\delta_\alpha$,
and we are using MS-indexing, $M^\Tt_\alpha$ is active
type 2 and $\kappa=\lgcd(M^\Tt_\alpha)$.}
\item $\pi\rest\pow(\kappa)\sub\pi'$ and
\item if $\alpha<\alpha'$ and $(i,\xi)=(k^\Tt_\alpha,\delta_\alpha)$
then:
\begin{enumerate}[label=--]
\item $\pi\rest\iota^\Tt_\alpha\sub\pi'$ and 
$\widehat{\pi}(\delta)\leq\pi'(\delta)$
for every $\delta<\lh(E^\Tt_\alpha)$,\footnote{$\iota^\Tt_\alpha=\lgcd(\exit^\Tt_\alpha)$
unless  $\exit^\Tt_\alpha$ is MS-indexed type 2, in which case
$\iota^\Tt_\alpha=\lh(E^\Tt_\alpha)$.
And recall that $\widehat{\pi}$ is the map induced
by $\pi$ via the Shift Lemma.}
\item if 
$\lh(E^\Tt_\alpha)<\OR(M_{\alpha'i'})$ then
$\lh(E^\Xx_{\delta_\alpha})\leq\widehat{\pi'}(\lh(E^\Tt_\alpha))$,
\item if 
$\lh(E^\Tt_\alpha)=\OR(M_{\alpha' i'})$ then $\alpha'=\alpha+1$, $i'=0$,
$\lh(E^\Xx_{\delta_\alpha})=\OR(M^\Xx_{\gamma_{\alpha+1}})$
\footnote{So if $\Pi$ is exactly bounding then
$\gamma_{\alpha+1}=\delta_{\alpha+1}$.},
$\pi_{\alpha+1,0}
=\widehat{\pi}\rest M_{\alpha+1,0}$,
and $M^{*\Xx}_{\gamma_{\alpha+1}}=Q_{\eps\delta_\eps}$
where $\eps=\pred^\Tt(\alpha+1)$.\footnote{It follows that 
we are using MS-indexing, 
$E^\Tt_\alpha$ is superstrong and $M^\Tt_{\alpha+1}$ is active type 2, $E^\Xx_{\delta_\alpha}$ is 
superstrong and $M^\Xx_{\gamma_{\alpha+1}}$ is active type 2.}
\end{enumerate}
\end{enumerate}

\item\label{item:T_commutativity} (Commutativity)
Let $(\chi,i),(\beta+1,0),(\eps,0)\in \Delta$ with
$\chi<_\Tt\beta+1\leq_\Tt\eps$ and
$\chi=\pred^\Tt(\beta+1)$ and
\[ (M_{\chi i},m_{\chi i})=(M^{*\Tt}_{\beta+1},\deg^\Tt(\beta+1)).\] 
(So $i>0$ iff $\beta+1\in\dropset_{\deg}^\Tt$.)
Let $\xi=\pred^\Xx(\gamma_{\beta+1})$.
Then:
\begin{enumerate}[label=\tu{(}\alph*\tu{)}]
\item\label{item:xi_in_I_alpha,i} $\xi\in I_{\chi i}$,
and if $i>0$ then [$\xi=\gamma_{\chi i}$ iff $\gamma_{\beta+1}\in 
\dr_{\deg}^\Xx$].
\item\label{item:emb_comm_at_beta+1}
$\pi_{\beta+1, 0}\com i^{*\Tt}_{\beta+1}=i^{*\Xx}_{\gamma_{\beta+1}}
\com\pi_{\chi i\xi}$.
\item\label{item:emb_comm_to_eps} If $(\chi,\eps]_\Tt\inter 
\dr_{\deg}^\Tt=\emptyset$ (so $i=0$ and 
$(\gamma_\chi,\gamma_\eps]_\Xx\inter\dr_{\deg}^\Xx=\emptyset$) then
\[ \pi_{\eps 0}\com 
i^{\Tt}_{\chi\eps}=i^\Xx_{\gamma_\chi\gamma_\eps}\com\pi_{\chi 0}. \]
\item\label{item:T_shift_lemma} (Shift Lemma)
Let $\kappa=\crit(E^\Tt_\beta)$,
so $\kappa<\nutilde(E^\Tt_\chi)$ and $i=n^\Tt_{\kappa}$, so 
\ref{item:T_emb_agmt} applies. Then (i) $\xi$ (defined above)
is also as defined in \ref{item:T_emb_agmt} and
\[ (M^{*\Xx}_{\gamma_{\beta+1}},\deg^\Xx_{\gamma_{\beta+1}})
=(U_{\chi i\xi},m_{\chi i}).\]
So by  \ref{item:T_emb_agmt},
the Shift Lemma applies to the embeddings 
$\pi_{\chi i\xi}$
and
\[ \omega_{\beta\delta_\beta}:\exit^\Tt_\beta\to Q_{\beta\delta_\beta}.\]
Moreover, (ii)
$\pi_{\beta+1,0}$ is just the map given by the Shift Lemma
(this makes sense as $M^\Xx_{\gamma_{\beta+1}}=U_{\beta+1,0}$ by  
\ref{item:T_dropdowns_lift}).\qedhere
\end{enumerate}
\end{enumerate}
\end{dfn}

\begin{lem}\label{lem:Pi_is_standard}
$\Pi:(\Tt,\theta)\hookrightarrow_{\putamin}\Xx$ have degree $m$,
on an $m$-standard $M$.
Then $\Pi:(\Tt,\theta)\hookrightarrow_{\min}\Xx$, $\Pi$ is standard
and $\Tt\rest\theta$ has (well-defined and) wellfounded models.
\end{lem}
\begin{proof}
We adopt the notation of \ref{dfn:min_inf_ults_maps}.
The proof is by induction on $\theta$.

Suppose $\theta=1$. If $\lh(\Tt)=1$ then everything is easy.  So suppose $\lh(\Tt)>1$,
so $E^\Tt_0$ exists. Property \ref{item:T_commutativity} 
is trivial. By puta-minimality,
$\vec{F}_{\delta_0}$ is $(\exit^\Tt_0,0)$-pre-good and
$Q_{0\xi}\ins M^\Xx_\xi$ for every $\xi<^\Xx\delta_0$.
Now $\delta_{00}$ is the largest
$\xi\in I_0$ such that $\vec{F}_\xi$ is $(M,m)$-pre-good.
Property \ref{item:T_dropdowns_lift} for $(\alpha,i)=(0,0)$,
and the fact that $\delta_{00}\leq\delta_{0i}$ for each $i\leq k^\Tt_0$, then 
both follow  from Lemma \ref{lem:ult_dropdown},
by induction on $\xi\in I_{00}$. It doesn't matter here whether
$\Pi$ is bounding. In fact, note that
if $\beta+1\leq^\Xx\delta_{0}$
and $\kappa=\crit(E^\Xx_\beta)$ and $\left<(N'_n,m'_n)\right>_{n\leq 
k'}$ is the revex
$((M^\Xx_{\xi},\deg^\Xx_\xi),(\exit^\Xx_\xi,0))$-dd
and $\left<(N_n,m_n)\right>_{n\leq k}$
 the revex $((M^\Xx_\xi,\deg^\Xx_\xi),(Q_{0\xi},0))$-dd,
then $(N_n,m_n)=(N'_n,m'_n)$ whenever $n=0$ or
[$n\leq k$ and $\rho_{m_n+1}^{N_n}\leq\kappa$] or
[$n\leq k'$ and 
$\rho_{m'_n+1}^{N'_n}\leq\kappa$].
Hence, if  $\delta_{00}<\delta_0$
and $\beta+1=\successor^\Xx(\delta_{00},\delta_0)$,
then $(M^{*\Xx}_{\beta+1},\deg^\Xx_{\beta+1})=(N_n,m_n)$
for some $n>0$, which is
one of the $(U_{0i\xi},m_{0i\xi})$.
This gives that $\delta_{00}=\gamma_{0 i'}=\delta_{0i'}=\gamma_{0i}<\delta_{0i}$
and property \ref{item:T_dropdowns_lift}
for $(\alpha,i')$ for all $i'<i$. For $(\alpha,i)$,
property \ref{item:T_dropdowns_lift}
and that $\delta_{0i}\leq\delta_{0i'}$ for all $i'>i$
now follows similarly to before.  Preceding in this way,
we get the full properties \ref{item:pre-standardness}
and \ref{item:T_dropdowns_lift}
(recalling $\theta=1$).
Property \ref{item:T_emb_agmt} is now straightforward.

Now suppose that
$\Pi\rest(\beta+1):(\Tt,\beta+1)\hookrightarrow_{\minim}\Xx$ is standard
and $\beta+1<\lh(\Tt)$.
We prove  $\Pi\rest(\beta+2):(\Tt,\beta+2)\hookrightarrow_{\minim}$
and is 
standard.

Property \ref{item:T_commutativity}: We must just verify
this for $\beta+1$, with $\eps=\beta+1$.
Adopt notation as there (this defines $\xi,i,\kappa$ etc).
Parts \ref{item:xi_in_I_alpha,i} and \ref{item:T_shift_lemma}(i) follow 
routinely from the inductive hypotheses.
Given these, we verify the rest.
We have $\chi=\pred^\Tt(\beta+1)$
and $(M^{*\Tt}_{\beta+1},\deg^\Tt_{\beta+1})=(M_{\chi i},m_{\chi i})$.
Note that Lemma \ref{lem:extender_comm} applies to 
$(M_{\chi i},m_{\chi i})$ and $P=\exit^\Tt_\beta$,
with extender sequences $\vec{E}=\vec{F}_\xi$
and $\vec{F}$ where $\vec{F}_{\delta_\beta}=\vec{F}_\xi\conc\vec{F}$. Moreover,
\begin{enumerate}[label=--]
 \item $\chi=\pred^\Tt(\beta+1)$ and $\xi=\pred^\Xx(\gamma_{\beta+1})$,
 \item  $n\eqdef\deg^\Tt(\beta+1)=m_{\beta+1,0}=m_{\chi 
i}=\deg^\Xx(\gamma_{\beta+1})$,
 \item $M_{\chi i}=M^{*\Tt}_{\beta+1}$ and $U_{\chi 
i\xi}=M^{*\Xx}_{\gamma_{\beta+1}}$,
 \item $M^\Tt_{\beta+1}=\Ult_m(M_{\chi i},E^\Tt_\beta)$ and 
$M^\Xx_{\gamma_{\beta+1}}=\Ult_m(U_{\chi i\xi},E^\Xx_{\delta_\beta})$,
 \item  
$\vec{F}_{\gamma_{\beta+1}}=\vec{F}_{\delta_\beta}=
\vec{E}\conc\vec{F}$
 where $\vec{E}$ is $\kappa$-bounded and 
\[ \mu\eqdef i^{\exit^\Tt_\beta,0}_{\vec{E}}(\kappa)=
i^{\exit^\Tt_\beta,0}_{\vec{F}_{\delta_\beta}}(\kappa)=
\crit(E^\Xx_{\delta_\beta})<\crit(\vec{F}), \]
 \item $U_{\chi i\xi}=\Ult_m(M_{\chi i},\vec{E})$
 and $\pi_{\chi i\xi}$ 
is the ultrapower map,
 \item $Q_{\beta\delta_\beta}=\Ult_0(\exit^\Tt_\beta,\vec{F}_{\delta_\beta})$ and $\om_{\beta\delta_\beta}$ is the ultrapower map,
 \item $U_{\beta+1,0}=\Ult_{m}(M^\Tt_{\beta+1},\vec{F}_{\gamma_{\beta+1}})$ and
 $\pi_{\beta+1,0}$ is the ultrapower map.
\end{enumerate}
So by Lemma \ref{lem:extender_comm}, $U_{\beta+1,0}=M^\Xx_{\gamma_{\beta+1}}$
and $\pi_{\beta+1,0}$ is the Shift Lemma map, giving part 
\ref{item:T_shift_lemma}(ii),
and commutativity holds, giving  \ref{item:emb_comm_at_beta+1} and \ref{item:emb_comm_to_eps}.

Property  \ref{item:T_dropdowns_lift} for 
$\alpha=\beta+1$ and $\xi=\gamma_{\beta+1}$,
and the fact that $\gamma_{\beta+1}\leq\delta_{\beta i}$
for each $i\leq k^\Tt_{\beta+1}$,
follow from the observations above (such as that
$U_{\beta+1,0}=M^\Xx_{\gamma_{\beta+1}}$
and $m_{\beta+1,0}=\deg^\Xx(\gamma_{\beta+1})$),
together with
the $m$-standardness of $M$
and Lemma \ref{lem:ult_dropdown},
and using that $\lh(E^\Tt_\beta)\leq\lh(E^\Tt_{\beta+1})$
when verifying that $\vec{F}_{\gamma_{\beta+1}}=\vec{F}_{\delta_\beta}$
is $(\exit^\Tt_{\beta+1},0)$-good, for example.
The rest of properties
\ref{item:pre-standardness} and
\ref{item:T_dropdowns_lift} for $\alpha=\beta+1$
are then
like in the case that $\theta=1$.

Property \ref{item:T_emb_agmt}: In the main instance of interest, $\alpha=\beta$ and $\alpha'=\beta+1$.
In this instance, the property follows as usual from the Shift Lemma,
using the fact that $\pi_{\beta+1,0}$ is in fact the Shift Lemma map. The rest is routine.

This completes the proof that $\Pi\rest(\beta+2)$ is standard.

Now let $\beta$ be a limit and suppose that $\Pi\rest\beta$ is standard.
We verify that $\Pi\rest(\beta+1)$ is standard.
The main issue is to see that $\vec{F}_{\gamma_\beta}$ is 
$(M^\Tt_\beta,m_{\beta 0})$-good,
$M^\Xx_{\gamma_\beta}=U_{\beta 0}$
and
$\pi_{\beta 0}\com 
i^\Tt_{\alpha\beta}=i^\Xx_{\gamma_\alpha\gamma_\beta}\com\pi_\alpha$
for sufficiently large $\alpha<_\Tt\beta$; the rest is as before.
Let $\pi^*:M^\Tt_\beta\to M^\Xx_{\gamma_\beta}$ be the map commuting 
in this way (by induction,
$\pi^*$ exists and is a $\deg^\Tt_\beta=\deg^\Xx_{\gamma_\beta}$-embedding).
Let $\delta=\delta(\Tt\rest\beta)$ and $\delta'=\delta(\Xx\rest\gamma_\beta)$.
By commutativity,
$\vec{F}_{\gamma_\beta}$ is equivalent to the
$(\delta,\delta')$-extender 
derived from $\pi^*$ (and 
$\delta'=\sup\pi^*``\delta\leq\pi^*(\delta)$).
It easily follows that $U_{\beta 0}=M^\Xx_{\gamma_\beta}$
and $\pi^*=\pi_{\beta+1,0}$, as desired.
\end{proof}

\begin{rem}\label{rem:pres_dropdown}
The basic observation which made the lemma above possible --
the fact that extenders in $\es_+(M^\Tt_\alpha)$
lift to extenders in $\es_+(M^\Xx_\beta)$
for the appropriate $\alpha,\beta$, under degree $0$ ultrapower maps --
and the ensuing idea for full normalization (as opposed to embedding 
normalization), weak hull embeddings and minimal hull condensation -- was due to Steel.
The fact that one must also keep track of how the dropdown
sequence is shifted all along the entire interval $I_\alpha$
(in order to see that the $\Tt$ and $\Xx$ drop to corresponding
segments) was noticed somewhat later,
independently by both the author and Steel.
\end{rem}

\begin{dfn}\label{dfn:Pi_subscript_notation}
Let $\Pi$ be a minimal tree embedding and adopt notation as before. We use 
notation analogous to that of \cite{iter_for_stacks}; the subscript ``$\Pi$'' 
indicates objects  associated to $\Pi$.
That is, $I_{\Pi\alpha}=I_\alpha$, $\gamma_{\Pi\alpha}=\gamma_\alpha$, etc. 
\end{dfn}

\begin{dfn} \label{dfn:pi_kappa}
(Cf.~\cite[Figure 4]{iter_for_stacks}.)
 Let
$\Pi:(\Tt,\theta)\tembto_{\minim}\Xx$
and $\gamma_\beta=\gamma_{\Pi\beta}$, 
etc. 
Let $\beta<\theta$
and $\kappa\leq\OR(M^\Tt_\beta)$ with
$\beta=\alpha^\Tt_\kappa$,
and let $n=n^\Tt_\kappa$ (Definition \ref{dfn:tree_dropdown}).
Let  $N^\Tt_{\kappa}=M^\Tt_{\beta n}$ and
$\xi=\xi_{\Pi\kappa}\in I_{\beta n}$ be defined
as $\xi_\kappa$ in
\ref{dfn:min_inf_coherent}(\ref{item:T_emb_agmt}),
or if $\lh(\Tt)=\beta+1$ and $\kappa=\OR(M^\Tt_\beta)$,
then $\xi=\xi_{\Pi\kappa}=\delta_\beta$.
Also let $U_{\Pi\kappa}=U_{\beta n\xi}$ and
$\pi_{\Pi\kappa}:N^\Tt_{\kappa}\to U_{\Pi\kappa}$
the corresponding ultrapower map.
\end{dfn}

\begin{dfn}
Let $\Pi:(\Tt,\theta)\tembto_{\minim}\Xx$. Let $\beta\in\theta\inter\lh(\Tt)^-$ 
and $\xi\in I_{\Pi\beta}$. Then $E_{\Pi\xi}$ denotes
$F^{Q_{\beta\xi}}$ 
(the lift of $E^\Tt_\beta$ in $\es_+(M^\Xx_\xi)$).
\end{dfn}

\begin{dfn}\label{dfn:trivial_tree_embedding} Given $\Tt,\Xx$ two putative 
$m$-maximal trees, $\Xx$  a true tree,
the \emph{trivial}  tree 
embedding
$\Pi:(\Tt,1)\tembto_{\minim}\Xx$ is that with $I_{\Pi 0}=[0,0]$.

If $\Tt=\Xx$, the \emph{identity} embedding
$\Pi:(\Tt,\lh(\Tt))\hookrightarrow(\Tt,\lh(\Tt))$ is that
with $I_{\Pi\beta}=[\beta,\beta]$ for all $\beta<\lh(\Tt)$.
\end{dfn}

\begin{lem}\label{lem:Pi_min_copy_normal}Let $\Pi:(\Tt,\alpha+1)\tembto_{\minim}\Xx$
where $\alpha+1<\lh(\Tt)$. Then $E_{\Pi\delta_{\Pi\alpha}}$ is 
$\Xx\rest(\delta_{\Pi\alpha}+1)$-normal.\end{lem}

The lemma follows easily from the fact that
$\vec{F}_{\delta_{\Pi\alpha}}$ is $(\exit^\Tt_\alpha,0)$-good.
(It is important here that in particular,
$F_\xi$ does not measure \emph{more} subsets of its critical point
than those in $\Ult_0(\exit^\Tt_\alpha,\vec{F}_\xi)$.)

We can propagate minimal tree embeddings 
$\Tt\hookrightarrow_{\minim}\Xx$ via ultrapowers 
analogously to in \cite[4.23, 4.24]{iter_for_stacks},
so there are two possibilities: an extender is either \emph{$\Tt$-copying} or \emph{$\Tt$-inflationary}.
We first consider the $\Tt$-copying case.

\begin{dfn}\label{dfn:one-step}
 Let $\Pi:(\Tt,\alpha+1)\tembto_{\minim}\Xx$ be of degree $m$, with $\alpha+1<\lh(\Tt)$.
 Let $\gamma_\alpha=\gamma_{\Pi\alpha}$ etc.
Let
\[ \Xx'=\text{ the }m\text{-maximal tree }\Xx\rest(\delta_\alpha+1)\conc\left<E^\Pi_{\delta_\alpha}\right>\]
(by \ref{lem:Pi_min_copy_normal},  $E^\Pi_{\delta_\alpha}$ is $\Xx\rest(\delta_\alpha+1)$-normal).
Suppose that $M^{\Xx'}_{\delta_\alpha+1}$ is wellfounded.
Let
\[ \Pi':(\Tt,\alpha+2)\hookrightarrow_{\pre}\Xx \]
be such that $\Pi'\rest(\alpha+1)=\Pi$ and 
$I_{\alpha+1}^{\Pi'}=[\delta_\alpha+1,\delta_\alpha+1]_{\Xx'}$.
We say  $(\Xx',\Pi')$ (or just $\Pi'$ for short) is the 
\emph{one-step copy extension} of $(\Xx,\Pi)$ (or of $\Pi$).
\end{dfn}
\begin{lem}\label{lem:one-step_copy_extension_is_standard}
 Adopt the hypotheses of \ref{dfn:one-step}.
 Suppose $\Tt,\Xx$ are on an $m$-standard
 $M$.
 Then $\Pi'$ is standard, so $\Tt\rest(\alpha+2)$ has wellfounded 
models.
\end{lem}
\begin{proof}
$\Pi'$ is puta-minimal as $\gamma_{\alpha+1}=\delta_{\alpha+1}$,
so is minimal and standard by Lemma \ref{lem:Pi_is_standard}.
\end{proof}

We next consider the $\Tt$-inflationary case.

\begin{dfn}\label{dfn:inflationary_extender}
Let $\Pi:(\Tt,\theta)\tembto_{\minim}\Xx$, of degree $m$,
with $\lh(\Xx)=\eta+1$.
Let $\gamma_\alpha=\gamma_{\Pi\alpha}$, etc.
Let $E\in\es_+(M^\Xx_\eta)$ be $\Xx$-normal and $\Xx'$ be the putative 
$m$-maximal tree $\Xx\conc\left<E\right>$.
Let $\xi=\pred^{\Xx'}(\eta+1)$. 
Suppose that:
\begin{enumerate}[label=--]
\item $M^{\Xx'}_\infty$ is wellfounded,
\item $\xi\in I_\beta$ for some $\beta<\theta$,
 \item if $\beta+1<\lh(\Tt)$ then $E$ is a $Q_{\Pi\beta\xi}$-extender
and $\crit(E)<\nutilde(Q_{\Pi\beta\xi})$, and
 \item if $\beta+1=\lh(\Tt)$ then
 $\eta+1\notin\dropset_{\deg}^{\Xx'}$.
\end{enumerate}

The \emph{minimal $E$-inflation of $(\Xx,\Pi)$} is
$(\Xx',\Pi')$, where
$\Pi':(\Tt,\beta+1)\tembto_{\pre}\Xx'$ is such that 
$I_{\Pi'\beta}=(I_\beta\inter(\xi+1))\cup\{\eta+1\}$ and
$I_{\Pi'\alpha}=I_\alpha$ for every 
$\alpha<\beta$.
\end{dfn}

\begin{lem}\label{lem:E-inflation_is_minimal}
   Adopt the hypotheses of \ref{dfn:inflationary_extender}.
   Suppose $\Tt,\Xx$ are on an $m$-standard $M$.
   Then $\Pi'$ is minimal and standard.
\end{lem}

The lemma is a direct consequence of the definitions
and Lemma \ref{lem:Pi_is_standard}.

\subsection{Minimal inflation}

We now proceed to the definition of a \emph{minimal inflation} of a normal 
iteration tree $\Tt$.
This is almost the exact definition of \emph{inflation} from \cite{iter_for_stacks};
the only differences are that here we are not considering wcpms (coarse structures),
and we use the minimal one-step copy extension and minimal-$E$-inflation
at successor steps, instead of the non-minimal versions.
But we will write out the definition explicitly, for convenience.
An intuitive introduction can be seen in \cite[\S4.3]{iter_for_stacks}.
We will use notation much like there.

\begin{dfn}[Minimal inflation]\label{dfn:min_inflation}
Let $M$ be $m$-standard
and $\Tt,\Xx$ be putative $m$-maximal trees on $M$,
$\Xx$ a true tree.
We say that $\Xx$ is a \emph{minimal inflation} of $\Tt$,
written $\Tt\mininflatearrow\Xx$, iff there is
$\left(t,C,C^-,f,\left<\Pi_\alpha\right>_{\alpha\in C}\right)$
with the following properties (which unique the tuple); we will also define further 
notation:
\begin{enumerate}[label=\arabic*.,ref=\arabic*]
\item 
\index{type}\index{$\Tt$-copying}\index{copying}\index{$\Tt$-inflationary}
\index{inflationary}We have $t:\lh(\Xx)^-\to\{0,1\}$. The value of $t(\alpha)$ 
indicates the 
\emph{type} of 
$E^\Xx_\alpha$, either \emph{$\Tt$-copying} (if $t(\alpha)=0$) or 
\emph{$\Tt$-inflationary} (if 
$t(\alpha)=1$).
\item\label{item:C_nature} $C\sub\lh(\Xx)$
and $C\inter[0,\alpha]_\Xx$ is a closed\footnote{One could consider dropping the closure 
requirement here; cf.~\cite[Footnote q/17]{iter_for_stacks}.}
initial segment of $[0,\alpha]_\Xx$.
\item\label{item:def_C^-} We have $f:C\to\lh(\Tt)$ and
$C^-=\{\alpha\in C\bigm|   f(\alpha)+1<\lh(\Tt)\}$.
\item For $\alpha\in C$ we have 
$\Pi_\alpha:(\Tt, f(\alpha)+1)\tembto_{\minim}\Xx\rest(\alpha+1)$, with 
$\delta_{\alpha; f(\alpha)}=\alpha$, where 
we write $\delta_{\alpha;\beta}=\delta_{\Pi_\alpha\beta}$, etc.
\index{$\gamma_{\alpha;\beta}$, $\delta_{\alpha;\beta}$, etc}
\item $0\in C$ and $f(0)=0$ and
$\Pi_0:(\Tt,1)\tembto_{\minim}\Xx\rest 1$
is trivial (see \ref{dfn:trivial_tree_embedding}).
\item Let $\alpha+1<\lh(\Xx)$. Then:\footnote{One might
also consider weakening these conditions;
cf.~\cite[Footnotes r/18, s/19]{iter_for_stacks}.}
\begin{enumerate}[label=--]
 \item If $\alpha\in C^-$ then 
$\lh(E^\Xx_\alpha)\leq\lh(E_{\Pi_\alpha\alpha})$.
\item $t(\alpha)=0$ iff [$\alpha\in C^-$ and
$E^\Xx_\alpha=E^{\Pi_\alpha}_\alpha$].
\end{enumerate}
\item Let $\alpha+1<\lh(\Xx)$ be such that $t(\alpha)=0$. Then we interpret 
$E^\Xx_\alpha=E_{\Pi_\alpha\alpha}$ as a copy from $\Tt$, as follows:
\begin{enumerate}[label=--]
 \item  $\alpha+1\in C$ and $f(\alpha+1)= f(\alpha)+1$.
 \item $(\Xx\rest\alpha+2,\Pi_{\alpha+1})$
is the minimal one-step copy  extension of 
$(\Xx\rest\alpha+1,\Pi_\alpha)$.
\end{enumerate}
\item Let $\alpha+1<\lh(\Xx)$ be such that $t(\alpha)=1$.
Then we interpret $E^\Xx_\alpha$ as $\Tt$-inflationary, as follows.
Let $\xi=\pred^\Xx(\alpha+1)$. Then:
\begin{enumerate}[label=--]
\item $\alpha+1\in C$ iff:
\begin{enumerate}[label=--]
 \item $\xi\in C^-$ and $Q_{\xi;f(\xi)}\ins M^{*\Xx}_{\alpha+1}$, or
\item $\xi\in C\cut C^-$ and
$\alpha+1\notin\dropset_{\deg}^\Xx$.
\end{enumerate}
\item If $\alpha+1\in C$ then:
\begin{enumerate}[label=--]
\item $f(\alpha+1)=f(\xi)$.
\item $(\Xx\rest\alpha+2,\Pi_{\alpha+1})$
is the minimal $E^\Xx_\alpha$-inflation of
$(\Xx\rest\alpha+1,\Pi_\xi)$.
\end{enumerate}
\end{enumerate}
\item\label{item:inflation_internal_agreement} Let $\alpha\in C$ and $\beta\in 
I_{\alpha;\gamma}$ for some $\gamma\leq f(\alpha)$. Then:
\begin{enumerate}[label=--]
\item $\beta\in C$ and $f(\beta)=\gamma$.
\item  $I_{\alpha;\eps}=I_{\beta;\eps}$ for all $\eps<f(\beta)=\gamma$,
\item $I_{\beta;f(\beta)}=I_{\alpha;f(\beta)}\inter(\beta+1)$.
\end{enumerate}
\item\label{item:inflation_limit_ordinal} If $\alpha\in C$ is a 
limit\footnote{By condition \ref{item:C_nature},
if $\alpha$ is a limit then $\alpha\in C$ iff $[0,\alpha)_\Xx\sub C$.} then 
$ f(\alpha)=\sup_{\beta<^\Xx\alpha}f(\beta)$.\qedhere
\end{enumerate}

\end{dfn}

\begin{rem}
Note that in the definition of minimal inflation,
we assume that $M^\Tt_0$ is $m$-standard,
where $\Tt,\Xx$ are $m$-maximal.

Adopt the hypotheses and notation
of condition \ref{item:inflation_internal_agreement} above.
Note that
\[ U_{\alpha;f(\beta)0}=M^\Xx_{\gamma_{\alpha;f(\beta)}}
=M^\Xx_{\gamma_{\beta;f(\beta)}}=U_{\beta;f(\beta)0}\]
and
$\pi_{\alpha;f(\beta)0}=\pi_{\beta;f(\beta)0}$.
By \ref{lem:intervals_I_cover_X-branches}\ref{item:<^X_down_close_I_xi},
if $\widetilde{\beta}\leq_\Xx\alpha$
then $\widetilde{\beta}\in I_{\alpha;\widetilde{\gamma}}$
for some $\widetilde{\gamma}\leq^\Tt f(\alpha)$, so condition 
\ref{item:inflation_internal_agreement}
applies to $\widetilde{\beta},\widetilde{\gamma}$,
and therefore $f(\widetilde{\beta})\leq_\Tt f(\alpha)$.

Let $\alpha\in C$ be a limit.
As in \cite{iter_for_stacks}, 
$\Pi_\alpha$
 is determined by $\left<\Pi_\beta\right>_{\beta<\alpha}$ and $\Tt$ and 
$\Xx\rest(\alpha+1)$.
For suppose
$f(\alpha)>f(\beta)$ for all $\beta<_\Xx\alpha$.
From condition \ref{item:inflation_internal_agreement}, for $\xi<f(\alpha)$, it 
follows that
\[ I_{\alpha;\xi}=(\lim_{\beta<_\Xx\alpha}I_{\beta;\xi})=\text{unique value of }I_{\beta;\xi}\text{ for sufficiently large }\beta<_\Xx\alpha.\]
So
$\alpha=\left(\lim_{\xi<f(\alpha)}\gamma_{\alpha;\xi}\right)=
\gamma_{\alpha;f(\alpha)}=
\delta_{\alpha;f(\alpha)}$,
so $I_{\alpha;f(\alpha)}=[\alpha,\alpha]$,
determining $\pi_{\alpha;f(\alpha)0}$, 
etc. Suppose now 
$f(\alpha)=f(\beta)$ for some $\beta<_\Xx\alpha$.
For such $\beta$ we have $\gamma_{\alpha;f(\alpha)}=\gamma_{\beta;f(\alpha)}$,
and $\delta_{\alpha;f(\alpha)}=\alpha$. This determines
the remaining objects ($I_{\alpha;f(\alpha)}$, etc); they are just the
natural direct limits.
\end{rem}

Using the preceding remark, 
the reader will verify the following uniqueness:

\begin{lem}
If $\Tt\mininflatearrow\Xx$, and
  $w=(t,C,C^-,f,\Pivec)$ and  $w'=(t',C',(C^-)',f',\Pivec')$
both witness this fact,
then $w=w'$.
\end{lem}

\begin{dfn}
 If $\Tt\mininflatearrow\Xx$, then write
$(t,C,C^-,f,\Pivec)^{\Tt\mininflatearrow\Xx}$
for the unique witness $w$. For $\alpha\in C^-$  write
$E^{\Tt\mininflatearrow\Xx}_\alpha\eqdef E_{\Pi_\alpha\alpha}$.
\end{dfn}

As in \cite{iter_for_stacks}, we may freely extend inflations at successor stages, given wellfoundedness:

\begin{lem}
Let $\Tt\mininflatearrow\Xx$, of degree $m$ \tu{(}so $M^\Tt_0$ is $m$-standard\tu{)},
with $\Xx$ of successor length $\beta+1$.
Let $C^-=(C^-)^{\Tt\mininflatearrow\Xx}$. Then:
\begin{enumerate}[label=\arabic*.,ref=\arabic*]
\item\label{item:copy_is_Xx-normal} If $\beta\in C^-$ then $E^{\Tt\mininflatearrow\Xx}_\beta$ is $\Xx$-normal.

\item\label{item:extend_inflation} Let $E\in\es_+(M^\Xx_\beta)$ be $\Xx$-normal, 
with $\lh(E)\leq\lh(E^{\Tt\mininflatearrow\Xx}_\beta)$ if $\beta\in C^-$.
Let $\Xx'$ be the putative $m$-maximal tree $\Xx\conc\left<E\right>$, and suppose $M^{\Xx'}_\infty$
is wellfounded.
Then $\Xx'$ is a minimal inflation of $\Tt$.
\end{enumerate}
\end{lem}
\begin{proof}
 Part \ref{item:copy_is_Xx-normal} follows from \ref{lem:Pi_min_copy_normal},
 and part \ref{item:extend_inflation} from \ref{lem:one-step_copy_extension_is_standard}
and \ref{lem:E-inflation_is_minimal}.
\end{proof}

However, just as in \cite{iter_for_stacks},
at limit stages we need to assume some condensation holds of $\Sigma$,
in order to extend. See \cite[\S4.4]{iter_for_stacks}
for some introduction.

\subsection{Strategy condensation}

\begin{dfn}\index{inflation condensation, $\lambda$-indexed}\label{dfn:minimal_inflation_condensation}
Let $\Omega>\om$ be regular.
Let $\Sigma$ be an $(m,\Omega+1)$-strategy for an $m$-standard pm $M$.
Then $\Sigma$ has \dfnemph{minimal inflation condensation} (\dfnemph{mic}) or is \dfnemph{minimal-inflationary}
iff for all trees $\Tt,\Xx$, if
\begin{enumerate}[label=--]
 \item $\Tt,\Xx$ are via $\Sigma$,
 \item $\Xx$ is a minimal inflation of $\Tt$, as witnessed by $(f,C,\ldots)$,
 \item $\Xx$ has limit length and $\lh(\Xx)\leq\Omega$,
 \item $b\eqdef\Sigma(\Xx)\sub C$ and $f``b$ has limit ordertype,
\end{enumerate}
then letting  $\eta=\sup f``b$, we have 
$f``b=\Sigma(\Tt\rest\eta)$.
\end{dfn}

Like in \cite{iter_for_stacks},
the definition immediately gives that minimal inflations via minimal-inflationary $\Sigma$ can be continued at limit stages:

\begin{lem}\label{lem:inflationary_limit_continuation}
Let $\Omega>\om$ be regular. Let $\Sigma$ be a minimal-inflationary 
$(m,\Omega+1)$-strategy for an $m$-standard $M$.
Let $\Tt,\Xx$ be such that $\Xx$ is via $\Sigma$, $\Xx$ is a minimal inflation of $\Tt$,
as witnessed by $(f,C,\ldots)$, and
$\lh(\Tt)=\sup_{\alpha\in C}(f(\alpha)+1)$.
Then $\Tt$ is via $\Sigma$.

Suppose also that $\Xx$ has limit length $\lambda$ and
 let $\Xx'=(\Xx,\Sigma(\Xx))$.
Then there is $\Tt'$ via $\Sigma$ such that $\Tt\ins\Tt'$ and $\Xx'$ is an inflation of $\Tt'$, as witnessed by $(C',f',\ldots)$.
 Moreover,  we may take $\Tt'$ such that either:
 \begin{enumerate}[label=--]
 \item $\Tt'=\Tt$ and if $\lambda\in C'$ then $f'(\lambda)<\lh(\Tt)$, or
 \item $\Tt$ has limit length $\bar{\lambda}$,  
  $\Tt'=(\Tt,\Sigma(\Tt))$, $\lambda\in C'$, $f'(\lambda)=\bar{\lambda}$
 and $\gamma'_{\lambda;\bar{\lambda}}=\lambda$.
 \end{enumerate}
 Further, the choice of $\Tt'$ is uniqued by adding these requirements.
\end{lem}

Also as in \cite{iter_for_stacks}, we have a simple characterization
of when $\Tt\mininflatearrow\Xx$, given that $\Tt,\Xx$ are via a common minimal-inflationary strategy:
\begin{lem}
Let $\Omega>\om$ be regular. Let $\Sigma$ be a minimal-inflationary 
$(m,\Omega+1)$-strategy for an $m$-standard $M$
and $\Tt,\Xx$ be via $\Sigma$.

Then \tu{(}i\tu{)} 
$\Tt\mininflatearrow\Xx$ iff:
\begin{enumerate}[label=--]
 \item  $\Xx$ satisfies the bounding requirements on extender indices imposed by $\Tt$;
that is, for each $\alpha+1<\lh(\Xx)$, if
\[ \Tt\mininflatearrow\Xx\rest(\alpha+1)\text{ and }\alpha\in(C^-)^{\Tt\mininflatearrow\Xx\rest(\alpha+1)}\]
then
$\lh(E^\Xx_\alpha)\leq\lh(E^{\Tt\mininflatearrow(\Xx\rest\alpha+1)}_\alpha)$,
and
\item if $\Tt$ has limit length then $\Xx$ does not determine a $\Tt$-cofinal branch; that is,
for each limit $\eta<\lh(\Xx)$, if
\[ \Tt\mininflatearrow\Xx\rest\eta\text{ and }
(f,C)=(f,C)^{\Tt\mininflatearrow\Xx\rest\eta}\text{ and }[0,\eta)_\Xx\sub C \]
then $\lh(\Tt)>\sup_{\alpha<_\Xx\eta}f(\alpha)$.
\end{enumerate}

Moreover, \tu{(}ii\tu{)} suppose $\Tt\mininflatearrow\Xx$ and
$\lh(\Xx)$ is a limit. Let $\Xx'=(\Xx,\Sigma(\Xx))$.
Then either $\Tt\mininflatearrow\Xx'$
or \tu{[}$\Tt$ has limit length and $(\Tt,\Sigma(\Tt))\mininflatearrow\Xx'$\tu{]}.
\end{lem}

We can also define the \emph{minimal} analogue of strong hull condensation.
It easily implies minimal inflation condensation; we do not know whether the converse holds.

\begin{dfn}\label{dfn:mhc}\index{strong hull condensation}
Let $\Omega>\om$ be regular. 
Let $\Sigma$ be an $(m,\Omega+1)$-strategy for an $m$-standard $M$.
We say that $\Sigma$ has \dfnemph{minimal hull condensation} (\dfnemph{mhc})
iff whenever $\Xx$ is via $\Sigma$ and $\Pi:\Tt\hookrightarrow_\minim\Xx$ is a minimal tree embedding,
then $\Tt$ is also via $\Sigma$.
\end{dfn}

One can also define the minimal analogue
of \emph{extra inflationary} from \cite{iter_for_stacks},
but we don't need it.
We now give some important examples of strategies with minimal hull condensation.
The proofs are just direct translations of
\cite[Lemma 4.45, Theorem 4.47]{iter_for_stacks}.

\begin{lem}\label{lem:unique_strat_has_inf_cond}
Let $\Sigma$ be an $(m,\Omega+1)$-strategy
for an $m$-standard $M$.
Suppose that $\Sigma$ is the unique $(m,\Omega+1)$-strategy for $M$.
Then $\Sigma$ has minimal hull condensation.
\end{lem}

\begin{rem}\label{rem:wDJ_implies_cond}
The second result
 deals with strategies with the
(weak) Dodd-Jensen property. We abbreviate \emph{Dodd-Jensen} with 
\emph{DJ}.
Note that only the first part
of the proof of \cite[Theorem 4.47]{iter_for_stacks},
which regards $\lambda$-indexing, is
relevant here;
in our setting it adapts immediately to give the proof for both indexings.

For this result, we assume that $M$ is a \emph{pure} extender mouse,
thus, not a strategy mouse.
This is important because the proof involves a comparison.
\end{rem}

\begin{tm}\label{tm:wDJ_implies_cond}
Let $\Omega>\om$ be regular.
Let $M$ be an $m$-standard pure $L[\es]$-premouse
with $\card(M)<\Omega$. Let $\Sigma$ be an $(m,\Omega+1)$-strategy for $M$ such that either
$\Sigma$ has the DJ property, or $M$ is countable and $\Sigma$ has weak DJ.
Then $\Sigma$ has minimal hull condensation.
\end{tm}
\begin{proof}
A routine adaptation of the first part of the proof of \cite[Theorem 4.47]{iter_for_stacks}.
\end{proof}

\subsection{Further minimal inflation terminology}

We adapt some further terminology from \cite{iter_for_stacks}:

\begin{dfn}\label{dfn:inflation_notation}
Let $\Tt\mininflatearrow\Xx$. Let
\[ (t,C,C^-,f,\Pivec)=(t,C,C^-,f,\Pivec)^{\Tt\mininflatearrow\Xx} \]
and let $\gamma_{\alpha;\beta}$, etc, be as in 
\ref{dfn:min_inflation}.
Suppose that $\Xx$ has successor length $\alpha+1$.

We say that $\Xx$ is:
\begin{enumerate}[label=--]
 \item \emph{\tu{(}$\Tt$\tu{)}-pending} iff $\alpha\in C^-$.\index{pending}
 \item \emph{non-\tu{(}$\Tt$\tu{)}-pending} iff $\alpha\notin C^-$.
\item \emph{\tu{(}$\Tt$\tu{)}-terminal} iff $\Tt$ has successor length and $\Xx$ 
is non-$\Tt$-pending.\index{terminal}
\end{enumerate}

We say that $\Xx$ is:\footnote{The terminology here is slightly different
to that in \cite{iter_for_stacks},
because we only deal with \emph{non-dropping},
as opposed to both \emph{non-dropping} and \emph{non-model-dropping},
and here, $\alpha\in C\cut C^-$ requires no drop in model or degree,
whereas only no drop in model in \cite{iter_for_stacks}.}
\begin{enumerate}[label=--]\index{terminally-non-dropping}
\item \emph{$\Tt$-terminally-non-dropping} iff
$\Tt$-terminal and $\alpha\in C$; hence, 
$f(\alpha)+1=\lh(\Tt)$ and  
\[ (\gamma_{\alpha;f(\alpha)},\delta_{\alpha;f(\alpha)}]_{\Xx}\inter\dropset_{
\deg}^\Xx=\emptyset, \]
\item \emph{$\Tt$-terminally-dropping} iff $\Tt$-terminal
and $\alpha\notin C$.
\end{enumerate}

Suppose $\Xx$ is $\Tt$-terminally-non-dropping and let 
$\alpha+1=\lh(\Xx)$ and 
$\beta=f(\alpha)$. Then we define\index{$\pi_\infty^{\Tt\mininflatearrow\Xx}$}
\[ \pi_\infty^{\Tt\mininflatearrow\Xx}:M^\Tt_\beta\to M^\Xx_\alpha \]
by $\pi_\infty^{\Tt\mininflatearrow\Xx} =
\pi_{\alpha;\beta 0\alpha}$.
\end{dfn}

\begin{rem}\label{rem:non-drop_inf_comm}
Suppose $\Xx$ is $\Tt$-terminally-non-dropping 
and $\Tt,\Xx$ are $m$-maximal. Note that
$\pi_\infty=\pi_\infty^{\Tt\mininflatearrow\Xx}$ is an $n$-embedding,
where $n=\deg^\Xx(\infty)$. 
If $\Tt$ is  also
terminally non-dropping, then note that $\Xx$ is terminally non-dropping,
$n=m$ and $\pi_\infty\com i^\Tt=i^\Xx$.
\end{rem}

\section{The factor tree $\Xx/\Tt$}\label{sec:factor_tree}

We now discuss the \emph{minimal} analogue of the factor tree of \cite[\S8]{iter_for_stacks}.
For the first part of the discussion there is essentially nothing new,
so we refer the reader to \cite{iter_for_stacks} for most of it.

\subsection{The factor tree order $<_{\Xx/\Tt}$}

Define \emph{$\Tt$-unravelling},
\emph{minimal-$\Tt$-good} (or just 
\emph{good}), as in \cite[Definition 8.1]{iter_for_stacks}
(with minimal inflations replacing inflations throughout).
For a good 
minimal inflation 
$\Xx$ of $\Tt$,
define the 
associated objects
$\lambda^\alpha$, $\zeta^\alpha$, $L^\alpha$
$\eta_\delta$, $\Xx^\alpha$,
$(t^\alpha,C^\alpha,\ldots)$, $\theta^\alpha$, 
$(\lambda^\alpha,\zeta^\alpha,L^\alpha,\Xx^\alpha,t^\alpha,\ldots)^{
\Tt\inflatearrow_{\minim}\Xx}$,
$(I^\alpha_{\xi})^{\Tt\inflatearrow_{\minim}\Xx}$,
$(\pi^\alpha_{\xi 
i})^{\Tt\inflatearrow_{\minim}\Xx}$,
etc,
as in \cite[Definition 8.2]{iter_for_stacks}.
Define ${<^{\Xx/\Tt}}$ as in \cite[8.3]{iter_for_stacks},
and 
$\Vv^{\geq\alpha}$,
$<_0^{(\alpha)}$ as in \cite[8.5]{iter_for_stacks}.
Then \cite[Lemmas 8.4, 8.6]{iter_for_stacks} hold,
after replacing inflations with minimal inflations.
We restate \cite[8.7]{iter_for_stacks},
because we use $\xi^\alpha_\kappa$
here, as opposed to the
$\gamma^\alpha_{\theta\kappa}$ of \cite{iter_for_stacks}
(recall $\xi^\alpha_\kappa=\xi_{\Pi_\alpha\kappa}$
and $\pi^\alpha_\kappa=\pi_{\Pi_\alpha\kappa}$
were specified in Definition \ref{dfn:pi_kappa}):

\begin{lem}\label{lem:<_Xx/Tt}
Let $\Tt\inflatearrow_{\minim}\Xx$ be good.
Adopt notation as above.
Let $\alpha\leq_{\Xx/\Tt}\beta<\lh(\Xx/\Tt)$ with $\lambda^\beta\in C^\beta$ \tu{(}so
 $\lambda^\alpha\in C^\alpha$ by \cite[8.6]{iter_for_stacks}\tu{)}.
Then:

\begin{enumerate}[label=\arabic*.,ref=\arabic*]
 \item\label{item:<_Xx/Tt_is_it_tree_order} $<_{\Xx/\Tt}$ is an iteration tree order on $\lh(\Xx/\Tt)$.
   \item\label{item:<_Xx/Tt_respects_<_Xx} For all 
$\mu<_\Xx\lambda<\lh(\Xx)$, we have
$\eta_\mu\leq_{\Xx/\Tt}\eta_\lambda$.
 \item\label{item:xi_class_is_Xx/Tt_below}
For all $(\theta,\kappa)$
with $\theta<\lh(\Tt)$
and $\kappa\leq\OR(M^\Tt_\theta)$
and $\theta=\alpha^\Tt_\kappa$, either 
$\xi^\alpha_{\kappa}\in[\lambda^\alpha,\lh(\Xx^\alpha))$ or 
$\xi^\alpha_{\kappa}\in L^\delta$ for some $\delta<_{\Xx/\Tt}\alpha$.
\item\label{item:gamma_agmt_along_branch} Suppose $\alpha<\beta$. Let $\xi+1=\successor^{\Xx/\Tt}(\alpha,\beta)$
and ${\chi}=\pred^{\Xx}(\lambda^{\xi+1})$. Then:
\begin{enumerate}[label=\tu{(}\alph*\tu{)}]
\item  $\chi\in L^{\alpha}$ and $\theta^\alpha\leq\theta\eqdef 
f^\alpha(\chi)\leq\theta^\beta$.
 \item For  ${\theta'}<\theta$, we have 
$I^\alpha_{\theta'}=I^\beta_{\theta'}\sub\chi$
and for $\kappa\leq\OR(M^\Tt_{\theta'})$
with $\theta'=\alpha^\Tt_\kappa$, we have
 $\xi^\alpha_{\kappa}=\xi^\beta_{{\kappa}}<\chi$,
 \item $\gamma^\alpha_{{\theta}}=\gamma^\beta_{\theta}$
 but $\delta^\alpha_{\theta}=\chi<_\Xx\delta^\beta_{\theta}$,
 \item If $\theta+1<\lh(\Tt)$ then for
$\kappa<\nutilde(\exit^\Tt_{\theta})$
with $\theta=\alpha^\Tt_\kappa$,
if
$\pi^\alpha_{\kappa}(\kappa)<\crit(E^\Xx_{\zeta^\xi})$ then
$\xi^\alpha_{\kappa}=\xi^\beta_{\kappa}$, and if
$\pi^\alpha_{\kappa}(\kappa)\geq\crit(E^\Xx_{\zeta^\xi})$ then
$\xi^\alpha_{\kappa}=\chi<_\Xx\xi^\beta_{\kappa}$.
\end{enumerate}
\end{enumerate}
\end{lem}
\begin{proof}
See the proof of \cite[Lemma 8.7]{iter_for_stacks}.\end{proof}

 \begin{lem}\label{lem:M^Xx^alpha_infty_ult_rep}
Let $\Tt\inflatearrow_{\minim}\Xx$ be good.
Adopt notation as above.
Let $\alpha<\lh(\Xx/\Tt)$ with $\lambda^\alpha\in C^{\Tt\mininflatearrow\Xx}$. Then \tu{(}i\tu{)}
for  $\lambda\in[\lambda^\alpha,\lh(\Xx^\alpha))$,
\[ 
\vec{F}^\alpha_\infty\eqdef\vec{F}^\alpha_{\lambda}=
\left<E^\Xx_{\zeta^\gamma}\right>_{\gamma+1
\leq_{\Xx/\Tt}\alpha}.\]
Therefore \tu{(}ii\tu{)}
$M^{\Xx^\alpha}_\infty=\Ult_n(M^\Tt_\infty,\vec{F}^\alpha_\infty)$
where $n=\deg^\Tt(\infty)$.
\end{lem}
\begin{proof}
 Recall here that $\vec{F}^\alpha_\lambda=\vec{F}_{\Pi_\alpha\lambda}$
 was defined in \ref{dfn:F-vec_tree_emb}.
 Part (i) is verified by an easy induction on $\lh(\Xx/\Tt)$. Part (ii) follows 
via Lemma  \ref{lem:Pi_is_standard}.
\end{proof}

 \begin{dfn}
Let $\Tt\mininflatearrow\Xx$ be good, 
$C=C^{\Tt\mininflatearrow\Xx}$ and $n=\deg^\Tt(\infty)$. Adopt notation as above.
Given $\alpha\leq_{\Xx/\Tt}\beta$ with $\lambda^\alpha,\lambda^\beta\in C$,
then 
\[ \pi^{\alpha\beta}:M^{\Xx^\alpha}_\infty\to M^{\Xx^\beta}_\infty \]
denotes the natural factor map given by 
\ref{lem:M^Xx^alpha_infty_ult_rep};
that is,
$\pi^{\alpha\beta}=i^{M^{\Xx^\alpha},n}_{\vec{F}^\beta_\infty\cut\vec{F}
^\alpha_\infty}$.

Suppose instead $\alpha\leq_{\Xx/\Tt}\beta<\lh(\Xx/\Tt)$ with
$\lambda^\alpha\notin C$. So by \cite[8.6]{iter_for_stacks},
 $\lh(\Xx^\alpha)=\lambda^\alpha+1$
and $\lh(\Xx^\beta)=\lambda^\beta+1$ and $\lambda^\alpha\leq_\Xx\lambda^\beta$.
If $(\lambda^\alpha,\lambda^\beta]_\Xx\inter\dropset_{\deg}^{\Xx}=\emptyset$,
let
\[ 
\pi^{\alpha\beta}=i^\Xx_{\lambda^\alpha\lambda^\beta}:M^{\Xx^\alpha}_\infty\to 
M^{\Xx^\beta}_\infty.\]
If  $\alpha$ is also a successor ordinal, let
\[ 
\pi^{*\alpha\beta}=i^{*\Xx}_{\lambda^\alpha\lambda^\beta}:M^{*\Xx^\alpha}_{
\lambda^\alpha}\to M^{\Xx^\beta}_\infty; \]
note that $M^{*\Xx^\alpha}_{\lambda^\alpha}\ins M^{\Xx^\gamma}_\lambda$
for some $\lambda\in L^\gamma$ where $\gamma=\pred^{\Xx/\Tt}(\alpha)$,
and if $\lambda^\gamma\in(C^-)^\gamma$
then possibly $\lambda^\gamma<\zeta^\gamma$.
\end{dfn}

We can now easily describe 
the full factor tree $\Xx/\Tt$:

\begin{dfn}
Let $\Tt\mininflatearrow\Xx$ be good and adopt notation as before.
Let $(N,n)=(M^\Tt_\infty,\deg^\Tt_\infty)$.
Then the \emph{factor tree} $\Xx/\Tt$
(or the \emph{flattening} of $(\Tt,\Xx)$) is the $n$-maximal tree 
$\Uu$ on $N$
such that  $\lh(\Uu)=\lh(\Xx/\Tt)$, ${<_\Uu}={<_{\Xx/\Tt}}$, and
$E^\Uu_\alpha=E^\Xx_{\zeta^\alpha}$ for each 
$\alpha+1<\lh(\Uu)$
(see Lemma \ref{lem:flattening_properties}).
\end{dfn}

\begin{lem}\label{lem:flattening_properties}
 Let $\Tt\mininflatearrow\Xx$ be good.  Then:
 \begin{enumerate}
 \item\label{item:flattening_exists} The factor tree $\Uu=\Xx/\Tt$ exists
  \tu{(}in particular, $\Uu$ has wellfounded models\tu{)}, and is unique.
 \item\label{U_drops_when_lambda_alpha} $[0,\alpha]_\Uu$ drops in model or 
degree iff $\lambda^\alpha\notin C^\alpha$, for all $\alpha<\lh(\Uu)$.
 \item\label{item:models_match} 
$(M^\Uu_\alpha,\deg^\Uu_\alpha)=(M^{\Xx^\alpha}_\infty,\deg^{\Xx^\alpha}
_\infty)$ for all $\alpha<\lh(\Uu)$.
 \item\label{item:main_embs_match} For $\alpha\leq^\Uu\beta$,
we have $(\alpha,\beta]_\Uu\inter\dropset^\Uu_{\deg}=\emptyset$
 iff $\pi^{\alpha\beta}$ is defined, and 
if $(\alpha,\beta]_\Uu\inter\dropset^\Uu_{\deg}=\emptyset$ then
 $i^\Uu_{\alpha\beta}=\pi^{\alpha\beta}$.
 \item\label{item:extra_embs_match} 
 If $\alpha+1\leq^\Uu\beta$ and 
$(\alpha+1,\beta]^\Uu\inter\dropset^\Uu_{\deg}=\emptyset$
then
$i^{*\Uu}_{\alpha+1,\beta}=\pi^{*\alpha+1,\beta}$.
\item\label{item:factor_comm} Suppose $\Tt,\Uu$ have successor length
and $\Xx$ is non-$\Tt$-pending,
so $\Xx$ also has successor length.
Let $\Ttvec=(\Tt,\Uu)$.
Then $b^{\Ttvec}\inter\dropset_{\deg}^{\Ttvec}=\emptyset$
iff $b^\Xx\inter\dropset_{\deg}^\Xx=\emptyset$.
If $b^{\Ttvec}\inter\dropset_{\deg}^{\Ttvec}=\emptyset$
then $i^{\Ttvec}_{0\infty}=i^\Xx_{0\infty}$.
 \end{enumerate}
\end{lem}
\begin{proof}
The uniqueness in part \ref{item:flattening_exists} is clear.
Parts \ref{U_drops_when_lambda_alpha}--\ref{item:extra_embs_match}
are by induction on segments $(\Xx/\Tt)\rest\eta$ and $\Uu\rest\eta$
of $\Xx/\Tt$ and $\Uu$ (see below), and then
part \ref{item:factor_comm}
follows easily from those things and the commutativity
properties of minimal tree embeddings (which we leave to the reader).

If $\eta=1$, everything is trivial, as we have $\Xx^0=\Tt$.

Now suppose we have the inductive hypotheses at $\eta=\alpha+1$
(so $M^\Uu_\alpha=M^{\Xx^\alpha}_\infty$ etc);
we want to extend to $\alpha+2$.
Let $E=E^\Xx_{\zeta^\alpha}$. We have $E\in\es_+(M^{\Xx^\alpha}_{\zeta^\alpha})$,
but $\lh(E)<\lh(E^{\Xx^\alpha}_{\zeta^\alpha})$ if $\zeta^\alpha+1<\lh(\Xx^\alpha)$,
as $E$ is $\Tt$-inflationary.
So $E\in\es_+(M^{\Xx^\alpha}_\infty)=\es_+(M^\Uu_\alpha)$.
And  $E$ is $\Uu\rest(\alpha+1)$-normal,
since $\zeta^\beta<\zeta^\alpha$  for 
$\beta<\alpha$, so
$\lh(E^\Uu_\beta)=\lh(E^\Xx_{\zeta^\beta})\leq\lh(E)$.

Let $\beta=\pred^\Uu(\alpha+1)$,
so $\beta$ is least such that $\crit(E)<\nutilde(\exit^\Uu_\beta)$.
So by \cite[8.4]{iter_for_stacks}
$\beta=\pred^{\Xx/\Tt}(\alpha+1)$, also as desired.\footnote{In 
\cite[8.4]{iter_for_stacks}
it says $\crit(E)<\iota(\exit^\Uu_\beta)$, but this is equivalent
saying $\crit(E)<\nutilde(\exit^\Uu_\beta)$ here,
as $\iota(\exit^\Uu_\beta)=\nutilde(\exit^\Uu_\beta)$
unless $\exit^\Uu_\beta$ is MS-indexed type 2, in which
case $\iota(\exit^\Uu_\beta)=\OR(\exit^\Uu_\beta)$.}
Let $\lambda=\pred^\Xx(\lambda^{\alpha+1})$. So $\lambda\in L^\beta$.
If $\lambda+1=\lh(\Xx^\beta)$
then by induction,
\[ 
(M^\Uu_\beta,\deg^\Uu_\beta)=
(M^{\Xx^\beta}_\infty,\deg^{\Xx^\beta}_\infty)=(M^{\Xx}_\lambda,\deg^{\Xx}_
\lambda), \]
and since $E^\Uu_\alpha=E=E^{\Xx^\alpha}_{\zeta^\alpha}$,
the inductive hypotheses are immediately maintained
(note there can be a drop in model or degree in this case).
So suppose $\lambda+1<\lh(\Xx^\beta)$, 
so $\lambda^\beta,\lambda\in(C^-)^{\beta}$
and $[0,\beta]_\Uu\inter\dropset_{\deg}^\Uu=\emptyset$.
But either $(\exit^{\Xx^\beta}_\lambda)^\passive$
is a cardinal proper
segment of $M^{\Xx^\beta}_\infty=M^\Uu_\beta$,
or we have MS-indexing, $E^{\Xx^\beta}_\lambda$ is superstrong,
$\lh(\Xx^\beta)=\lambda+2$ and 
$\OR(M^{\Xx^\beta}_{\lambda+1})=\lh(E^{\Xx^\beta}_\lambda)$). So either:
\begin{enumerate}
\item\label{item:no_drop_in_model} $E$ is total over 
$\exit^{\Xx^\beta}_\lambda$,
$\lambda^{\alpha+1}\in(C^-)^{\alpha+1}$, and $\Uu$ does not drop in model or degree at $\alpha+1$, or
\item\label{item:drop_in_model} $\lambda^{\alpha+1}\in\dropset^\Xx$ and $\alpha+1\in\dropset^\Uu$ and $M^{*\Xx}_{\lambda^{\alpha+1}}=M^{*\Uu}_{\alpha+1}\pins\exit^{\Xx^\beta}_\lambda$
(so note then that $\zeta^\beta=\lambda$, as $E$ is total over $\exit^\Uu_\beta$)
and
$\deg^\Xx(\lambda^{\alpha+1})=\deg^\Uu(\alpha+1)$.
\end{enumerate}
In case \ref{item:drop_in_model}, it is again routine to see that the hypotheses are maintained.
In case \ref{item:no_drop_in_model} (recalling
$[0,\beta]_\Uu\inter\dropset_{\deg}^\Uu=\emptyset$),
$M^\Uu_{\alpha+1}=\Ult_n(M^\Uu_\beta,E)$
and $i^\Uu_{\beta,\alpha+1}$ is the ultrapower map.
But also, using facts about minimal tree embeddings,
\[ M^{\Xx^\beta}_\infty=\Ult_n(N,\vec{F}^\beta_\infty),\]
\[ M^{\Xx^{\alpha+1}}_\infty=\Ult_n(N,\vec{F}^{\alpha+1}_\infty),\]
and $\vec{F}^{\alpha+1}_\infty=\vec{F}^\beta_\infty\conc\left<E\right>$,
so
$M^{\Xx^{\alpha+1}}_\infty=\Ult_n(M^{\Xx^\beta}_\infty,E)$
and $\pi^{\beta,\alpha+1}$ is the ultrapower map, as desired.

Now suppose that we have the inductive hypotheses below a limit $\eta$,
and we consider $\eta+1$.
Since $<_{\Xx/\Tt}$ is an iteration tree order
and ${<_{\Xx/\Tt}}\rest\eta={<_\Uu}\rest\eta$,
$[0,\eta)_\Uu=[0,\eta)_{\Xx/\Tt}$ does indeed give a $\Uu\rest\eta$-cofinal branch.
If for some $\alpha<_\Uu\eta$, we have $\lambda^\alpha\notin(C^-)^\alpha$,
then everything is easy, so suppose otherwise.
In particular, $\Uu$ does not drop in model or degree in $[0,\eta)_\Uu$.
Now $M^\Uu_\eta$ is just the direct limit as usual, and $i^\Uu_{\alpha\eta}$
the associated direct limit map.
By induction then,
\[ M^\Uu_\eta=\dirlim_{\alpha\leq_\Uu\beta<_\Uu\eta}\left(M^{\Xx^\alpha}_\infty,M^{\Xx^\beta}_\infty;\pi^{\alpha\beta}\right), \]
and $i^\Uu_{\alpha\eta}$ is the associated direct limit map.
But we have $\lambda^\eta\in C^{\eta}$,
$\vec{F}^\eta_\infty=\lim_{\alpha<_\Uu\eta}\vec{F}^\alpha_\infty$
and
$M^{\Xx^\eta}_\infty=\Ult_n(N,\vec{F}^\eta_\infty)$
and $\pi^{\alpha\eta}$ is an associated factor map,
and likewise for the $\pi^{\alpha\beta}$ for $\beta<_\Uu\eta$.
But these match the direct limit and associated maps just described,
as desired.
\end{proof}

\section{Minimal comparison}\label{sec:min_comp}

In this section we quickly adapt the techniques of comparative and genericity inflation
from \cite[\S5]{iter_for_stacks} to minimal inflations.

\subsection{Minimal comparison}\label{subsec:min_inf}

\begin{dfn}\label{dfn:min_comp}\index{minimal comparison}
Let $\Omega>\om$ be regular. Let $M$ be $m$-standard.
Let $\mathscr{T}$ be a set of 
$m$-maximal trees on $M$,
with each $\Tt\in\mathscr{T}$ of length $\leq\Omega+1$.
Let $\Xx$ be  $m$-maximal on $M$.
We say that $\Xx$ is a  \dfnemph{minimal comparison} of $\mathscr{T}$
(with respect to 
$\Omega$) 
iff:
\begin{enumerate}[label=--]
 \item $\Xx$ is a minimal inflation of each $\Tt\in\mathscr{T}$; let 
$t^\Tt=t^{\Tt\mininflatearrow\Xx}$, etc,
 for $\Tt\in\mathscr{T}$, 
 \item for each $\alpha+1<\lh(\Xx)$ there is $\Tt\in\mathscr{T}$ such that $t^{\Tt}(\alpha)=0$,
 \item $\Xx$ has successor length $\beta+1\leq\Omega+1$,
 \item if $\beta+1<\Omega$ then there is no $\Tt\in\mathscr{T}$
 with
$\beta\in(C^-)^{\Tt}$.\qedhere
 \end{enumerate}
\end{dfn}

The proof of \cite[Lemma 5.2]{iter_for_stacks} gives:

\begin{lem}[Minimal comparison]\label{lem:min_comp}
Let $\Omega>\om$ be regular.
Let $M$ be $m$-standard and $\Sigma$ be an $(m,\Omega+1)$-strategy for $M$
with minimal inflation condensation.
Let $\mathscr{T}$ be as in \ref{dfn:min_comp},
and suppose  $\card(\mathscr{T})<\Omega$
and each $\Tt\in\mathscr{T}$ is via $\Sigma$
with $\lh(\Tt)\leq\Omega+1$.
Then there is a unique minimal comparison $\Xx$ of $\mathscr{T}$
 via $\Sigma$.
Moreover, there is $\Tt\in\mathscr{T}$ such that, letting $\Tt'=\Tt$
if $\Tt$ has successor length, and
$\Tt'=\Tt\conc\Sigma(\Tt)$ otherwise, we have
\begin{enumerate}[label=--]
 \item  $\Xx$ is $\Tt'$-terminally-non-dropping, and
 \item if $\lh(\Xx)=\Omega+1$ then $\lh(\Tt')=\Omega+1$.
\end{enumerate}
\end{lem}

\subsection{Minimal genericity inflation}

The minimal version of genericity inflation
works essentially identically
to the non-minimal version  in \cite[\S5]{iter_for_stacks}. The key difference is of course that we inflate extenders minimally, as elsewhere in this paper.
The remaining details are as in \cite{iter_for_stacks}, so we omit further discussion. (Also recall that  \cite{iter_for_stacks} deals with u-fine structure for Mitchell-Steel indexing,
whereas that is not relevant here.)

\section{Minimal inflation stacks}\label{sec:inf_comm}

\subsection{Commutativity}

We need the adaptation of  commutativity of inflation 
\cite[Lemma 6.2]{iter_for_stacks}.
Properties \ref{item:f_comm}--\ref{item:gamma_comm,internal_coverage} are just 
as in
\cite{iter_for_stacks}.
But there are some differences in property
\ref{item:extended_comm_at_end}:
drops in model in \cite{iter_for_stacks}
correspond more to drops in model or degree here,
the conclusions
of \ref{item:extended_comm_at_end}\ref{item:map_commutativity}
are crucially different, because of minimality, and 
\ref{item:extended_comm_at_end}\ref{item:F-vec_description} is new.

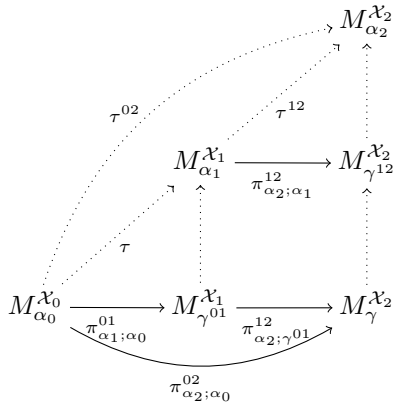
\begin{figure}
\centering
\begin{tikzpicture}
 [mymatrix/.style={
    matrix of  nodes,
    row sep=1.2cm,
    column sep=1.2cm}]
   \matrix(m)[mymatrix]{
{}&{}&{$M^{\Xx_2}_{\alpha_2}$}\\
 {}&{$M^{\Xx_1}_{\alpha_1}$}&{$M^{\Xx_2}_{\gamma^{12}}$}\\
{$M^{\Xx_0}_{\alpha_0}$}&{$M^{\Xx_1}_{\gamma^{01}}$}&{$M^{\Xx_2}_\gamma$}\\};
\path[->,font=\scriptsize,shorten >= -1pt, shorten <= -1pt]
(m-3-1) edge node[below] {$\pi^{01}_{\alpha_1;\alpha_0}$} (m-3-2)
(m-3-1) edge[bend right] node[below] {$\pi^{02}_{\alpha_2;\alpha_0}$} (m-3-3)
(m-3-1) edge[dotted] node[below] {$\ \ \tau$} 
(m-2-2)
(m-3-1) edge[dotted,bend left] node[left] 
{$\tau^{02}$} (m-1-3)
(m-3-2) edge node[below] {$\pi^{12}_{\alpha_2;\gamma^{01}}\ \ $} (m-3-3)
(m-3-2) edge[dotted] node {} (m-2-2)
(m-3-3) edge[dotted] node {} (m-2-3)
(m-2-2) edge node[below] {$\pi^{12}_{\alpha_2;\alpha_1}$} (m-2-3)
(m-2-2) edge[dotted] node[below] 
{$\ \ \tau^{12}$} 
(m-1-3)
(m-2-3) edge[dotted] node {} (m-1-3);
\end{tikzpicture}
\caption{Commutativity of minimal inflation.
We have
$\alpha_2\in C^{02}$,
$\alpha_1=f^{12}(\alpha_2)$, $\alpha_0=f^{02}(\alpha_2)=f^{01}(\alpha_1)$,
$\gamma^{k\ell}=\gamma^{k\ell}_{\alpha_\ell;\alpha_k}$,
$\gamma=\gamma^{12}_{\alpha_2;\gamma^{01}}$,
$\tau^{k\ell}=\pi^{k\ell}_{\alpha_\ell;\alpha_k i^{k\ell}\alpha_\ell}$
where $i^{k\ell}=i^{k\ell}_{\alpha_\ell;\alpha_k\alpha_\ell}$, and
$\tau=\pi^{01}_{\alpha_1;\alpha_0 i^{02}\alpha_1}$.
(So possibly $\dom(\tau)\neq\dom(\tau^{01})$, and possibly 
$\tau\not\sub\tau^{01}$.)
Note $\alpha_2=\delta^{02}_{\alpha_2;\alpha_0}=\delta^{12}_{\alpha_2;\alpha_1}$
and $\alpha_1=\delta^{01}_{\alpha_1;\alpha_0}$
and $\gamma^{01}\leq^{\Xx_1}\alpha_1$ and
$\gamma\leq^{\Xx_2}\gamma^{12}\leq^{\Xx_2}\alpha_2$.
Solid arrows indicate total embeddings, and dotted arrows 
indicate partial embeddings (with domain and codomain 
 initial segments of the models in the figure).
The vertical arrows 
depict ultrapowers by branch extenders;
for example, the left-most depicts
the ultrapower map corresponding
to $\Ult_n(U^{01}_{\alpha_1;\alpha_0 
i^{02}},\vec{E}^{\Xx_1}_{\gamma^{01}\alpha_1})$
where $n=m^{\Xx_0}_{\alpha_0 i^{02}}$,
and we refer to $i^{02}$ here (not $i^{01}$)
as we are considering factoring $\tau^{02}$.
The diagram commutes, after restricting to the relevant 
domains.}\label{fgr:inflation_commutativity}
\end{figure}

\begin{lem}[Commutativity of minimal inflation]\label{lem:inflation_commutativity}
Let $M$ be $m$-standard and $\Xx_0,\Xx_1,\Xx_2$ be such that:
\begin{enumerate}[label=--]
 \item each $\Xx_i$ is $m$-maximal on $M$,
 \item  $\Xx_0,\Xx_1$ have successor length,
 \item each $\Xx_{i+1}$ is a minimal inflation of $\Xx_i$,
 \item $\Xx_1$ is non-$\Xx_0$-pending
\end{enumerate}
 \tu{(}but $\Xx_2$ could have limit length or be $\Xx_1$-pending\tu{)}.
Then $\Xx_2$ is a minimal inflation of $\Xx_0$, and things commute in a reasonable fashion.
That is, let 
\[ (t^{ij},C^{ij},(C^-)^{ij},f^{ij},\left<\Pi^{ij}_\alpha\right>_{\alpha\in C^{ij}})
=(t,C,\ldots)^{\Xx_i\mininflatearrow\Xx_j}\]
for $i<j$; we also use analogous notation for other associated objects.
Let $\alpha_2<\lh(\Xx_2)$. If $k<2$ and $\alpha_2\in C^{k2}$ let $\alpha_k=f^{k2}(\alpha_2)$.
Then \tu{(}cf.~Figure \ref{fgr:inflation_commutativity}, which depicts a key case of the lemma\tu{)}:

\begin{enumerate}[label=C\arabic*.,ref=C\arabic*]
\item\label{item:f_comm} If $\alpha_2\in C^{02}$ then
$\alpha_2\in C^{12}$, $\alpha_1\in C^{01}$ and
$\alpha_0=f^{02}(\alpha_2)=f^{01}(f^{12}(\alpha_2))=f^{01}(\alpha_1)$.
 \item\label{item:t_equiv}
   $\alpha_2\in (C^-)^{02}$ and
$t^{02}(\alpha_2)=0$ iff

$\alpha_2\in (C^-)^{12}$ and $t^{12}(\alpha_2)=0$ and
$\alpha_1\in (C^-)^{01}$ and $t^{01}(\alpha_1)=0$.
\item\label{item:gamma_in_C} Suppose $\alpha_2\in C^{12}$ and $\alpha_1\in 
C^{01}$. Then:
\begin{enumerate}[label=\tu{(}\alph*\tu{)}]
\item If $\alpha_1+1=\lh(\Xx_1)$ then $\alpha_2\in C^{02}$.
\item\label{item:gamma_gamma_in_C^02} If $\beta\leq f^{01}(\alpha_1)$ and 
$\xi\in I^{01}_{\alpha_1;\beta}$ then
$\gamma^{12}_{\alpha_2;\xi}\in C^{02}$.
\item If $\beta<f^{01}(\alpha_1)$ and
$\xi=\delta^{01}_{\alpha_1;\beta}$ then $\delta^{12}_{\alpha_2;\xi}\in C^{02}$.
\end{enumerate}

\item\label{item:gamma_comm,internal_coverage}
Suppose $\alpha_2\in C^{02}$. Then:
\begin{enumerate}[label=\tu{(}\alph*\tu{)}]
 \item\label{item:gamma_comm_in} If $\beta\leq\alpha_0$ and $\gamma=\gamma^{01}_{\alpha_1;\beta}$
then
$\gamma^{02}_{\alpha_2;\beta}=\gamma^{12}_{\alpha_2;\gamma}$ and 
$\pi^{02}_{\alpha_2;\beta}=\pi^{12}_{\alpha_2;\gamma}\com\pi^{01}_{\alpha_1;\beta}$.
\item\label{item:interval_coverage_in}
$\bigcup_{\beta\leq 
\alpha_0}I^{02}_{\alpha_2;\beta}\sub\bigcup_{\beta\leq \alpha_1}
I^{12}_{\alpha_2;\beta}\sub C^{12}$.
\item\label{item:internal_coverage_in_2} If $\beta\leq\alpha_0$ and $\gamma\in I^{02}_{\alpha_2;\beta}$ then 
$f^{12}(\gamma)\in I^{01}_{\alpha_1;\beta}$.
\end{enumerate}

\item\label{item:extended_comm_at_end} Suppose $\alpha_2\in C^{02}$.
For $k<\ell\leq 2$ let:
\begin{enumerate}[label=--]
 \item $\gamma^{k\ell}=\gamma^{k\ell}_{{\alpha_\ell};\alpha_k\alpha_\ell}$
 \item $i^{k\ell}=i^{k\ell}_{{\alpha_\ell};\alpha_k\alpha_\ell}$
 \item $\tau^{k\ell}=
 \pi^{k\ell}_{{\alpha_\ell};\alpha_k i^{k\ell}\alpha_\ell}
\colon
M^{\Xx_k}_{\alpha_k i^{k\ell}}\to 
M^{\Xx_\ell}_{\alpha_\ell}$
\end{enumerate}
\tu{(}maybe $\gamma^{12}\neq\gamma^{12}_{\alpha_2;\gamma^{01}}$\tu{)}.
Then:
\begin{enumerate}[label=\tu{(}\alph*\tu{)}]
\item $i^{01}\leq i^{02}$ \tu{(}so
$(M^{\Xx_0}_{\alpha_0 i^{02}},m_{\alpha_0 i^{02}})\ins 
(M^{\Xx_0}_{\alpha_0 i^{01}},m_{\alpha_0 i^{01}})$,
with equality iff $i^{02}=i^{01}$\tu{)}.
\item $i^{01}+i^{12}=i^{02}$.
\item $i^{01}=i^{02}_{\gamma^{12};\alpha_0}$; that is, $i^{01}$ is the least $i'$ such that $\gamma^{12}\in 
I^{02}_{{\alpha_2};\alpha_0 i'}$.
\item If $i^{01}=i^{02}$ \tu{(}which holds iff $i^{12}=0$ iff 
$(\gamma^{12},\alpha_2]_{\Xx_2}\inter\dropset^{\Xx_2}_{\deg}=\emptyset$\tu{)}
then
$\tau^{02}=\tau^{12}\com\tau^{01}$.
\item\label{item:map_commutativity} Suppose $i^{01}<i^{02}$ \tu{(}which 
holds iff $i^{12}>0$ iff $(\gamma^{12},{\alpha_2}]_{\Xx_2}\inter
\dropset^{\Xx_2}_{\deg}\neq\emptyset$
iff 
$(M^{\Xx_0}_{\alpha_0 i^{02}},m^{\Xx_0}_{\alpha_0 i^{02}})
\pins
(M^{\Xx_0}_{\alpha_0 i^{01}},m^{\Xx_0}_{\alpha_0 i^{01}})$\tu{)}. Then
\[ M^{\Xx_1}_{\alpha_1 i^{12}}=U^{01}_{\alpha_1;\alpha_0 i^{02}\alpha_1}
=\Ult_{m^{\Xx_0}_{\alpha_0 i^{02}}}(M^{\Xx_0}_{\alpha_0 i^{02}},
\vec{F}^{01}_{\alpha_1;\alpha_1}), \]
and
$\tau^{02}=\tau^{12}\com \pi^{01}_{{\alpha_1};\alpha_0 
i^{02}\alpha_1}$.
\item\label{item:F-vec_description} 
$\vec{F}^{02}_{\alpha_2;\alpha_2}=\vec{F}^{01}_{\alpha_1;\alpha_1}\cd
\vec{F}^{12}_{\alpha_2;\alpha_2}$
(see Definitions \ref{dfn:F-vec_tree_emb} and \ref{dfn:G*F}).
\end{enumerate}
\end{enumerate}
\end{lem}

\begin{proof}[Proof of Lemma \ref{lem:inflation_commutativity}]
By induction on $\lh(\Xx_2)$. Fix $\alpha+1<\lh(\Xx_2)$,
and suppose the lemma holds for $\Xx_2\rest(\alpha+1)$.
Note first that the lemma for $\Xx_2\rest(\alpha+2)$
says the same things 
with respect to $\alpha_2\leq\alpha$ as does the lemma
for $\Xx_2\rest(\alpha+1)$, except for  \ref{item:f_comm}
when $\alpha_2=\alpha$,
as  
$\alpha\in\dom(t^{\Xx_i\mininflatearrow\Xx_2\rest(\alpha+2)})$
for $i=0,1$, 
but $\alpha\notin\dom(t^{\Xx_i\mininflatearrow\Xx_2'})$.
So letting $\alpha_2=\alpha$, 
we just need to verify
to verify \ref{item:f_comm} for $\alpha_2$,
and then verify the other parts for $\alpha_2+1$.
We consider three cases.

\begin{casetwo}\label{case:copy_copy}
$\alpha_2$ is $\Xx_1$-copying
and $\alpha_1$ is $\Xx_0$-copying; that is,
$\alpha_2\in(C^-)^{12}$ and $t^{12}(\alpha_2)=0$ and 
$\alpha_1\in(C^-)^{01}$ and $t^{01}(\alpha_1)=0$.
 
We first establish part \ref{item:t_equiv} for $\alpha_2$.
(Cf.~Figure 
\ref{fgr:inflation_commutativity},
which is related.)
Note that $\gamma^{12}\in C^{02}$,
by induction with part \ref{item:gamma_in_C}\ref{item:gamma_gamma_in_C^02},
applied with $\beta=\alpha_0'$
and $\xi=\delta_{\alpha_1;\beta}^{12}=\alpha_1$.
So by induction with
$\Xx_2\rest(\gamma^{12}+1)$,
we have $f^{02}(\gamma^{12})=\alpha_0'$
and  (by part \ref{item:extended_comm_at_end}\ref{item:F-vec_description}) 
\begin{equation}\label{eqn:ext_seqs_corresp}
 \vec{F}^{02}_{\gamma^{12};\gamma^{12}}=
\vec{F}^{01}_{\alpha_1;\alpha_1}\cd\vec{F}^{12}_{\gamma^{12};\gamma^{12}}.
\end{equation}
Since $t^{01}(\alpha_1)=0$, we have
$\exit^{\Xx_1}_{\alpha_1}=Q^{01}_{\alpha_1;\alpha_0'}
 =\Ult_0(\exit^{\Xx_0}_{\alpha_0'},\vec{F}^{01}_{\alpha_1;\alpha_1})$,
so by line (\ref{eqn:ext_seqs_corresp}), 
\[ Q\eqdef Q^{12}_{\gamma^{12};\gamma^{12}}=
 \Ult_0(\exit^{\Xx_1}_{\alpha_1},\vec{F}^{12}_{\gamma^{12};\gamma^{12}})
 =\Ult_0(\exit^{\Xx_0}_{\alpha_0'},\vec{F}^{02}_{\gamma^{12};\gamma^{12}})=
 Q^{02}_{\gamma^{12};\gamma^{12}}.\]
Moreover, since $\alpha_2=\delta^{12}_{\alpha_2;\alpha_1}$,
$\vec{E}\eqdef\vec{E}^{\Xx_2}_{\gamma^{12}\alpha_2}$
is $(Q,0)$-good and
$\exit^{\Xx_2}_{\alpha_2}=\Ult_0(Q,\vec{E})$.
But since $t^{12}(\beta)=1$
for every $\beta+1\in(\gamma^{12},\alpha_2]_{\Xx_2}$,
by induction $t^{02}(\beta)=1$ also,
and it follows that $\alpha_2\in C^{02}$ and
$\alpha_0'=f^{02}(\alpha_2)$
and $\delta^{02}_{\alpha_2;\alpha_0'}=\alpha_2$
and $Q^{02}_{\alpha_2;\alpha_0'}=\exit^{\Xx_2}_{\alpha_2}$,
establishing part \ref{item:t_equiv}.

So let $\alpha_0=\alpha_0'$.
It follows that  $\alpha_2+1\in C^{02}\inter C^{12}$
and $\alpha_1+1\in C^{01}$ and
$I^{ij}_{\alpha_i+1;\alpha_j+1}=\{\alpha_j+1\}$ 
and $f^{ij}(\alpha_j+1)=\alpha_i+1$ 
for $0\leq i<j\leq 2$.
This gives part \ref{item:f_comm} (for 
for $\Xx_2\rest(\alpha_2+2)$, with $\alpha_2+1$ replacing $\alpha_2$).
Part \ref{item:t_equiv} for $\alpha_2+1$ is trivial
(for the reasons discussed at the start of the proof).
Part \ref{item:gamma_in_C} is clear by induction.
For part \ref{item:gamma_comm,internal_coverage},
use induction and the considerations just mentioned,
and to see that 
\begin{equation}\label{eqn:pi_comm_in_inf_comm} 
\pi^{02}_{\alpha_2+1;\alpha_0+1}=
\pi^{12}_{\alpha_2+1;\alpha_1+1}\com
\pi^{01}_{\alpha_1+1;\alpha_0+1},\end{equation}
just use that $\pi^{ij}_{\alpha_j+1;\alpha_i+1}$
is the ultrapower map
$i^{M^{\Xx_0}_{\alpha_0+1},\deg^{\Xx_0}_{\alpha_0+1}}_{\vec{F}}$
where
\[ \vec{F}=
 \vec{F}^{ij}_{\alpha_j+1;\alpha_j+1}=\vec{F}^{ij}_{\alpha_j+1;\alpha_j}
=\vec{F}^{ij}_{\alpha_j;\alpha_j}, \]
which by part \ref{item:extended_comm_at_end}\ref{item:F-vec_description}
gives line (\ref{eqn:pi_comm_in_inf_comm}).
 These considerations also give part \ref{item:extended_comm_at_end}
 for $\alpha_2+1$ (in this case, with notation as there,
 we have $i^{02}=i^{01}=i^{12}=0$).
\end{casetwo}

\begin{casetwo}\label{case:Xx_1-inflationary} $\alpha_2$ is 
$\Xx_1$-inflationary (that is,
$t^{12}(\alpha_2)=1$).
 
Part  \ref{item:t_equiv} for $\alpha_2$:  We claim
$t^{02}(\alpha_2)=1$.
For if $\alpha_2\in(C^-)^{02}$ then by induction,
$\alpha_2\in C^{12}$ and $\alpha_1\in C^{01}$ and $f^{01}(\alpha_1)=\alpha_0$,
hence $\alpha_1\in(C^-)^{01}$, but then since $\Xx_1$ is non-$\Xx_0$-pending,
$\alpha_1+1<\lh(\Xx_1)$ and 
$\lh(E^{\Xx_1}_{\alpha_1})\leq\lh(E^{01}_{\alpha_1})$,
so $\alpha_2\in(C^-)^{12}$ and
 (as $t^{12}(\alpha_2)=1$
and considering part 
\ref{item:extended_comm_at_end}\ref{item:F-vec_description},
much as in the previous case)
$\lh(E^{\Xx_2}_{\alpha_2})<\lh(E^{12}_{\alpha_2}
)\leq\lh(E^{02}_{\alpha_2})$,
so $t^{02}(\alpha_2)=1$.

Let $\xi_2=\pred^{\Xx_2}(\alpha_2+1)$.

 Part \ref{item:f_comm}:
 Suppose $\alpha_2+1\in C^{02}$. Then
$\xi_2\in C^{02}$; let $\xi_0=f^{02}(\xi_2)$ and $\xi_1=f^{12}(\xi_2)$,
so also $\xi_1\in C^{01}$ and $\xi_0=f^{01}(\xi_1)$.

Suppose $\xi_0+1<\lh(\Xx_0)$.
Then $\xi_1+1<\lh(\Xx_1)$, since $\xi_1\in C^{01}$
and $\Xx_1$ is non-$\Xx_0$-pending.
So
$\exit^{\Xx_1}_{\xi_1}\ins Q^{01}_{\xi_1;\xi_0}$,
which implies $Q^{12}_{\xi_2;\xi_1}\ins Q^{02}_{\xi_2;\xi_0}$
(again using  part 
\ref{item:extended_comm_at_end}\ref{item:F-vec_description}),
and since $\alpha_2+1\in C^{02}$,
$E^{\Xx_2}_{\alpha_2}$ is $(Q^{02}_{\xi_2;\xi_0},0)$-good,
but then also $(Q^{12}_{\xi_2;\xi_1},0)$-good,
so $\alpha_2+1\in C^{12}$, and $f^{12}(\alpha_2+1)=f^{12}(\xi_2)=\xi_1$,
and since $\xi_1\in C^{01}$ and $f^{01}(\xi_1)=\xi_0$, this
gives part \ref{item:f_comm}.

If instead $\xi_0+1=\lh(\Xx_0)$ then
$\alpha_2+1\notin\dropset_{\deg}^{\Xx_2}$,
and it is straightforward (and similar to before).

Parts \ref{item:gamma_in_C} and \ref{item:gamma_comm,internal_coverage}
for $\alpha_2+1$
are easy by induction.

Part \ref{item:extended_comm_at_end}: Suppose $\alpha_2+1\in C^{02}$.
Then
$\Pi^{i2}_{\alpha_2+1}$ is the
minimal $E^{\Xx_2}_{\alpha_2}$-inflation of $\Pi^{i2}_{\xi_2}$ for $i=0,1$,
and part \ref{item:extended_comm_at_end}
at $\alpha_2+1$ follows that part at $\xi_2$.
In particular,
Figure \ref{fgr:inflation_commutativity} at stage $\alpha_2+1$
is derived easily from the corresponding figure at $\xi_2$, by 
simply adding one further step of iteration above $M^{*\Xx_2}_{\alpha_2+1}\ins 
M^{\Xx_2}_{\xi_2}$, and regarding part 
\ref{item:extended_comm_at_end}\ref{item:F-vec_description}, we have
 $\vec{F}^{i2}_{\alpha_2+1;\alpha_2+1}=
\vec{F}^{i2}_{\xi_2;\xi_2}\conc\left<E\right>$
where $E=E^{\Xx_2}_{\alpha_2}$
for $i=0,1$,
and by induction,
\[ \vec{F}^{01}_{\xi_1;\xi_1}\cd \vec{F}^{12}_{\xi_2;\xi_2}=
 \vec{F}^{02}_{\xi_2;\xi_2},\]
and note that by normality of $\Xx_2$
(and that all
inflated $\vec{F}^{12}_{\xi_2;\xi_2}$-images of 
the extenders
in $\vec{F}^{01}_{\xi_1;\xi_1}$ are either used along $(0,\xi_2]_{\Xx_2}$
or are nested into some other extender used along that branch),
$E$ is non-nested
in the $\cd$-product
\[ 
\vec{F}^{01}_{\xi_1;\xi_1}\cd(\vec{F}^{12}_{\xi_2;\xi_2}\conc\left<E\right>),\]
so this $\cd$-product is just
$(\vec{F}^{01}_{\xi_1;\xi_1}\cd\vec{F}^{12}_{\xi_2;\xi_2})\conc\left<E\right>=
\vec{F}^{02}_{\alpha_2+1;\alpha_2+1}$.
 \end{casetwo}

 \begin{casetwo}\label{case:copy_inflation}
 $\alpha_2$ is $\Xx_1$-copying but $\alpha_1$ is $\Xx_0$-inflationary 
($\alpha_2\in(C^-)^{12}$ and $t^{12}(\alpha_2)=0$ and $t^{01}(\alpha_1)=1$).

So $\alpha_2+1\in C^{12}$ and $f^{12}(\alpha_2+1)=\alpha_1+1$
and $\gamma^{12}_{\alpha_2+1;\alpha_1+1}=\alpha_2+1$.
Note $t^{02}(\alpha_2)=1$,
so \ref{item:t_equiv} holds for $\alpha_2$.
Let $\xi_i=\pred^{\Xx_i}(\alpha_i+1)$ for $i=1,2$.
Then $\xi_2\in C^{12}$ and $f^{12}(\xi_2)=\xi_1$.
Applying induction to stage $\xi_2$,
and with calculations similar to before,
we get that $\alpha_2+1\in C^{02}$
iff $\alpha_1+1\in C^{01}$; and if $\alpha_2+1\in C^{02}$ then, letting $\xi_0=f^{02}(\xi_2)=f^{01}(\xi_1)$,
we have $f^{02}(\alpha_2+1)=\xi_0=f^{01}(\alpha_1+1)$.
So part \ref{item:f_comm} holds.

Parts \ref{item:gamma_in_C} and \ref{item:gamma_comm,internal_coverage}
for $\alpha_2+1$ 
follow easily from the preceding remarks and induction.
Part \ref{item:extended_comm_at_end}: Suppose $\alpha_2+1\in C^{02}$,
so as discussed above, $\alpha_1+1\in C^{01}$,
and $\xi_i\in C^{0i}$ for $i=1,2$.
Note 
$i^{12}=0$ (notation as in \ref{item:extended_comm_at_end}).
The diagram at $\alpha_2+1$
is given by adding a commuting square
to the top of the diagram at $\xi_2$,
using
$E^{\Xx_1}_{\alpha_1}$ and $E^{\Xx_2}_{\alpha_2}$.
Consider part \ref{item:extended_comm_at_end}\ref{item:F-vec_description}.
Since $t^{0i}(\alpha_i)=1$ for $i=1,2$,
we have
$\vec{F}^{0i}_{\alpha_i+1;\alpha_i+1}=\vec{F}^{0i}_{\xi_i;\xi_i}\conc
\left<E^{\Xx_i}_{\alpha_i}\right>$
for $i=1,2$.
By induction,
$\vec{F}^{02}_{\xi_2;\xi_2}=
\vec{F}^{01}_{\xi_1;\xi_1}\cd\vec{F}^{12}_{
\xi_2;\xi_2}$,
so
\[  \vec{F}^{02}_{\alpha_2+1;\alpha_2+1}=
(\vec{F}^{01}_{\xi_1;\xi_1}\cd\vec{F}^{12}_{
\xi_2;\xi_2})\conc\left<E^{\Xx_2}_{\alpha_2}\right>, \]
and letting 
$\vec{F}^{12}_{\alpha_2;\alpha_2}=\vec{F}^{12}_{\xi_2;\xi_2}\conc\vec{F}$
(note $\vec{F}^{12}_{\xi_2;\xi_2}\ins\vec{F}^{12}_{\alpha_2;\alpha_2}$), note 
that
$\crit(E)<\crit(E^{\Xx_2}_{\alpha_2})<\crit(F)$
for all $E\in\vec{F}^{12}_{\xi_2;\xi_2}$
and  $F\in\vec{F}$,
\[ \exit^{\Xx_2}_{\alpha_2}=\Ult_0(
\exit^{\Xx_1}_{\alpha_1},\vec{F}^{12}_{\xi_2;\xi_2}\conc\vec{F}),\]
and all $F\in\vec{F}$ are nested in this product, so
\[ \vec{F}^{01}_{\alpha_1+1;\alpha_1+1}\cd
 \vec{F}^{12}_{\alpha_2+1;\alpha_2+1}=
(\vec{F}^{01}_{\xi_1;\xi_1}\cd\vec{F}^{12}_{\xi_2;\xi_2})\conc
\left<E^{\Xx_2}_{\alpha_2}\right>=\vec{F}^{02}_{\alpha_2+1;\alpha_2+1},\]
as desired.
\end{casetwo}

This completes the successor case. The limit case is analogous,
and like in the proof of  \cite[Lemma 6.2]{iter_for_stacks},
so the reader should refer there.
\end{proof}

An easy consequence is (for terminology etc see Definition \ref{dfn:inflation_notation}):

\begin{cor}\label{cor:terminal_dropping_equiv} Let $\Xx_0,\Xx_1,\Xx_2$ be as in 
\ref{lem:inflation_commutativity}.
Suppose that $\Xx_2$ is $\Xx_1$-terminal and $\Xx_1$ is 
$\Xx_0$-terminal.
Then $\Xx_2$ is $\Xx_0$-terminal.
Moreover, $\Xx_2$ is 
$\Xx_0$-terminally-dropping 
iff either $\Xx_1$ is $\Xx_0$-terminally-dropping
or $\Xx_2$ is $\Xx_1$-terminally-dropping.
Moreover, if $\Xx_2$ is $\Xx_0$-terminally-non-dropping then
$\pi_\infty^{02}=\pi_\infty^{12}\com\pi_\infty^{01}$.
\end{cor}

\subsection{Continuous stacks}

\begin{dfn}\label{dfn:Xxvec_good}
 Let $M$ be $m$-standard and 
$\vec{\Xx}=\left<\Xx_\alpha\right>_{\alpha<\lambda}$ a sequence of $m$-maximal
trees on $M$.
 We say $\vec{\Xx}$ is a \emph{\tu{(}degree $m$, on $M$\tu{)} terminal minimal 
inflation stack} iff $\Xx_\beta$ is an $\Xx_\alpha$-terminal minimal inflation
 of $\Xx_\alpha$ for all $\alpha<\beta<\lambda$,
and $\lh(\Xx_\alpha)$ a successor
 for each $\alpha<\lambda$.
 If $\vec{\Xx}$ is a terminal minimal inflation stack, $\vec{\Xx}$ is 
\emph{continuous} iff for each limit $\eta<\lambda$,
 $\Xx_\eta$ is a minimal comparison of 
$\left<\Xx_\alpha\right>_{\alpha<\eta}$.\end{dfn}

\begin{lem}\label{lem:cont_term_stack_event}
 Let $\vec{\Xx}$ be a continuous terminal minimal inflation stack.
 For $\nu<\eta<\lh(\vec{\Xx})$, write $C^{\nu\eta}=C^{\Xx_\nu\mininflatearrow\Xx_\eta}$, etc.
 Let $\eta<\lh(\vec{\Xx})$ be a limit.
 Then:
 \begin{enumerate}
  \item\label{item:every_xi_associated} For every $\xi<\lh(\Xx_\eta)$, there is $\nu<\eta$
  such that for all $\alpha\in[\nu,\eta)$:
   \begin{enumerate}
   \item\label{item:assoc_no_drop} $\xi\in C^{\alpha\eta}$ and $\xi\in I^{\alpha\eta}_{\xi;f^{\alpha\eta}(\xi)0}$, and
   \item\label{item:copying} if $\xi+1<\lh(\Xx_\eta)$ then $t^{\alpha\eta}(\xi)=0$.
   \end{enumerate}
  \item\label{item:ev_term_non-drop} There is $\nu<\eta$ such that 
   $\Xx_\eta$ is $\Xx_\nu$-terminally-non-dropping (and hence,
   $\Xx_\eta$ is $\Xx_\alpha$-terminally-non-dropping for each $\alpha\in[\nu,\eta)$).
   \item\label{item:charac_dropping} For all $\beta<\gamma<\lh(\vec{\Xx})$,
   $\Yy_\gamma$ is $\Yy_\beta$-terminally-non-dropping
   iff $\Yy_{\alpha+1}$ is $\Yy_\alpha$-terminally-non-dropping
   for all $\alpha\in[\beta,\gamma)$.
 \end{enumerate}
\end{lem}

For the first two parts we only use the continuity of the stack at $\eta$, not other limits.

\begin{proof}
Part \ref{item:every_xi_associated}:
The proof is by induction on $\xi$. In general it suffices
to find $\nu$ witnessing part \ref{item:assoc_no_drop},
because then if $\xi+1<\lh(\Xx_\eta)$
and we take $\nu'<\eta$ with
$t^{\nu'\eta}(\xi)=0$, then $\max(\nu,\nu')$
works, because by commutativity of inflation \ref{lem:inflation_commutativity},
then $t^{\alpha\eta}(\xi)=0$ for all $\alpha\in[\nu',\eta)$.
Part \ref{item:assoc_no_drop} for $\xi=0$ is trivial.
For $\xi=\zeta+1$, note that if $\nu$ witnesses both parts
for $\zeta$,
then $\nu$ witnesses part \ref{item:assoc_no_drop} for $\xi+1$.
For limit $\xi$,
let $\xi'<^{\Xx_\eta}\xi$ be such that
$(\xi',\xi)_{\Xx_\eta}\inter\dropset_{\deg}^{\Xx_\eta}=\emptyset$
and let $\nu$ be a witness for $\xi'$; then $\nu$
also works for part \ref{item:assoc_no_drop} for $\xi$.

Part \ref{item:ev_term_non-drop}: Since $\Xx_\eta$ is 
to be $\Xx_\nu$-terminal for all $\nu<\eta$,
this follows
immediately from part \ref{item:assoc_no_drop}
at stage $\xi$ where $\xi+1=\lh(\Xx_\eta)$.

Part \ref{item:charac_dropping}:
If $\Yy_\gamma$ is $\Yy_\beta$-terminally-non-dropping
then it follows that for each $\alpha\in[\beta,\gamma)$,
$\Yy_{\alpha+1}$ is $\Yy_\alpha$-terminally-non-dropping,
by iterated application of Corollary 
 \ref{cor:terminal_dropping_equiv}
to the stack $(\Yy_\beta,\Yy_\alpha,\Yy_{\alpha+1},\Yy_\gamma)$.
Conversely, supposing that $\Yy_{\alpha+1}$ is $\Yy_\alpha$-terminally-non-dropping
for each $\alpha\in[\beta,\gamma)$,
proceed by induction on $\eta\in(\beta,\gamma]$ to show
that $\Yy_\eta$ is $\Yy_\beta$-terminally-non-dropping,
again using the same corollary, together with part \ref{item:ev_term_non-drop} to handle limits $\eta$.
\end{proof}

\begin{dfn}
 Let $\vec{\Xx}$ be a continuous terminal minimal inflation stack
 of length $\lambda$. Write
 $C^{\nu\eta}$, etc, as above.
 Let $\eta<\lambda$ with $\eta$ a limit and  $\xi<\lh(\Xx_\eta)$.
Fix
$\nu<\eta$ with
  $\xi\in C^{\nu\eta}$. For $\alpha\in[\nu,\eta]$ (note 
$\xi\in C^{\alpha\eta}$, by commutativity
of minimal inflation \ref{lem:inflation_commutativity})
  let $\xi_\alpha=f^{\alpha\eta}(\xi)$ (so $\xi=\xi_\eta$ and $\xi_\alpha\in C^{\nu\alpha}$ and $\xi_\alpha=\delta^{\nu\alpha}_{\xi_\alpha;\xi_\nu}$).
  Let $\nu\leq\alpha\leq\beta\leq\eps\leq\eta$. Now
\begin{enumerate}[label=(\roman*)]
 \item   If (*i) $\xi_\eta\in(C^-)^{\nu\eta}$ and 
$t^{\nu\eta}(\xi_\eta)=0$, then let
$\omega^{\alpha\beta}=\omega^{\alpha\beta}_{\xi_\beta;\xi_\alpha\xi_\beta}
:\exit^{\Xx_\alpha}_{\xi_\alpha}\to\exit^{\Xx_\beta}_{\xi_\beta}$
(so 
$\omega^{\alpha\eps}=\omega^{\beta\eps}\com\omega^{\alpha\beta}$).
  \item If (*ii) $(\gamma^{\nu\eta}_{\xi_\eta;\xi_\nu},\xi_\eta]_{\Xx_\eta}$ 
does not drop in model or degree
(so neither does 
$(\gamma^{\alpha\beta}_{\xi_\beta;\xi_\alpha},\xi_\beta]_{\Xx_\beta}$), then let
 \[ \pi^{\alpha\beta}=\pi^{\alpha\beta}_{\xi_\beta;\xi_\alpha 0\xi_\beta}\colon M^{\Xx_\alpha}_{\xi_\alpha}\to M^{\Xx_\beta}_{\xi_\beta} \]
(so $\pi^{\alpha\eps}=\pi^{\beta\eps}\com\pi^{\alpha\beta}$).
\end{enumerate}

Given such $\eta,\xi,\nu,\left<\xi_\alpha\right>_{\nu\leq\alpha\leq\eta}$,
we say that $\vec{\Xx}$ is \emph{exit-good at $(\eta,\xi,\nu)$} iff,
if ($*$i) holds then
\[ \exit^{\Xx_\eta}_{\xi_\eta}=\dirlim_{\nu\leq\alpha\leq\beta<\eta}\left(\exit^{\Xx_\alpha}_{\xi_\alpha},\exit^{\Xx_\beta}_{\xi_\beta};\omega^{\alpha\beta}\right)\text{ and }\]
\[ \omega^{\alpha\eta}\text{ is the direct limit map for }\nu\leq\alpha<\eta,\]
and \emph{model-good at $(\eta,\xi,\nu)$} iff,
if ($*$ii) holds then
\[ M^{\Xx_\eta}_{\xi_\eta}=\dirlim_{\nu\leq\alpha\leq\beta<\eta}\left(M^{\Xx_\alpha}_{\xi_\alpha},M^{\Xx_\beta}_{\xi_\beta};\pi^{\alpha\beta}\right)\text{ and } \]
\[ \pi^{\alpha\eta}\text{ is the direct limit map for }\nu\leq\alpha<\eta.\]
We say that $\vec{\Xx}$ is \emph{good} iff $\vec{\Xx}$ is exit-
and model-good at all such $(\eta,\xi,\nu)$.
\end{dfn}

\begin{lem}\label{lem:cont_term_stack_good}
 Let $\vec{\Xx}=\left<\Xx_\alpha\right>_{\alpha<\lambda}$ be a continuous terminal minimal inflation stack.
 Then $\vec{\Xx}$ is good.
\end{lem}
\begin{proof}
We prove exit- and model-goodness at each $(\eta,\xi,\nu)$, by induction on 
limits
$\eta<\lh(\vec{\Xx})$, with a sub-induction $\xi<\lh(\Xx_\eta)$.
So fix $\eta$ and $\xi<\lh(\Xx_\eta)$ and $\nu<\eta$ with $\xi_\eta=\xi\in 
C^{\nu\eta}$
 and adopt notation as in \ref{dfn:Xxvec_good}.

 \begin{casethree} $\xi=0$
  
  This case is trivial.
 \end{casethree}
\begin{casethree} $\xi=\upsilon+1$.
 
 Suppose \ref{dfn:Xxvec_good}($*$ii) holds;
 that is, $(\gamma^{\nu\eta}_{\xi_\eta;\xi_\nu},\xi_\eta]_{\Xx_\eta}$ does not drop in model or degree;
 we must verify model-goodness at $(\eta,\xi,\nu)$.
 
 Suppose first that in fact, ($*$) $\upsilon\in(C^-)^{\nu\eta}$ and
 $t^{\nu\eta}(\upsilon)=0$. Let $\upsilon_\alpha=f^{\alpha\eta}(\upsilon)$ for $\nu\leq\alpha\leq\eta$.
 Then by exit-goodness at $(\eta,\upsilon,\nu)$,
 \[ \exit^{\Xx_\eta}_{\upsilon_\eta}=\dirlim_{\nu\leq\alpha\leq\beta<\eta}\left(\exit^{\Xx_\alpha}_{\upsilon_\alpha},\exit^{\Xx_\beta}_{\upsilon_\beta};\omega^{\alpha\beta}_{\upsilon_\beta;\upsilon_\alpha\upsilon_\beta}\right)\text{ and } \]
\[ \om^{\alpha\eta}_{\upsilon_\eta;\upsilon_\alpha\upsilon_\eta}\text{ is the direct limit map for }\nu\leq\alpha<\eta.\]
But by properties of minimal inflation, for $\nu\leq\alpha\leq\beta\leq\eta$,
\[ 
\exit^{\Xx_\beta}_{\upsilon_\beta}=\Ult_0(\exit^{\Xx_\alpha}_{\upsilon_\alpha},
\vec{F}^{\alpha\beta}_{\upsilon_\beta;\upsilon_\beta})\text{ and }\]
\[ \om^{\alpha\beta}_{\upsilon_\beta;\upsilon_\alpha\upsilon_\beta}\text{ is the ultrapower map},\]
and letting $k=\deg^{\Xx_\alpha}(\xi_\alpha)$ (note $k$ is independent of $\alpha$),
\[ 
M^{\Xx_\beta}_{\xi_\beta}=\Ult_k(M^{\Xx_\alpha}_{\xi_\alpha},\vec{F}^{
\alpha\beta}_{\xi_\beta;\xi_\beta})\text{ and 
}\vec{F}^{\alpha\beta}_{\xi_\beta;\xi_\beta}=
\vec{F}^{\alpha\beta}_{\upsilon_\beta;\upsilon_\beta}\text{ and}\]
\[ \pi^{\alpha\beta}\text{ is the ultrapower map}.\]
So let
\[ 
\bar{M}=\dirlim_{\nu\leq\alpha\leq\beta<\eta}\left(M^{\Xx_\alpha}_{\xi_\alpha},
M^{\Xx_\beta}_{\xi_\beta};\pi^{\alpha\beta}\right),\]
\[ \bar{\pi}_\alpha:M^{\Xx_\alpha}_{\xi_\alpha}\to\bar{M}\text{ be the direct limit map and}\]
\[ \sigma:\bar{M}\to M^{\Xx_\eta}_{\xi_\eta}\text{ the map with }
\sigma\com\bar{\pi}_\alpha=\pi^{\alpha\eta}\text{ for all }\alpha.
\]
(the latter existing since 
$\pi^{\beta\eta}\com\pi^{\alpha\beta}=\pi^{\alpha\eta}$).
Note that
$\sigma\rest\nu(E^{\Xx_\eta}_{\upsilon_\eta})=\id$,
and so by commutativity (and the degree of elementarity of the maps),
therefore
$\bar{M}=M^{\Xx_\eta}_{\xi_\eta}$ and $\sigma=\id$ and 
$\bar{\pi}_\alpha=\pi^{\alpha\eta}$,
which gives model-goodness at $(\eta,\xi,\nu)$.

Next suppose instead that ($*$) above fails.
Let $\nu'\in(\nu,\eta)$ be such that ($*$) holds at $\nu'$.
Still $\xi_\nu\in C^{\nu\nu'}$ and 
$[\gamma^{\nu\nu'}_{\xi_\nu;\xi_{\nu'}},\xi_{\nu'})_{\Xx_{\nu'}}$
does not drop in model or degree (and $\xi_{\nu'}=\delta^{\nu\nu'}_{\xi_\nu;\xi_{\nu'}}$).
By model-goodness at $(\eta,\xi,\nu')$, it suffices to verify 
that $\pi^{\nu\eta}$ is an appropriate direct limit map.
But by commutativity, $\pi^{\nu\eta}=\pi^{\nu'\eta}\com\pi^{\nu\nu'}$,
and since $\pi^{\nu'\eta}$ is an appropriate direct limit map,
so is $\pi^{\nu\eta}$.

Note that by model-goodness,
for $\nu\leq\alpha<\eta$ such that \ref{dfn:Xxvec_good}($*$ii) holds
for $\nu$, ($\dagger$) $\vec{F}^{\alpha\eta}_{\xi_\eta;\xi_\eta}$
is derived from $\pi^{\alpha\eta}$.

Now suppose that \ref{dfn:Xxvec_good}($*$i) holds,
that is, $\xi_\eta\in(C^-)^{\nu\eta}$ and $t^{\nu\eta}(\xi_\eta)=0$;
 we must verify exit-goodness at $(\eta,\xi,\nu)$.
If \ref{dfn:Xxvec_good}($*$ii) also holds, then
exit-goodness follows from
model-goodness, because
\[ \exit^{\Xx_\eta}_{\xi_\eta}=\Ult_0(\exit^{\Xx_\alpha}_{\xi_\alpha},
\vec{F}^{\alpha\eta}_{\xi_\eta;\xi_\eta}) \]
and by $(\dagger)$,
and because of the agreement between
the direct limit maps relevant to exit-goodness
with those relevant to model-goodness.
Now suppose that \ref{dfn:Xxvec_good}($*$ii) fails.
Let $\nu'\in(\nu,\eta)$ be such that \ref{dfn:Xxvec_good}($*$ii) holds at $(\eta,\xi,\nu')$.
Then by commutativity, much as in the previous paragraph, model-goodness at $(\eta,\xi,\nu')$
implies exit-goodness at $(\eta,\xi,\nu)$.
\end{casethree}

\begin{casethree}
 $\xi$ is a limit.

 Suppose that \ref{dfn:Xxvec_good}($*$ii) holds,
 that is, $(\gamma^{\nu\eta}_{\xi_\eta;\xi_\nu},\xi_\eta]_{\Xx_\eta}$ does not drop in model or degree;
 we must verify model-goodness at $(\eta,\xi,\nu)$.
  
 Suppose first that $\gamma^{\nu\eta}_{\xi_\eta;\xi_\nu}=\xi_\eta$.
 Then model-goodness at $(\eta,\xi,\nu)$ follows easily by induction,
 using the commutativity of the various maps and the fact that $M^{\Xx_\eta}_{\xi_\eta}$ is the direct limit under iteration maps of $\Xx_\eta$.
 Here is some more detail.
 For $\nu\leq\alpha\leq\beta\leq\eps\leq\eta$ and $\pi^{\alpha\beta}$ as before,
 we have
 \begin{equation}\label{eqn:pis_comm} \pi^{\alpha\eps}=\pi^{\beta\eps}\com\pi^{\alpha\beta},\end{equation}
 and for $\upsilon_\nu<_{\Xx_\nu}\xi_\nu$ such that $(\upsilon_\nu,\xi_\nu]_{\Xx_\nu}$
 does not drop in model or degree,
 and $\upsilon_\alpha=\gamma^{\nu\alpha}_{\upsilon_\alpha;\upsilon_\nu}$,
 so $\upsilon_\beta=\gamma^{\alpha\beta}_{\upsilon_\beta;\upsilon_\alpha}$ and
 \[ \bar{\pi}^{\alpha\beta}\eqdef\pi^{\alpha\beta}_{\upsilon_\beta;\upsilon_\alpha0\upsilon_\beta}:M^{\Xx_\alpha}_{\upsilon_\alpha}\to M^{\Xx_\beta}_{\upsilon_\beta},\] 
 we have $\upsilon_\alpha<_{\Xx_\alpha}\xi_\alpha$ and $(\upsilon_\alpha,\xi_\alpha]_{\Xx_\alpha}$
 does not drop in model or degree and
 \[ \pi^{\alpha\beta}\com i^{\Xx_\alpha}_{\upsilon_\alpha\xi_\alpha}=i^{\Xx_\beta}_{\upsilon_\beta\xi_\beta}\com\bar{\pi}^{\alpha\beta} \]
 and
 $\bar{\pi}^{\alpha\eps}=\bar{\pi}^{\beta\eps}\com\bar{\pi}^{\alpha\beta}$.
 By induction (using model-goodness) we also have
 \[ M^{\Xx_\eta}_{\upsilon_\eta}=\bigcup_{\alpha<\eta}\rg(\bar{\pi}^{\alpha\eta}).\]
But then since
$M^{\Xx_\eta}_{\xi_\eta}=\bigcup_{\rho<_{\Xx_\eta}\xi_\eta}\rg(i^{\Xx_\eta}_{\rho\xi_\eta})$,
and
$\Gamma^{\nu\eta}_{\xi_\eta}``[0,\xi_\nu)_{\Xx_\nu}$
is cofinal in $\xi_\eta$, therefore
\[ M^{\Xx_\eta}_{\xi_\eta}=\bigcup_{\alpha<\eta}\rg(\pi^{\alpha\eta})\]
and so also by line (\ref{eqn:pis_comm}), $\pi^{\alpha\eta}$ is the direct limit map,
giving model-goodness at $(\eta,\xi,\nu)$, as desired.

Now suppose instead that $\gamma^{\nu\eta}_{\xi_\eta;\xi_\nu}<\xi_\eta$.
If there is $\nu'\in(\nu,\eta)$ such that $\gamma^{\nu'\eta}_{\xi_\eta;\xi_{\nu'}}=\xi_\eta$
then we can deduce model-goodness at $(\eta,\xi,\nu)$ from model-goodness at $(\eta,\xi,\nu')$ as in the successor case.
So suppose there is no such $\nu'$. The argument here is fairly similar to the previous subcase.
We have $\pi^{\alpha\beta}$ for $\nu\leq\alpha\leq\beta\leq\eta$, commuting
like before,
and it suffices to see 
\[ M^{\Xx_\eta}_{\xi_\eta}=\bigcup_{\alpha<\eta}\rg(\pi^{\alpha\eta}).\]
So let $x\in M^{\Xx_\eta}_{\xi_\eta}$. Let $\upsilon_\eta<_{\Xx_\eta}\xi_\eta$
and $\bar{x}$ be such that $x=i^{\Xx_\eta}_{\upsilon_\eta\xi_\eta}(\bar{x})$
and $\upsilon_\eta=\gamma^{\alpha\eta}_{\xi_\eta;\xi_\alpha}$ for some $\alpha<\eta$.
Write $\upsilon_\beta=\gamma^{\alpha\beta}_{\xi_\beta;\xi_\alpha}$
for $\alpha\leq\beta\leq\eta$,
so $\upsilon_\beta\leq_{\Xx_\beta}\xi_\beta$ and $(\upsilon_\beta,\xi_\beta]_{\Xx_\beta}$
does not drop in model or degree, and $\upsilon_\eps=\gamma^{\beta\eps}_{\xi_\eps;\upsilon_\beta}$
for $\beta\leq\eps\leq\eta$.
Write
\[ \bar{\pi}^{\beta\eps}=\pi^{\beta\eps}_{\xi_\eps;\upsilon_\beta 0\upsilon_\eps}:M^{\Xx_\beta}_{\upsilon_\beta}\to M^{\Xx_\eta}_{\upsilon_\eps}.\]
By induction (with model-goodness) we can fix $\beta\in[\alpha,\eta)$
and $\bar{\bar{x}}$ such that $\bar{x}=\bar{\pi}^{\beta\eta}(\bar{\bar{x}})$.
Then by commutativity,
\[ \pi^{\beta\eta}(i^{\Xx_\beta}_{\upsilon_\beta\xi_\beta}(\bar{\bar{x}}))=i^{\Xx_\eta}_{\upsilon_\eta\xi_\eta}(\bar{\pi}^{\beta\eta}(\bar{\bar{x}}))=x, \]
so $x\in\rg(\pi^{\beta\eta})$, which suffices.

The rest of the limit case is dealt with like in the successor case.
\end{casethree}

This completes the proof.
\end{proof}

\section{Normalization of transfinite stacks}\label{sec:normalization}

In this section we put things
together to prove Theorem 
\ref{tm:full_norm_stacks_strategy},
and also Theorem \ref{tm:stacks_iterability_2} below.

\begin{rem}
The proof will in fact give an explicit construction of a specific such strategy
$\Sigma^*$ from $\Sigma$, and we denote this $\Sigma^*$ by 
$\Sigma^{\stk}_{\minim}$
(or just $\Sigma^{\stk}$ for short,
though it seems this may be in conflict with the notation in 
\cite{iter_for_stacks}).
Given $\Ttvec,\Xx$  as above, we denote $\Xx$
by $\Xx_\Sigma(\Ttvec)$ (note that $\Xx$ is uniquely determined by $\Ttvec,\Sigma$).
 \end{rem}

\begin{tm}\label{tm:stacks_iterability_2}
 Let $\Omega>\om$ be regular. Let $m\leq\om$, $M$ be $m$-standard
 and $\Sigma$
 be an $(m,\Omega)$-strategy for $M$ with minimal inflation condensation.\footnote{If $m=\om$,
 this means that $\Sigma'$ has minimal inflation condensation,
 where $\Sigma'$ is the corresponding $(0,\Omega)$-strategy for $\J(M)$.
 See \S\ref{subsec:deg_om_deg_0}.}
 Then there is an
 optimal-$(m,{<\om},\Omega)$-strategy $\Sigma^*$ for $M$
 with $\Sigma\sub\Sigma^*$, such that for every stack
$\Ttvec=\left<\Tt_i\right>_{i<n}$ via $\Sigma^*$
with $n<\om$ and  $\lh(\Tt_i)$ a successor ${<\Omega}$
for each $i<n$,
 there is an \tu{(}$m$-maximal\tu{)} tree $\Xx$ via $\Sigma$ 
with $M^\Ttvec_\infty=M^\Xx_\infty$
 and $\deg^\Ttvec_\infty=\deg^\Xx_\infty$,
 such that $b^{\Ttvec}\inter\dropset_{\deg}^{\Ttvec}=\emptyset$
 iff $b^\Xx\inter\dropset_{\deg}^\Xx=\emptyset$,
and if $b^{\Ttvec}\inter\dropset_{\deg}^{\Ttvec}=\emptyset$
then
$i^{\Ttvec}=i^\Xx$.
\end{tm}

\begin{cor}
Let $\Omega>\om$ be regular.
Let $M$ be an $\om$-standard pure $L[\es]$-premouse with $\rho_\om^M=\om$,
and $\Sigma$ be a (hence the unique) $(\om,\Omega+1)$-strategy
for $M$.
Then $M$ has an optimal-$(\om,\Omega,\Omega+1)^*$-strategy
$\Sigma^*$
such that every $\Sigma^*$-iterate of $M$
of size ${<\Omega}$ is a $\Sigma$-iterate of $M$.
\end{cor}
\begin{proof}
 The  $(\om,\Omega+1)$-strategy for $M$ has minimal
 inflation condensation, as it is unique.
 (Recall this mean that the $(0,\Omega+1)$-strategy for $\J(M)$
 has minimal inflation condensation, as it is unique.)
\end{proof}

One can also prove the natural minimal
version of \cite[Theorem 9.6]{iter_for_stacks}
by combining proofs.

\begin{proof}[Proof of Theorems \ref{tm:full_norm_stacks_strategy}, 
\ref{tm:stacks_iterability_2}]
We will construct an appropriate stacks strategy $\Sigma^*$ for $M$, extending $\Sigma$.
Modulo what we have already established regarding minimal inflation,
the construction of the strategy is a simplification of the analogous construction
in \cite{iter_for_stacks}.
\footnote{The absorption maps $\varrho_\alpha$
and $\varsigma_\alpha$
of \cite{iter_for_stacks} reduce here to the identity, which renders certain issues in \cite{iter_for_stacks} trivial.}

We start with the successor case:
converting a stack of two normal trees into a single normal tree.
Let $\Tt$ be an
$m$-maximal tree on $M$ of successor length $<\Omega$, via $\Sigma$. Let 
$N=M^\Tt_\infty$ and $n=\deg^\Tt_\infty$.
Note that we get a unique
$(n,\Omega)$-strategy ($(n,\Omega+1)$-strategy
respectively) $\Psi=\Psi^\Sigma_\Tt$ for $N$
by demanding that whenever
 $\Uu$ is via $\Psi$ with $\lh(\Uu)<\Omega$,
there is a  tree $\Yy$ on $M$ via $\Sigma$
such that $\Tt\mininflatearrow\Yy$ and $\Uu$
is the factor tree $\Yy/\Tt$
(note also that $\Yy$ is determined uniquely
by this requirement);
and if $\Sigma$ is an $(m,\Omega+1)$-strategy
and $\Uu$ has length $\Omega+1$,
then there is likewise such a $\Yy$,
except that
now we can only demand that $\Yy\rest\Omega+1$
is via $\Sigma$
(and $\Yy$ is the $\Tt$-unravelling
of $\Yy\rest(\Omega+1)$, which has wellfounded models
as $\cof(\Omega)>\om$).

We can repeat this process finitely often,
and using Lemma \ref{lem:flattening_properties}
part \ref{item:factor_comm} for the commutativity etc, this
yields Theorem \ref{tm:stacks_iterability_2}.

We now complete the proof
of Theorem \ref{tm:full_norm_stacks_strategy}.
Assume $\Sigma$ is an 
$(m,\Omega+1)$-strategy.
We define an optimal-$(m,\Omega,\Omega+1)^*$-strategy $\Sigma^*$ for $M$.
Given $\alpha<\Omega$, at the start of round $\alpha$, neither player 
having yet lost, we 
will  have sequences
$\vec{\Tt}=\left<\Tt_\beta\right>_{\beta<\alpha}$,
$\vec{\Yy}=\left<\Yy_\beta\right>_{\beta\leq\alpha}$ such that:

\begin{enumerate}[label=S\arabic*.,ref=S\arabic*]
\item\label{item:Ttvec_is_m-max_stack} $\Ttvec$ is an optimal $m$-maximal stack 
on $M$,
\item $\Yyvec$ is a continuous terminal minimal inflation stack of degree $m$ on 
$M$, with each $\Yy_\beta$ via $\Sigma$,
\item for each $\beta<\alpha$, $\lh(\Tt_\beta)$
and $\lh(\Yy_\beta)$ are successors ${<\Omega}$,
\item for each $\beta<\alpha$,
$\Tt_\beta$ is the factor tree $\Yy_{\beta+1}/\Yy_\beta$,
\item $M^{\Ttvec\rest\alpha}_\infty$ is well-defined
and $=M^{\Yy_\alpha}_\infty$,
$\deg^{\Ttvec\rest\alpha}_\infty=\deg^{\Yy_\alpha}_\infty$,
[$b^{\Ttvec\rest\alpha}\inter\dropset_{\deg}^{\Ttvec\rest\alpha}=\emptyset$
iff $b^{\Yy_\alpha}\inter\dropset_{\deg}^{\Yy_\alpha}=\emptyset$],
and if 
$b^{\Ttvec\rest\alpha}\inter\dropset_{\deg}^{\Ttvec\rest\alpha}=\emptyset$
then
 $i^{\Ttvec\rest\alpha}_{0\infty}=i^{\Yy_\alpha}_{0\infty}$.
 \item for all $\beta<\gamma\leq\alpha$, we have
\[b^{\Ttvec\rest[\beta,\gamma)}\inter
\dropset_{\deg}^{\Ttvec\rest[\beta,\gamma)}
=\emptyset\iff\Yy_\gamma\text{ is }\Yy_\beta\text{-terminally-non-dropping},\]
and if $b^{\Ttvec\rest[\beta,\gamma)}\inter\dropset_{\deg}
^{\Ttvec\rest[\beta,\gamma)}
=\emptyset$
then 
$i^{\Ttvec\rest[\beta,\gamma)}=
(\pi_\infty)^{\Yy_\beta\mininflatearrow\Yy_\gamma}$. 
\end{enumerate}

This does not break down at successor
stages (unless $\Tt_\alpha$ is produced of length $\Omega+1$,
in which case the game is over and player II has won),
by the discussion above and again using Lemma \ref{lem:flattening_properties},
Lemma \ref{lem:cont_term_stack_event},
by which
$\Yy_\gamma$ is $\Yy_\beta$-terminally-dropping
iff there is some $\alpha\in[\beta,\gamma)$
such that $\Yy_{\alpha+1}$ is $\Yy_\alpha$-terminally-dropping,
and Lemma \ref{cor:terminal_dropping_equiv},
by which $\pi_\infty^{\beta,\gamma+1}=\pi_\infty^{\gamma,\gamma+1}\com\pi_\infty^{\beta\gamma}$
when $\Yy_{\gamma+1}$ is $\Yy_\beta$-terminally-non-dropping.

So suppose $\eta<\Omega$ is a limit
and we have produced $\left<\Tt_\alpha\right>_{\alpha<\eta}$
and $\left<\Yy_\alpha\right>_{\alpha<\eta}$.
We must produce $\Yy_\eta$ and verify the conditions.
Let $\Yy_\eta$ be the unique minimal comparison
of $\left<\Yy_\alpha\right>_{\alpha<\eta}$
via $\Sigma$. This exists and has 
length $\lambda+1<\Omega$
by Lemma \ref{lem:min_comp}. 
By Lemma \ref{lem:cont_term_stack_event},
there is $\nu<\eta$ such that 
$\Yy_\gamma$ is $\Yy_\beta$-terminally-non-dropping
for all $\beta,\gamma$ such that $\nu\leq\beta<\gamma\leq\eta$.
So by induction, for all $\beta\in[\nu,\eta)$,
$b^{\Tt_\beta}$ does not drop in model or degree.
And by Lemma \ref{lem:cont_term_stack_good},
$\left<\Yy_\alpha\right>_{\alpha\leq\eta}$ is good,
so
\[\begin{array}{rcl}
M^{\Yy_\eta}_\infty&=&\dirlim_{\nu\leq\beta\leq\gamma<\eta}\left(M^{\Yy_\beta}_\infty,M^{\Yy_\gamma}_\infty;\pi_\infty^{\beta\gamma}\right)\\
&=&\dirlim_{\nu\leq\beta\leq\gamma<\eta}\left(M^{\Ttvec\rest\beta}_\infty,M^{\Ttvec\rest\gamma}_\infty;i^{\Ttvec\rest[\beta,\gamma)}\right)\\
& =&M^{\Ttvec\rest\eta}_\infty,\end{array}\]
and $\pi_\infty^{\beta\eta}=i^{\Ttvec\rest[\beta,\eta)}$ is the direct limit map.

This completes the construction and analysis of the strategy
through $\Omega$ rounds.
Finally for the limit stage $\Omega$,
because $\Omega$ is regular
(in fact $\cof(\Omega)>\om$ suffices), note that player II wins
(but in this case we do not try to define any tree $\Yy_\Omega$).
\end{proof}

The following corollary now follows easily:
\begin{cor} Let $(M,m,\Omega,\Sigma)$
	be appropriate, $\Tt,\Xx$ on $M$ via $\Sigma$, each of successor length ${<\Omega}$,
	such that $\Xx$ is a $\Tt$-terminal minimal inflation of $\Tt$.
	Then there is a unique
	$\Uu$ such that $(\Tt,\Uu)$ is via $\Sigma^{\stk}_{\min}$
	and $\Xx=\Xx_\Sigma(\Tt,\Uu)$.
	Moreover, if $b^\Uu$ does not drop
	in model or degree
	\tu{(}so $\Xx$ is $\Tt$-terminally-non-dropping\tu{)}
	then $i^\Uu_{0\infty}=\pi_\infty^{\Tt\mininflatearrow\Xx}$ \tu{(}Definition \ref{dfn:inflation_notation}\tu{)},
	and the extenders used along $[0,\infty]_\Uu$ are just those in $\vec{F}_\infty^{\Tt\mininflatearrow\Xx}$.
	\end{cor}

	\begin{rem}
	 It is straightforward to adapt the
	  foregoing normalization techniques
	 assuming only a \emph{partial} normal iteration strategy with condensation which applies only to some restricted collection $\mathscr{C}$ of normal trees, for normalizing stacks from some collection $\mathscr{C}'$, assuming  that the normalization process for stacks in $\mathscr{C}'$ yields trees in $\mathscr{C}$.
	 For example, one might start with a mouse $M$ and some $M$-cardinal $\delta$,
	 and consider only trees and stacks thereof which are based on $M|\delta$,
	 or consider only (stacks of)  nice normal trees, etc. There is a more extended discussion of similar considerations for normal realization in \cite{iter_for_stacks}.
	\end{rem}

\section{Analysis of comparison}\label{sec:analysis_of_comparison}

Let $\Omega>\om$ be regular.
Let $M$ be $m$-standard and $\Sigma$ be an 
$(m,\Omega+1)$-strategy
for $M$. Let $\Tt_0,\Tt_1$ be trees on $M$ according to $\Sigma$,
each of successor length ${<\Omega}$. Let $N_i=M^{\Tt_i}_\infty$
and  
$n_i=\deg^{\Tt_i}_\infty$.
Let $\Sigma_i$ be the $(n_i,\Omega+1)$-strategy
for $N_i$ which is just the second round of $\Sigma^{\stk}_{\minim}$ following
$\Tt_i$. We now analyze the comparison of $(N_0,N_1)$
via $(\Sigma_0,\Sigma_1)$. Let $(\Uu_0,\Uu_1)$
be this comparison, with padding as usual (such that
if $\alpha+1<\beta+1<\lh(\Uu_0,\Uu_1)$
and $E^{\Uu_i}_\alpha\neq\emptyset\neq 
E^{\Uu_{1-i}}_\beta$ then 
$\nutilde(E^{\Uu_i}_\alpha)<\nutilde(E^{\Uu_{1-i}}_\beta)$).

Let $\Xx$ be the minimal comparison of $(\Tt_0,\Tt_1)$.
So $\Tt_i\mininflatearrow\Xx$ for $i=0,1$.
Let $C^i=C^{\Tt_i\mininflatearrow\Xx}$ for $i=0,1$.
For each $\alpha<\lh(\Xx)$ we have
$\alpha\in C^0\cup C^1$
and if $\alpha+1<\lh(\Xx)$
then $t^i(\alpha)=0$ for $i=0$ or $i=1$.

Now we claim that $\Uu_i$ is the factor tree $\Xx/\Tt_i$
for $i=0,1$. For given $\Xx\rest(\alpha+1)$
where $\alpha+1<\lh(\Xx)$,
suppose $E^\Xx_\alpha=E^{\Tt_0\mininflatearrow\Xx}_\alpha$.
Then $\alpha\in(C^-)^0$ and if $\alpha\in(C^-)^1$
then $\lh(E^{\Tt_0\mininflatearrow\Xx}_\alpha)\leq
\lh(E^{\Tt_1\mininflatearrow\Xx}_\alpha)$.
 But if $\lambda\eqdef\lh(E^{\Tt_0\mininflatearrow\Xx}_\alpha)=
\lh(E^{\Tt_1\mininflatearrow\Xx}_\alpha)$
then  $\lambda$ is a cardinal in the corresponding
models $M^{\Uu_i}_\beta$ (for the appropriate $\beta$)
and
\[ M^{\Uu_0}_\beta|\lambda=M^{\Uu_1}_\beta|\lambda=
 (\exit^\Xx_\alpha)^\passive.
\]
However, if $\lambda\eqdef\lh(E^{\Tt_0\mininflatearrow\Xx}_\alpha)<
\lh(E^{\Tt_1\mininflatearrow\Xx}_\alpha)$
then $\lambda$ is a cardinal $M^{\Uu_0}_\beta$
and
\[ M^{\Uu_0}_\beta|\lambda= (\exit^\Xx_\alpha)^\passive, \]
but
\[ M^{\Uu_1}_\beta|\lambda= \exit^\Xx_\alpha, \]
so $E^{\Uu_0}_\beta=\emptyset$
and $E^{\Uu_1}_\beta=E^\Xx_\alpha$.

These considerations easily lead to the fact that
$\Uu_i=\Xx/\Tt_i$.

It is easy to see that the same argument
works
for an arbitrary collection of iterates
(given we have enough iterability for the comparison).

We now mention a simple corollary.
Suppose $M=M_1^\#\in L[x]$,
where $x\in\RR$.
Suppose $\Tt_0,\Tt_1$ are $\om$-maximal trees on $M_1$,
 both are maximal, in the sense that they have limit length and
 $\delta(\Tt_i)$ is Woodin in $L[M(\Tt_i)]$,
and both are countable in $L[x]$.
Woodin has asked whether the pseudo-comparison
of $(L[M(\Tt_0)],L[M(\Tt_1)])$ terminates in countably many steps in $L[x]$.
It seems one might hope to use the analysis
of comparison above to answer this question affirmatively.
There is a simple case where this does work:
\begin{cor}
 Suppose $M_1^\#\in L[x]$ where $x\in\RR$
 and $\Tt_0,\Tt_1$ are as in the previous paragraph, and
 \[ \lh(\Tt_0)=\lh(\Tt_1)=\om.\]
 Then the pseudo-comparison
 of $(L[M(\Tt_0)],L[M(\Tt_1)])$ lasts exactly $\omega$ many steps.
\end{cor}
\begin{proof}
 Let $\Xx$ be the minimal comparison of $(\Tt_0,\Tt_1)$.
 It suffices to see $\Xx$ lasts only $\om$ many steps.
 Suppose not. Then note that there is $i<2$
 such that $\om\in C^i$ and $f^i(\om)=\om$.
 Say $i=0$. But
 then by the maximality of $\Tt_0$,
 $\Xx$ is maximal, which implies the pseudo-comparison
 has finished at stage $\om$, a contradiction.
\end{proof}

It follows that the collection of such trees (maximal
of length $\om$) is closed under comparison of pairs
and Neeman genericity iteration. But note that
the direct limit $M_\infty$ of all such iterates of $M_1$
is not $\sub\HOD^{L[x]}$, because
the least measurable of $M_\infty$ is ${<\om_1^{L[x]}}$.

\section{Generic absoluteness of iterability (minimal version)}\label{sec:gen_abs_it}

\subsection{Extending strategies to generic extensions}

As we work in $\ZF$, we define the \emph{$\Omega$-chain condition}
as in \cite[Definition 7.1]{iter_for_stacks}.

\begin{tm}\label{thm:strat_with_cond_extends_to_generic_ext}
Let $\Omega>\om$ be regular.
Let $\PP$ be an $\Omega$-cc forcing and $G$ be $V$-generic for $\PP$.
Let $M$ be an $\ell$-sound premouse
and $\Gamma$ be an $(\ell,\Omega+1)$-strategy for $M$ with minimal hull 
condensation.
Then:

\begin{enumerate}[label=\arabic*.,ref=\arabic*]
\item\label{item:Gamma'_exists_unique} In $V[G]$ there is a unique 
$(\ell,\Omega+1)$-strategy $\Gamma'$ such that $\Gamma\sub\Gamma'$ and 
$\Gamma'$ 
has minimal inflation condensation.
\item\label{item:Gamma'_shc} In $V[G]$, $\Gamma'$ has minimal hull condensation.
\item\label{item:Gamma,Gamma'_DJ} Suppose $M$ is countable
and let $e$ be an enumeration of $M$ in ordertype $\om$.
Then:
\begin{enumerate}[label=--]
\item $\Gamma$ has Dodd-Jensen iff $\Gamma'$ has Dodd-Jensen in $V[G]$.
\item $\Gamma$ has weak Dodd-Jensen with respect to $e$
iff $\Gamma'$ has weak Dodd-Jensen with respect to $e$ in $V[G]$.
\end{enumerate}

\item\label{item:trees_via_Sigma'_embed} For every tree $\Tt\in V[G]$ via 
$\Gamma'$, there is a $\Tt$-terminally-non-dropping minimal inflation $\Xx$ of $\Tt$
such that $\Xx\in V$ and $\Xx$ is via $\Gamma$. Moreover,
if $\lh(\Tt)<\Omega$ then we can take $\lh(\Xx)<\Omega$.
\end{enumerate}
\end{tm}

\begin{proof}[Proof sketch]
Work in $V[G]$.
Let $\Gamma'$ be the set of all pairs $(\Tt,b)$
such that $\Tt$ is an $\ell$-maximal tree on $M$
of limit length $\leq\Omega$ and $b$ is $\Tt$-cofinal and there is a limit length tree
$\Xx\in V$ and $\Xx$-cofinal branch $c\in V$ with $(\Xx,c)$ via
$\Gamma$,
and there is a minimal tree embedding
$\Pi:(\Tt,b)\hookrightarrow_{\min}(\Xx,c)$.
Defining \emph{almost minimal tree embedding} by direct analogy
with \emph{almost tree embedding}, \cite[Definition 4.26]{iter_for_stacks},
the direct analogue of \cite[Lemma 4.28]{iter_for_stacks} holds.
So in the definition of $\Gamma'$ we could 
replace the requirement that  $\Pi$ be a minimal tree embedding
with the requirement that $\Pi:(\Tt,b)\hookrightarrow_{\mathrm{alm}\min}(\Xx,c)$ be an
almost minimal tree embedding, without changing the result.
Then  almost identically to
the proof of  \cite[Theorem 7.3]{iter_for_stacks}, one can prove
that $\Gamma'$ has the right properties.
Moreover, for each such $(\Tt,b)$, one
can actually find a witnessing
$(\Xx,c)\in V$
which is a terminally-non-dropping minimal inflation of $(\Tt,b)$.
This completes the sketch.
\end{proof}

\begin{rem}
As in \cite{iter_for_stacks},
we do not know whether minimal inflation condensation (instead of minimal hull condensation) is enough to prove some version
of the preceding theorem. However,
it is enough in the following context. Suppose $W$ is a proper class inner model of $\ZF$, $M\in W$,
and $M$ is iterable in $V$,
and both $W$ and $V$ satisfy the assumptions of the theorem,
but with minimal hull condensation
weakened to minimal inflation condensation,
and the strategy $\Gamma^W$ used in $W$ is just $\Gamma^V\rest W$, where $\Gamma^V$ is that in $V$. Suppose that for each $p\in\PP$ (we have $\PP\in W$)
there is $G\in V$ such that $G$ is $(W,\PP)$-generic and $p\in G$. Then $W$
satisfies that the conclusion of the theorem holds, excluding part \ref{item:Gamma'_shc}.
To see this, we define $\Gamma'\in W[G]$
as above, except that we also demand
that $(\Xx,c)$ is a terminally-non-dropping minimal inflation of $(\Tt,b)$, and $\Pi$
is the associated minimal tree embedding.
We claim that this works. However,
the proof has to differ a little: we can't
expect \cite[Claim 6, Theorem 7.3]{iter_for_stacks} to hold.
For most purposes, \cite[Claim 7, Theorem 7.3]{iter_for_stacks} is enough. But
there a point in the proof of \cite[Claim 8, Theorem 7.3]{iter_for_stacks} where it is not enough: when verifying
that $p_0$ forces that $\Xx$ is a
(now minimal) inflation of $\dot{\Tt}$.
In that argument, a contradiction
is reached by violating \cite[Claim 6, Theorem 7.3]{iter_for_stacks} (see also Footnote 44 there). But in the present context, note that
the situation in that paragraph (under the contradictory assumption) violates our hypotheses about $W$ and $V$. That is, 
if in $W[G]$ we can get such a pair of trees,
then since this is forced, we may assume $G\in V$, so we get such a pair of trees in $V$,
contradicting that $\Sigma$ has minimal inflation condensation in $V$
(since $\Xx\in W$ is via $\Gamma^W$,
hence via $\Gamma^V$). Note here that
if $G\in V$, then $\Gamma^{W[G]}\sub\Gamma^V$.

As a corollary, let $P$ be a proper class model of $\ZF$ and $\delta$ a countable ordinal and $P=\Hull^P(\mathscr{I}\cup\delta)$ where
 $\mathscr{I}$ is a (the) class of Silver indiscernibles for $P$ (with respect to ordinals $\leq\delta$). Suppose $M$ is fully iterable via strategy $\Sigma$ with minimal inflation condensation, $M\in P$, and $P$ is closed under $\Sigma$ and $\Sigma\rest P$ is a $P$-class. Then the theorem goes through
 for all generic extensions of $P$
 (excluding part \ref{item:Gamma'_shc} again).
 This is because, by indiscernibility,
 we may assume that $\PP\in P$ and $\pow(\PP)\cap P$ is countable, and so the preceding discussion applies. (Moreover,
 again if $G\in V$ then $\Gamma^{P[G]}\sub\Gamma^V$.)
 \end{rem}

\section{Properties of $\Sigma^{\stk}_{\minim}$}\label{sec:properties}

In this section, we say that a tuple $(M,m,\Omega,\Sigma)$
is \emph{appropriate} iff $m\leq\omega$,
$M$ is $m$-standard, $\Omega>\om$ is regular,
 $\Sigma$ is either an $(m,\Omega+1)$-strategy
 or an $(m,\Omega)$-strategy for $M$, and $\Sigma$ has
minimal inflation condensation. Let then $\Sigma_{\min}^{\stk}$
denote the $(m,\Omega,\Omega+1)^*$- or $(m,{<\omega},\Omega)$-strategy for $M$
constructed in the proof of Theorem \ref{tm:full_norm_stacks_strategy}
or \ref{tm:stacks_iterability_2}
 respectively. We abbreviate and write $\Sigma^{\stk}=\Sigma_{\min}^{\stk}$ (this abbreviation
is in conflict with
the notation of \cite{iter_for_stacks}).
Say a stack 
$\vec{\Tt}=\left<\Tt_\alpha\right>_{\alpha<\lambda}$ via $\Sigma^{\stk}$ is \emph{continuable}
if $\lambda<\Omega$ or $\lambda<\omega$ respectively and
  each $\Tt_\alpha$ has length ${<\Omega}$.
Suppose $\vec{\Tt}$ is continuable and has a last model.
Let $(N,n)=(M^{\vec{\Tt}}_\infty,\deg^{\vec{\Tt}}_\infty)$.
Then $(\Sigma^{\stk})_N$ denotes the $(n,\Omega,\Omega+1)^*$-
or $(n,{<\omega},\Omega)$-strategy for $N$ which is the tail
of $\Sigma^{\stk}$ from $N$; i.e.
$(\Sigma^{\stk})_N(\vec{\Uu})=\Sigma^{\stk}(\vec{\Tt}\conc\vec{\Uu})$.
Then (i) this depends only on $N$ (not otherwise on $\vec{\Tt}$ or $n$),
so the notation makes sense. Morever, (ii) letting $\Xx=\Xx(\vec{\Tt})$
be the normalization of $\vec{\Tt}$ via $\Sigma$,
then we get the same strategy $(\Sigma^{\stk})_N$ if we replace
$\vec{\Tt}$ by $\Xx$. Note that (ii) is just by construction of $\Sigma^{\stk}$, and (i) is a immediate consequence,
since the normal tree $\Xx$ is determined uniquely by $(M,m,\Sigma,N)$. We also write $\Sigma_N$ for the (normal)
$(n,\Omega+1)$- or $(n,\Omega)$-strategy for $N$
which is the first round of $(\Sigma^{\stk})_N$
(so with $\Xx$ as above,  $\Sigma_N(\Uu)$
is defined via normalization of the stack $(\Xx,\Uu)$).

\begin{lem}\label{lem:pres_vshc}
	Let $(M,m,\Omega,\Sigma)$ be appropriate.
	
Let $\Tt$ be via $\Sigma$, of successor length $<\Omega$,
and $(N,n)=(M^\Tt_\infty,\deg^\Tt_\infty)$.
Let $\Uu_0,\Uu_1$ be $n$-maximal on $N$, 
where $\Uu_1$ is via $\Sigma_N$ and both have successor length,
and
\[ \Pi:\Uu_0\hookrightarrow_{\min}\Uu_1,\]
such that letting $\alpha+1=\lh(\Uu_0)$,
then $\gamma_\alpha^{\Pi}+1=\lh(\Uu_1)$.
Let $\Xx_1=\Xx(\Tt,\Uu_1)$, i.e.~$\Tt\inflatearrow_{\min}\Xx_1$ is the $\Tt$-terminal minimal inflation via $\Sigma$ with
 $\Uu_1=\Xx_1/\Tt$. Let $\Xx_0=\Xx(\Tt,\Uu_0)$, stopping at the least ill-defined/illfounded model,
 if we reach one.
 Then $\Xx_0$ is a true tree and there is a unique
 \[\Pi':\Xx_0\hookrightarrow_{\min}\Xx_1 \]
such that $\vec{F}^{\Pi'}_\infty=\vec{F}^\Pi_\infty$. Therefore,
if $\Sigma$ has minimal hull condensation then
so does $\Sigma_N$, $\Xx_0$ is via $\Sigma$
and $\Uu_0$ via $\Sigma_N$.\end{lem}
\begin{proof}
In order to specify $\Pi'$, we just need to specify $I_\alpha'=[\gamma_\alpha',\delta_\alpha']_{\Xx_1}$ for each $\alpha<\lh(\Xx_0)$.
Note that these are uniquely determined by the requirement that
$\vec{F}^{\Pi'}_\infty=\vec{F}$ where $\vec{F}=\vec{F}^\Pi_\infty$.
That is, for each $\alpha+1<\lh(\Xx_0)$,
$\vec{F}^{\Pi'}_{\delta_\alpha'}=\vec{F}\rest\eta_\alpha$
where $\eta_\alpha$ is largest such that $\vec{F}\rest\eta_\alpha$
is $(\exit^{\Xx_0}_\alpha,0)$-pre-good (and then it must be in fact $(\exit^{\Xx_0}_\alpha,0)$-good), and if $\alpha+1=\lh(\Xx_0)$
then $\vec{F}^{\Pi'}_{\delta_\alpha'}=\vec{F}$,
and these conditions determine $\Pi'$.
It is straightforward to verify that this works.
\end{proof}

Given a finite stack $(\Tt_0,\Tt_1,\Tt_2)$
and some form of normalization/normal realization $\Xx(\Tt,\Uu)$, it is natural to ask whether that process is associative; that is,
whether $\Xx(\Tt_0,\Xx(\Tt_1,\Tt_2))=\Xx(\Xx(\Tt_0,\Tt_1),\Tt_2)$.
Siskind first proved such a result; see \cite[***Remark 3.82]{str_comparison}.
Theorem \ref{tm:associativity} below includes a version of this in our present context.

\begin{tm}\label{tm:associativity}
	Let $(M,m,\Omega,\Sigma)$ be appropriate.
Let $\Tt$ be via $\Sigma$, of successor length $<\Omega$,
and $(N,n)=(M^\Tt_\infty,\deg^\Tt_\infty)$.
Then:
\begin{enumerate}
\item\label{item:Sigma_N_mic} $\Sigma_N$ has minimal inflation condensation,
\item\label{item:Sigma_N_vshc} if $\Sigma$ has minimal hull condensation then so does $\Sigma_N$,
 \item\label{item:strat_agmt} $(\Sigma^{\stk})_N=(\Sigma_N)^{\stk}$.
\end{enumerate}
\end{tm}
\begin{proof}
 Part \ref{item:Sigma_N_mic}: This is much like the proof of Lemma \ref{lem:pres_vshc},
 except now we have $\Uu_0\mininflatearrow\Uu_1$,
 and we get $\Xx_0\mininflatearrow\Xx_1$.
 The witnessing tree embeddings are of course
 related as in Lemma \ref{lem:pres_vshc}.

 Part \ref{item:Sigma_N_vshc}: This is by Lemma \ref{lem:pres_vshc}.
 
 Part \ref{item:strat_agmt}: This is essentially an associativity fact for the normalization considered in this paper. 
 It can
essentially be deduced as consequences of what has been developed already in the paper.\footnote{
The proof of associativity is simpler for full normalization than normal realization (which we do not consider here), because of how $\Sigma$ and $M^\Tt_\infty$ determine $\Tt$ for normal $\Tt$ via $\Sigma$.}

Let
 $\left<\Tt_{i,i+1}\right>_{i<3}$ be
 an $m$-optimal stack of length 3 on $M$,
 with $\Tt_{i,i+1}$ an $m_i$-maximal tree on $M_i$ of length ${<\Omega}$, where $(N_0,n_0)=(M,m)$ and
 $(N_{i+1},n_{i+1})=(M^{\Tt_{i,i+1}}_\infty,\deg^{\Tt_{i,i+1}}_\infty)$,
and
 such that $\Tt_{01}$ is via $\Sigma$,
 $\Tt_{12}$ is via $\Sigma_{N_1}$,
 and $\Tt_{23}$ is via $\Sigma'_{N_2}=(\Sigma_{N_1})_{N_2}$ (we use here
 that by part \ref{item:Sigma_N_mic}, $\Sigma_{N_1}$ has minimal inflation condensation, so $(\Sigma_{N_1})_{N_2}$ is well-defined). Let $\Tt_{02}=
\Xx(\Tt_{01},\Tt_{12})$, the normalization of $(\Tt_{01},\Tt_{12})$ (which is via $\Sigma$).
 So $(N_2,n_2)=(M^{\Tt_{02}}_\infty,\deg^{\Tt_{02}}_\infty)$.
 We must see that $\Tt_{23}$
 is also via $\Sigma_{N_2}$.
 Let $\eta<\lh(\Tt_{23})$ be a limit,
 and assume by induction that $\Tt_{23}\rest\eta$ is via $\Sigma_{N_2}$.
 Let $b=\Sigma_{N_2}(\Tt_{23}\rest\eta)$
 and $b'=\Sigma'_{N_2}(\Tt_{23}\rest\eta)$. We must see that $b=b'$.
 
 Let $\Uu=(\Tt_{23}\rest\eta)\conc b$, 
 $\Uu_{03}=\Xx(\Tt_{02},\Uu)$,
 which is via $\Sigma$. 
 
 Let $\Uu'=(\Tt_{23}\rest\eta)\conc b'=\Tt_{23}\rest(\eta+1)$,
$\Uu_{13}'=\Xx(\Tt_{12},\Uu')$,
which is via $\Sigma_{N_1}$,
and $\Uu_{03}'=\Xx(\Tt_{01},\Uu_{13}')$,
which is via $\Sigma$.

Let $\delta=\delta(\Tt_{23}\rest\eta)$.
Let $\lambda$ be the limit
such that $\delta=\delta(\Uu_{03}\rest\lambda)$,
and $\lambda'$ likewise for $\Uu_{03}'$.
Then
\[M^{\Uu_{03}}_\infty|\delta=
M(\Uu_{03}\rest\lambda)=
M(\Tt_{23}\rest\eta)=
M(\Uu_{03}'\rest\lambda')=
M^{\Uu_{03}'}_\infty|\delta,\]
 and 
since $\Uu_{03}\rest(\lambda+1)$
and $\Uu_{03}'\rest(\lambda'+1)$
are also both via $\Sigma$,
therefore $\lambda=\lambda'$
and $\Uu_{03}\rest(\lambda+1)=\Uu_{03}'\rest(\lambda'+1)$.

Also let $\xi$ be the limit such that
$\delta=\delta(\Tt_{13}\rest\xi)$.

Let $c=[0,\lambda)_{\Uu_{03}}=[0,\lambda')_{\Uu'_{03}}$. Then:
\begin{enumerate}[label=--]
	\item $c$ determines
$b$
via the factor tree $(\Uu_{03}\rest(\lambda+1))/\Tt_{02}$,
\item  $c$ determines $d=[0,\xi)_{\Tt_{13}}$
via the factor tree $(\Uu_{03}\rest(\lambda+1))/\Tt_{01}$,
and
\item $d$ determines $b'$
via the factor tree $(\Tt_{13}\rest(\xi+1))/\Tt_{12}$.
\end{enumerate}

We can now easily deduce $b=b'$.
Let $\left<L_\alpha\right>_{\alpha<\eta}$ be the
sequence of intervals 
$L_\alpha=[\lambda_\alpha,\zeta_\alpha]\sub\lambda$ defined using the stack $(\Tt_{02},\Tt_{23}\rest\eta)$, so the ordinals $\zeta_\alpha$
enumerate exactly those $\zeta<\lambda$
such that $\lh(E^{\Uu_{03}}_\zeta)=\lh(E^{\Tt_{23}}_\alpha)$ for some $\alpha<\eta$, and
$\lambda_\alpha=\sup_{\gamma<\alpha}(\zeta_\gamma+1)$, with $\lambda_0=0$.
These intervals partition $\lambda$,
and if $\alpha<\eta$ and $c\cap L_\alpha\neq\emptyset$
then $\alpha\in b$, and since there are cofinally many such $\alpha<\eta$,
this determines $b$. (Recall though
that there can be $\alpha\in b$
with $c\cap L_\alpha=\emptyset$.)
We similarly have $\left<L'_\beta\right>_{\beta<\xi}$,
defined with respect to $(\Tt_{01},\Tt_{13}
\rest\xi)$,
and $\left<L''_\alpha\right>_{\alpha<\eta}$,
with respect to $(\Tt_{12},\Tt_{23}\rest\eta)$.
The former sequence
partitions $\lambda$,
while the latter partitions
$\xi$, and moreover,
for each $\alpha<\eta$
we have $L_\alpha=\bigcup_{\beta\in L''_\alpha}L'_\beta$;
this is just because
every extender used in $\Tt_{23}\rest\eta$
is also used in $\Tt_{13}\rest\xi$.
But then if $c\cap L'_\beta\neq\emptyset$
then $\beta\in d$,
so $d\cap L''_\alpha\neq\emptyset$
and $\alpha\in b'$
where $\beta\in L''_\alpha$.
But since $L'_\beta\sub L_\alpha$,
we also get $\alpha\in b$,
so $\alpha\in b\cap b'$.
So $b\cap b'$ is cofinal in $\eta$,
so $b=b'$, as desired.
\end{proof}
\begin{dfn}
	Under the assumptions of the theorem below,
	given an above-$\delta$,
	$m$-maximal
	tree $\Uu$ on $M$, the \emph{minimal $i^{\vec{\Tt}}$-copy of $\Uu$} is the (putative) $m$-maximal tree $\Vv$
	on $M^{\vec{\Tt}}_\infty$ given by
	copying extenders as in minimal 
	tree embeddings (and preserving tree order); i.e. for each $\alpha<\lh(\Uu)$, we set
	\[ \exit^{\Vv}_\alpha=\Ult_0(\exit^\Uu_\alpha,E)
\] where $E$
	is the $(\delta,i^{\vec{\Tt}}(\delta))$-extender derived from $i^{\vec{\Tt}}$.
	Given an above-$i^{\vec{\Tt}}(\delta)$,
	$m$-maximal strategy $\Gamma$ for $M^{\vec{\Tt}}_\infty$, the \emph{minimal
		$i^{\vec{\Tt}}$-pullback of $\Gamma$}
	is computed by copying in this manner.
	\end{dfn}
\begin{tm}\label{tm:easy_pullback_con}
	Let $(M,m,\Sigma,\Omega)$ be appropriate.
	Then $\Sigma^{\stk}$ is weakly pullback consistent, meaning that if $\delta$ is a strong cutpoint of $M$
	and $\vec{\Tt}$ is via $\Sigma^{\stk}$
	and continutable,
	 based on $M|\delta$,
	 has a last model, and
	 $b^{\vec{\Tt}}$ does not drop in model or degree,
	 	then the restriction of $\Sigma^{\stk}$
	to above-$\delta$ trees is just the
	minimal $i^{\vec{\Tt}}$-pullback of $\Sigma_{M^{\vec{\Tt}}_\infty}$. 
\end{tm}
\begin{proof}
	Let $\Uu$ be an above-$\delta$
	normal tree on $M$ via $\Sigma$.
	Then note that
	\[ \Uu\inflatearrow_{\min}\Xx(\vec{\Tt},i^{\vec{\Tt}}``\Uu), \]
	by lifting all relevant structures from $\Uu$
	to $i^{\vec{\Tt}}``\Uu$
	by taking ultrapowers by the extender
	derived from $i^{\vec{\Tt}}$,
	at the appropriate degrees
	(that is, assuming
	that $i^{\vec{\Tt}}``\Uu$
	has wellfounded models). 
	Then $i^{\vec{\Tt}}``\Uu$
	is via $\Sigma_{M^{\vec{\Tt}}_\infty}$
	(and hence has wellfounded models).
	For otherwise by minimizing on length,
	we easily produce a counterexample
	to minimal inflation condensation.
	
	Since the normal strategies $\Sigma$
	and $(\Sigma_{M^{\vec{\Tt}}_\infty})^{\stk}=(\Sigma^{\stk})_{M^{\vec{\Tt}}_\infty}$,
	this easily extends to stacks.
\end{proof}

\begin{tm}
	Let $(M,m,\Omega,\Sigma)$
	be appropriate.
	Then $\Sigma^{\stk}$ is commuting.
		In fact, if $\vec{\Tt}=\left<\Tt_\alpha\right>_{\alpha<\lambda}$,
		$\vec{\Uu}=\left<\Uu_\alpha\right>_{\alpha<\eta}$,
		$\vec{\Vv}=\left<\Vv_\alpha\right>_{\alpha<\xi}$ are such that
		 $\vec{\Tt}\conc\vec{\Uu}$
		and $\vec{\Tt}\conc\vec{\Vv}$
		are both  stacks via $\Sigma^{\stk}$, with normal 
		rounds given by the trees
		$\Tt_\alpha$ and $\Uu_\alpha$,
		and the trees $\Tt_\alpha$ and $\Vv_\alpha$, respectively,
		both stacks having a final model, $b^{\vec{\Uu}}$
		does not drop in model or degree,
		and $M^{\vec{\Uu}}_\infty=M^{\vec{\Vv}}_\infty$,
		then  $b^{\vec{\Vv}}$ does not drop in model or degree,
		$\deg^{\vec{\Uu}}_\infty=\deg^{\vec{\Vv}}_\infty$,
		and $i^{\vec{\Uu}}=i^{\vec{\Vv}}$.
	\end{tm}
\begin{proof}
	By Theorem \ref{tm:associativity}
	we can assume $\vec{\Tt}=\emptyset$.
	But then the conclusion follows
	immediately from Theorem \ref{tm:full_norm_stacks_strategy}
	and \ref{tm:stacks_iterability_2}.
	\end{proof}

\begin{tm}\label{tm:independence_from_tree_tail}
	Let $(M,m,\Omega,\Sigma)$
	be appropriate. Let $\vec{\Tt},\vec{\Tt}'$
	be stacks via $\Sigma^{\stk}$ with
	$\Xx=\Xx(\vec{\Tt})$ and $\Xx'=\Xx(\vec{\Tt}')$ each of successor length $<\Omega$, such that  $b^{\vec{\Tt}}$
	and $b^{\vec{\Tt'}}$ do not drop in model or degree
	(so neither do $b^\Xx$  or $b^{\Xx'}$).
	Let $\eta<\OR^M$ and $\alpha\in b^\Xx$
	be least such that  either $\alpha+1=\lh(\Xx)$ or $i^\Xx_{0\alpha}(\eta)<\crit(i^\Xx_{\alpha\infty})$, and $\alpha'\in b^{\Xx'}$ likewise. Suppose  $\alpha=\alpha'$
	and $\Xx\rest(\alpha+1)=\Xx'\rest(\alpha+1)$. Let $\Gamma$ be the restriction
	of $\Sigma_{M^\Xx_\infty}$
	to trees based on $M^\Xx_\infty|i^\Xx_{0\infty}(\eta)$
	and $\Gamma'$ likewise for $\Xx'$.
	Then $\Gamma=\Gamma'$.
	\end{tm}
\begin{proof}[Proof Sketch]
	We may assume that $\Xx=\Xx\rest(\alpha+1)$, so $\Xx\ins\Xx'$. Let $\Uu$ be a limit
	length tree via both $\Gamma$ and $\Gamma'$, and let $b=\Gamma(\Uu)$
	and $b'=\Gamma'(\Uu)$; we want to se $b=b'$. Let $\Yy=\Xx(\Xx,\Uu\conc b)$
	and $\Yy'=\Xx(\Xx',\Uu\conc b')$.
	Let $\lambda$ be the limit
	such that $\delta(\Yy\rest\lambda)=\delta(\Uu)$,
	and $\lambda'$ likewise for $\Yy'$.
	Then just note that $\lambda=\lambda'$
	and $\Yy\rest\lambda=\Yy'\rest\lambda$,
	so $\Yy\rest(\lambda+1)=\Yy'\rest(\lambda+1)$, and hence $b=b'$.
	\end{proof}

\begin{tm}Let $M$ be $m$-standard and $\delta$ be a strong cutpoint of $M$. Let $\Tt_0\conc\Tt_1$ and $\Xx_0\conc\Xx_1$ be successor length $m$-maximal trees on $M$,  such that $\Tt_0,\Xx_0$ are each based on $M|\delta$,
and	$b^{\Tt_0},b^{\Xx_0}$ do not drop in model or degree,
		and $\Tt_1,\Xx_1$ are above $i^{\Tt_0}(\delta)$ and $i^{\Xx_0}(\delta)$ respectively. Suppose
	\[ \Tt_0\conc\Tt_1\inflatearrow_{\min}\Xx_0\conc\Xx_1 \]
and $\Xx_0$ is a $\Tt_0$-terminally-non-dropping minimal inflation of $\Tt_0$. Let $\pi:M^{\Tt_0}_\infty\to M^{\Xx_0}_\infty$ be the final copy map associated to the inflation, so $\pi$ is also an iteration map.
	Let $\widehat{\Tt_1}=\pi``\Tt_1$ be the minimal $\pi$-copy of $\Tt_1$ to an $m$-maximal tree on $M^{\Xx_0}_\infty$. Then
	\[ \Xx_0\conc\widehat{\Tt_1}\inflatearrow_{\min}\Xx_0\conc\Xx_1.\]	
	\end{tm}
\begin{proof}[Proof sketch]
	Let $(C,f,\ldots)$ be the objects associated to the inflation of $\Tt_0\conc\Tt_1$ and $(C',f',\ldots)$ to that of $\Xx_0\conc\widehat{\Tt_1}$.
	Then $f'\rest\lh(\Xx_0)=\id$ and if $f(\lh(\Xx_0)+\alpha)=\lh(\Tt_0)+\beta$
	then $f'(\lh(\Xx_0)+\alpha)=\lh(\Xx_0)+\beta$, and 
 $\gamma'_{\lh(\Xx_0)+\alpha;\xi}=\xi$ for $\xi<\lh(\Xx_0)$ and $\gamma'_{\lh(\Xx_0)+\alpha;\lh(\Xx_0)+\xi}=\gamma_{\lh(\Xx_0)+\alpha;\lh(\Tt_0)+\xi}$.
 It follows that the domain structures ``above $\delta$'' for the minimal tree embedding $\Pi'_{\lh(\Xx_0)+\alpha}$ are just the $\pi$-ultrapowers of the corresponding ones for $\Pi_{\lh(\Xx_0)+\alpha}$, with ultrapowers taken at the appropriate degree.
 	Further, a sequence $\vec{F}\conc\vec{G}$ of inflationary extenders used in the  inflation of $\Tt_0\conc\Tt_1$,
 	where $\vec{F}$ is equivalent to the $(\delta,\pi(\delta))$-extender derived from $\pi$, corresponds to $\vec{G}$ being used in the inflation of $\Xx_0\conc\widehat{\Tt_1}$.
	\end{proof}

\begin{rem}
	\cite[Remark 9.12]{iter_for_stacks}
	adapts to our context here directly.
	(The existence of the branch $c=c_b$  was
	not discussed there. It probably should have been,
	and one comment there confuses the question somewhat.
	Let $\theta$ be as there,
	and consider the case that $f(\lambda^\beta)<\theta$ 
	for all $\beta\in b$ (otherwise it is easy);
	in particular, $\theta$ is a limit.
	In this case, \cite{iter_for_stacks}  says ``Note that for $c$ as desired to exist,
	we must have $\theta<\lh(\Tt)$''. But in fact, we \emph{do} have $\theta<\lh(\Tt)$, because the entire construction assumes that $\Tt$ has successor length.
	This is moreover important for the existence of $c$:
	let $a=[0,\theta)_{\Tt}$. Then the pair $(a,b)$ uniquely determines a branch $c$ with the right properties.)
	
	The original version of this fact, observed by Steel,
	also holds in our present context:
 Let $(M,m,\Omega,\Sigma)$ be appropriate. Let $\Tt$ be via $\Sigma$, of successor length ${<\Omega}$.
 Let $(N,n)=(M^{\Tt}_\infty,\deg^\Tt_\infty)$.
	Let $\Uu$ be on $N$, 
	via $\Sigma_N$ (hence $n$-maximal),
	of limit length ${<\Omega}$.
	Let $\Xx=\Xx(\Tt,\Uu)$, which has limit length $\lambda$.
 Then there is a one-to-one correspondence between $\Xx$-cofinal branches $c$
	and pairs $(a,b)$ such that $b$ is a $\Uu$-cofinal branch,
	and letting $C,C^-,f$ be as above, either:
	\begin{enumerate}
		\item there is $\beta\in b$ such that $\beta\notin C$,
and $a=\emptyset$,  or
\item otherwise, and letting $\theta=\sup_{\beta\in b}f(\lambda^\beta)$, then:
\begin{enumerate}
\item there is $\beta\in b$ such that $f(\lambda^\beta)=\theta$, and $a=[0,\theta]_{\Tt}$, or
\item\label{item:interesting_case} otherwise (so $\theta$ is a limit and $\theta<\lh(\Tt)$),
and $a$ is some $\Tt\rest\theta$-cofinal branch.
 \end{enumerate}
\end{enumerate}

(Note that in clause \ref{item:interesting_case}, $a$ might or might not be $[0,\theta)_{\Tt}$.) Moreover, we can pass from $c$ to $(a,b)$ as done for minimal inflation. For given an $\Xx$-cofinal branch
$c$, the pair $(a,b)$ is determined as in that process;
and given a pair $(a,b)$ as above, it is straightforward to
construct the corresponding $c$.

Now suppose that also, $\delta<\rho_m^M$ is inaccessible in $M$,
$\Tt$ is based on $M|\delta$, $[0,\infty]_{\Tt}$ does not drop
in model or degree, and $\Uu$ is based on $N|i^{\Tt}_{0\infty}(\delta)$. Let $(a,b),c$ be a corresponding pair. Then $c$ drops in model or degree iff either $a$ does or $b$ does. Suppose there are no such drops on $a,b,c$. Say
that $a$ is \emph{$\Tt$-good}
iff $i^\Tt_{a}(\delta)=\delta(\Tt)$
and $M^\Tt_a$ is $(\delta(\Tt)+1)$-wellfounded;
likewise for $b,c$. Then $c$ is $\Xx$-good
iff $a$ is $\Tt$-good and $b$ is $\Uu$-good, and moreover,
when these things hold,
we have $i^\Xx_c\rest(\delta+1)=i^\Uu_b\com i^\Tt_a\rest(\delta+1)$. In fact, note that these things are direct consequences of the properties of minimal inflation,
when generalized to allow trees with illfounded last model.
(For this, one needs to talk about the standard
fine structural concepts, like $n$-soundness,
$\rho_n$, etc, allowing illfounded models.
But we only need it for models which arise as direct limits
of wellfounded models, along the branch of an iteration tree.
This ensures that the relevant fine structural notions
can be appropriately defined; see for example \cite[3.11--3.14]{extmax}.
\end{rem}

\bibliographystyle{plain}
\bibliography{../bibliography/bibliography}

\end{document}